\documentclass{amsart}
\usepackage{a4wide,graphicx,subfigure,hyperref,url}
\usepackage[english]{babel}

\newtheorem{theorem}{Theorem}[section]
\newtheorem{prop}[theorem]{Proposition}
\newtheorem{lemma}[theorem]{Lemma}
\newtheorem{cor}[theorem]{Corollary}
\theoremstyle{remark}
\newtheorem{remark}[theorem]{Remark}
\newtheorem{ex}[theorem]{Example}
\theoremstyle{definition}
\newtheorem{defn}[theorem]{Definition}

\newcommand\decdot{\raisebox{0.1ex}{\textbf{.}}}
\DeclareMathOperator{\inn}{int}

\begin{document}

\title[Beta-expansions, natural extensions and multiple tilings]{Beta-expansions, natural extensions and multiple tilings associated with Pisot units}

\author{Charlene Kalle}
\address{Department of Mathematics, Utrecht University, Postbus 80.000, 3508 TA Utrecht, the Netherlands}
\email{c.c.c.j.kalle@uu.nl}

\author{Wolfgang Steiner}
\address{LIAFA, CNRS, Universit\'e Paris Diderot -- Paris 7, Case 7014, 75205 Paris Cedex 13, France}
\email{steiner@liafa.jussieu.fr}

\thanks{This research was supported by the Agence Nationale de la Recherche, grant ANR--06--JCJC--0073 ``DyCoNum''.}

\date{\today}


\begin{abstract}
{}From the works of Rauzy and Thurston, we know how to construct (multiple) tilings of some Euclidean space using the conjugates of a Pisot unit $\beta$ and the greedy $\beta$-transformation.
In this paper, we consider different transformations generating expansions in base~$\beta$, including cases where the associated subshift is not sofic.
Under certain mild conditions, we show that they give multiple tilings. 
We also give a necessary and sufficient condition for the tiling property, generalizing the weak finiteness property (W) for greedy $\beta$-expansions. 
Remarkably, the symmetric $\beta$-transformation does not satisfy this condition when $\beta$ is the smallest Pisot number or the Tribonacci number.
This means that the Pisot conjecture on tilings cannot be extended to the symmetric $\beta$-transformation.

Closely related to these (multiple) tilings are natural extensions of the transformations, which have many nice properties:
they are invariant under the Lebesgue measure; under certain conditions, they provide Markov partitions of the torus; they characterize the numbers with purely periodic expansion, and they allow determining any digit in an expansion without knowing the other digits. 
\end{abstract}

\maketitle

\section{Introduction}

Tilings generated by substitutions and the $\beta$-transformation are well-studied objects from various points of view. Tilings from substitutions were first introduced by Rauzy in the seminal paper~\cite{Rauzy82}. For the $\beta$-transformation, $T_{\beta} x = \beta x \bmod{1}$, $\beta > 1$, Thurston laid the ground work in~\cite{Thurston89}. The transformation $T_{\beta}$ can be used to obtain the greedy $\beta$-expansion of every $x \in [0,1)$ by iteration. The expansions obtained in this way are expressions of the form
\begin{equation}\label{q:exp}
x = \sum_{k=1}^{\infty} \frac{b_k}{\beta^k},
\end{equation}
where the digits $b_k$ are all elements of the set $\{0,1,\ldots, \lceil \beta \rceil - 1\}$. Here, $\lceil x \rceil$ denotes the smallest integer larger than or equal to~$x$. The expansions that $T_{\beta}$ produces are greedy in the sense that, for each $n \ge 1$, $b_n$ is the largest element of the set $\{0,1,\ldots, \lceil \beta \rceil - 1\}$ such that $\sum_{k=1}^{n} b_k \beta^{-k} \le x$. 
Thurston defined tiles in $\mathbb{R}^{d-1}$ when $\beta$ is a Pisot unit of degree~$d$.
Akiyama (\cite{Akiyama99}) and Praggastis~(\cite{Praggastis99}), independently of one another, showed that these tiles form a tiling of $\mathbb{R}^{d-1}$ when $\beta$ satisfies the finiteness property (F) defined in~\cite{FrougnySolomyak92} by Frougny and Solomyak, which means that the set of numbers with finite greedy expansion is exactly $\mathbb{Z}[\beta^{-1}] \cap [0,1)$. They also showed that the origin is an inner point of the central tile in this case. 

The tiles can be constructed by using two-sided admissible sequences. These are sequences $\cdots w_{-1} w_0 w_1 w_2 \cdots$ of elements from the digit set $\{0,1,\ldots, \lceil \beta \rceil - 1\}$ such that each right-sided truncation $w_k w_{k+1} \cdots$ corresponds to an expansion generated by $T_{\beta}$. In some sense, this construction makes the non-invertible transformation $T_{\beta}$ invertible at the level of sequences. In ergodic theory, a way to replace a non-invertible transformation by an invertible one, without losing its dynamics, is by constructing a version of the natural extension. A natural extension of a non-invertible dynamical system is an invertible dynamical system that contains the original dynamics as a subsystem and that is minimal in a measure theoretical sense. Much theory about natural extensions was developed by Rohlin (\cite{Rohlin61}). He gave a canonical way to construct a natural extension and showed that the natural extension is unique up to isomorphism. Many properties of the original dynamical system can be obtained through the natural extension, for example all mixing properties of the natural extension are inherited by the original system. The tilings described above and natural extensions of $T_{\beta}$ are thus closely related concepts.

The question of whether or not Thurston's construction gives a tiling when conditions are relaxed, is equivalent to a number of questions in different fields in mathematics and computer science, like spectral theory (see Siegel \cite{Siegel04}), the theory of quasicrystals (Arnoux et al.~\cite{ArnouxBertheEiIto01}), discrete geometry (Ito and Rao \cite{ItoRao06}) and automata (\cite{Siegel04}). 
In \cite{Akiyama02}, Akiyama defined a weak finiteness property (W) and proved that it is equivalent to the tiling property. 
He also stated there that it is likely that all Pisot units satisfy this condition~(W). 
It is thus conjectured that we get a tiling of the appropriate Euclidean space for all Pisot units~$\beta$. This is a version of the Pisot conjecture, which is discussed at length in the survey paper \cite{BertheSiegel05} by Berth\'e and Siegel.
Classes of Pisot numbers~$\beta$ satisfying (W) are given in~\cite{FrougnySolomyak92,Hollander96,AkiyamaRaoSteiner04,BakerBargeKwapisz06}. 

The transformation $T_{\beta}$ is not the only transformation that can be used to generate number expansions of the form (\ref{q:exp}) dynamically. In~\cite{ErdosJooKomornik90}, Erd\H{o}s et al.\ defined the lazy algorithm that also gives expansions with digits in the set $\{0,1,\ldots, \lceil \beta \rceil - 1\}$. In~\cite{DajaniKraaikamp02}, Dajani and Kraaikamp gave a transformation, which they called the lazy transformation, that generates exactly these expansions in a dynamical way. This transformation is defined on the extended interval $\big[0,\frac{\lceil\beta\rceil-1}{\beta-1}\big]$. Pedicini (\cite{Pedicini05}) introduced an algorithm that produces number expansions of the form (\ref{q:exp}), but with digits in an arbitrary finite set of real numbers~$A$. He showed that if the difference between two consecutive elements in $A$ is not too big, then every $x$ in a certain interval has an expansion with digits in~$A$. These expansions generalize the greedy expansions with digits in $\{0,1,\ldots, \lceil \beta \rceil - 1\}$, and are thus called greedy $\beta$-expansions with arbitrary digits. In~\cite{DajaniKalle08}, a~lazy algorithm is given by Dajani and the first author, that can be used to get lazy expansions with arbitrary digits. In the same article, both a greedy and lazy transformation are defined to generate expansions with arbitrary digits dynamically. Another type of transformations that generate expansions like (\ref{q:exp}), but with digits in $\{-1,0,1\}$, is given in \cite{FrougnySteiner08} by Frougny and the second author. It is shown there, among other things, that for specific $\beta > 1$ and $\alpha > 0$, the transformation $T:\, [-\beta\alpha, \beta\alpha) \to [-\beta\alpha, \beta\alpha)$ defined by $T x = \beta x - \lfloor \frac{x}{2\alpha} + \frac{1}{2} \rfloor$ provides $\beta$-expansions of minimal weight, i.e., expansions in which the number of non-zero digits is as small as possible.
These expansions are interesting e.g.\ for applications to cryptography.

\smallskip
In this paper, we consider a class of piecewise linear transformations with constant slope, that contains all the transformations mentioned above. By putting some restrictions on $\beta$ and on the digit set, we can mimic the construction of a tiling of a Euclidean space, as it is given in~\cite{Akiyama02}. We establish some properties of the tiles we obtain by this construction and state conditions under which these tiles give a multiple tiling.

In Section~\ref{sec:admiss-expans}, we define the class of transformations that we will be considering and characterize of the set of admissible sequences for these transformations. 
For the greedy $\beta$-transformation, this characterization was first given by Parry (\cite{Parry60}) and depends only on the expansion of~1. In our case, we have to consider the orbits of all the endpoints of the transformation. In this section, we impose only very mild restrictions on $\beta$ and the digit set. In Section~\ref{s:natext}, we construct a version of the natural extension of the non-invertible transformation under consideration. 
We look in detail at the domain of the natural extension and give some examples of natural extensions. All this is done under the assumption that $\beta$ is a Pisot unit and that the digit set is contained in $\mathbb{Q}(\beta)$. 
We also show that the set of points with eventually periodic expansion is $\mathbb{Q}(\beta)$, and that the points with purely periodic expansion are characterized by the natural extension domain.
In Section~\ref{s:tilings}, we define tiles in $\mathbb{R}^{d-1}$, where $d$ is the degree of $\beta$, and show that almost every point is contained in the same finite number of tiles, i.e., the construction gives a multiple tiling. We give a necessary and sufficient condition for the multiple tiling to be a tiling, generalizing the (W) property.
The tiling property is also equivalent to the fact that the natural extension domain gives a tiling of the torus~$\mathbb{T}^d$, which in turn allows to determine any digit in the expansion of a number without knowing the previous digits.
For the examples of natural extensions defined in Section~\ref{s:natext}, we discuss whether they form tilings or multiple tilings. 
We also find quasi-periodic tilings where the underlying shift space is non-sofic and which are not self-affine in the sense of~\cite{Praggastis99} and~\cite{Solomyak97}.
In some examples, some tiles consist of a single point or of countably many points and have therefore zero Lebesgue measure.

Remarkable examples of double tilings come from the symmetric $\beta$-transformation defined by Akiyama and Scheicher (\cite{AkiyamaScheicher07}), for two Pisot units~$\beta$: the Tribonacci number and the smallest Pisot number.
This means that the Pisot conjecture cannot be extended to the symmetric $\beta$-transformation.
It is unclear why the Tribonacci number and the smallest Pisot number give double tilings while many other Pisot units give tilings, and it is possible that getting more insight into this question may lead to a proof or a disproof of the Pisot conjecture.

\section{Admissible sequences} \label{sec:admiss-expans}
Throughout the paper, we consider transformations $T:\, X \to X$ defined by $T x = \beta x - a$ for $x \in X_a$, $a \in A$, where $A$ is a finite subset of~$\mathbb{R}$, $X$~is the disjoint union of non-empty bounded sets $X_a \subset \mathbb{R}$, and $\beta > 1$.
We are interested in the digital expansions generated by~$T$, as defined in Definition~\ref{d:Texp}.
We denote by $A^{\omega}$ the set of right infinite sequences with elements in~$A$, and by $\preceq$ the lexicographical order on~$A^{\omega}$.

\begin{defn}[$T$-expansion, $T$-admissible sequence] \label{d:Texp}
Let $T$ be as in the preceding paragraph.
For $x \in X$, the sequence $b(x) = b_1(x) b_2(x) \cdots \in A^{\omega}$ satisfying $b_k(x) = a$ if $T^{k-1}(x) \in X_a$, $a \in A$, is called the \emph{$T$-expansion} of~$x$.
A sequence $u \in A^{\omega}$ is called \emph{$T$-admissible} if $u = b(x)$ for some~$x \in X$.
\end{defn}

Note that $T x = \beta x - b_1(x)$, so $x = (b_1(x) + T x)/\beta$, and inductively
\begin{equation}\label{q:xtonT}
x = \sum_{k=1}^{n} \frac{b_k(x)}{\beta^k} + \frac{T^n x}{\beta^n}
\end{equation}
for all $n \ge 1$. Since $X$ is bounded, we have $\lim_{n \to \infty} (T^n x) \beta^{-n} = 0$ and thus $x = \sum_{k=1}^{\infty} b_k(x) \beta^{-k}$.
Therefore, we define the \emph{value} of a sequence $u = u_1 u_2 \cdots \in A^{\omega}$ by $\decdot u = \sum_{k\ge1} u_k \beta^{-k}$.

A~first characterization of $T$-admissible sequences is given by the following lemma.

\begin{lemma} \label{l:Xuk}
A~sequence $u = u_1 u_2 \cdots \in A^{\omega}$ is $T$-admissible if and only if $\decdot u_k u_{k+1} \cdots \in X_{u_k}$ for all $k \ge 1$.
\end{lemma}

\begin{proof}
For each $k \ge 1$, set $x_k = \decdot u_k u_{k+1} \cdots$.
Suppose first that $u = b(x)$ for some $x \in X$. 
Then we have $x = \decdot b(x) = \decdot u = x_1$, hence $x_k = T^{k-1} x \in X_{u_k}$ for all $k \ge 1$.
Now suppose that $x_k \in X_{u_k}$ for all $k \ge 1$. 
Then $T x_k = \beta x_k -u_k = x_{k+1}$ for each $k \ge 1$, hence $u = b(x_1)$. 
\end{proof}

Theorem~\ref{t:admissible} provides a simpler characterization, when $T$ satisfies some additional conditions.

\begin{lemma} \label{l:xlessy}
Let $x, y \in X$ and assume that $\sup X_a \le \inf X_{a'}$ for all $a < a'$. Then we have
\[ 
x < y \quad\mbox{if and only if}\quad b(x) \prec b(y). 
\] 
\end{lemma}

\begin{proof}
Clearly, $b(x) = b(y)$ is equivalent to $x = y$. So we can assume that there exists some $k \ge 1$ such that $b_1(x) \cdots b_{k-1}(x) = b_1(y) \cdots b_{k-1}(y)$ and $b_k(x) \neq b_k(y)$. Then $x < y$ is equivalent to $T^{k-1} x < T^{k-1} y$ by~(\ref{q:xtonT}). Since we have $T^{k-1} x \in X_{b_k(x)}$, $T^{k-1} y \in X_{b_k(y)}$, we obtain that $T^{k-1} x < T^{k-1} y$, $b_k(x) \ne b_k(y)$, is equivalent to $b_k(x) < b_k(y)$. This proves the lemma.
\end{proof}

{}From now on, we assume that the sets $X_a$ are finite unions of left-closed, right-open intervals and that $T X = X$, i.e., that $\bigcup_{a\in A} T X_a = \bigcup_{a\in A} (\beta X_a - a) = X$.

\begin{defn}[Left-, right-continuous $\beta$-transformation] \label{d:trfm}
Let $\beta > 1$, let $A$ be a finite subset of $\mathbb{R}$, $X$~be the disjoint union of non-empty sets $X_a$, $a \in A$, where each $X_a$ is a finite union of intervals $[\ell_i,r_i) \subset \mathbb{R}$, $i \in I_a$, and $\bigcup_{a\in A} (\beta X_a - a) = X$.
Then we call the map $T:\, X \to X$ defined by $T x = \beta x - a$ for all $x \in X_a$, $a \in A$, a \emph{right-continuous $\beta$-transformation}.

The corresponding \emph{left-continuous $\beta$-transformation} $\widetilde{T}:\, \widetilde{X} \to \widetilde{X}$ is defined by $\widetilde{T} x = \beta x - a$ for $x \in \widetilde{X}_a$, where $\widetilde{X}_a = \bigcup_{i\in I_a} (\ell_i,r_i]$ for each $a \in A$ and $\widetilde{X} = \bigcup_{a \in A} \widetilde{X}_a$.
For $x \in \widetilde{X}$, the $\widetilde{T}$-expansion of $x$ is denoted by $\tilde{b}(x)$.
\end{defn}

Note that Lemma~\ref{l:xlessy} also holds for $\widetilde{X}$ and the sequences~$\tilde{b}(x)$. 

\begin{theorem} \label{t:admissible}
Let $T$ be a right-continuous $\beta$-transformation, where each $X_a$, $a \in A$, is a single interval $[\ell_a, r_a)$, and $r_a \le \ell_{a'}$ if $a < a'$.
Let $\widetilde{T}$ be the corresponding left-continuous $\beta$-transformation. 
Then a sequence $u = u_1 u_2 \cdots \in A^{\omega}$ is $T$-admissible if and only if 
\begin{equation} \label{q:lex}
b(\ell_{u_k}) \preceq u_k u_{k+1} \cdots \prec \tilde{b}(r_{u_k}) \quad\mbox{for all}\ k \ge 1.
\end{equation}
A~sequence $u = u_1 u_2 \cdots \in A^{\omega}$ is $\widetilde{T}$-admissible if and only if 
\[ 
b(\ell_{u_k}) \prec u_k u_{k+1} \cdots \preceq \tilde{b}(r_{u_k}) \quad\mbox{for all}\ k \ge 1. 
\]
\end{theorem}

\begin{proof}
We will prove only the first statement, since the proof of the second one is very similar. 

If $u = b(x)$ for some $x \in X$, then $T^{k-1} x \in X_{u_k} = [\ell_{u_k},r_{u_k})$ and $b(T^{k-1} x) = u_k u_{k+1} \cdots$ for all $k \ge 1$. By Lemma~\ref{l:xlessy}, we obtain immediately that $b(\ell_{u_k}) \preceq b(T^{k-1} x)$. We show that $b(T^{k-1} x) \prec \tilde{b}(r_{u_k})$. Since $T^{k-1} x < r_{u_k}$, there must be an index $n$ such that $b_n(T^{k-1} x) \neq \tilde{b}_n(r_{u_k})$ and $b_i(T^{k-1} x) = \tilde{b}_i(r_{u_k})$ for all $i < n$. Thus, $T^{n+k-2} x < \widetilde{T}^{n-1} r_{u_k}$, which implies that $b_n(T^{k-1} x) < \tilde{b}_n(r_{u_k})$. Hence, $b(T^{k-1} x) \prec \tilde{b}(r_{u_k})$ and (\ref{q:lex}) holds.

For the other implication, suppose that $u$ satisfies (\ref{q:lex}) and set $x_k = \decdot u_k u_{k+1} \cdots$ for all $k \ge 1$. 
By Lemma~\ref{l:Xuk}, it suffices to show that $x_k \in [\ell_{u_k}, r_{u_k})$ for all $k \ge 1$.
Since $u_k = \tilde{b}_1(r_{u_k})$, there exists some $s(k) > k$ such that $u_k \cdots u_{s(k)-1} = \tilde{b}_1(r_{u_k}) \cdots \tilde{b}_{s(k)-k}(r_{u_k})$ and $u_{s(k)} < \tilde{b}_{s(k)-k+1}(r_{u_k})$.
Then we have
\[ 
r_{u_k} - x_k = \frac{\widetilde{T}^{s(k)-k}r_{u_k}-x_{s(k)}}{\beta^{s(k)-k}} > \frac{\ell_{\tilde{b}_{s(k)-k+1}(r_{u_k})}-x_{s(k)}}{\beta^{s(k)-k}} \ge \frac{r_{u_{s(k)}}-x_{s(k)}}{\beta^{s(k)-k}} 
\]
for all $k \ge 1$. 
By iterating $s^n(k) = s(s^{n-1}(k)) > s^{n-1}(k)$, $n \ge 1$, we obtain 
\[ 
r_{u_k} - x_k > \lim_{n\to\infty} \frac{r_{u_{s^n(k)}}-x_{s^n(k)}}{\beta^{s^n(k)-k}} = 0, 
\]
where we have used that $\{x_k:\, k \ge 1\}$ is bounded and that $\lim_{n\to\infty} s^n(k) = \infty$. 
Similarly, we can show that $x_k \ge \ell_{u_k}$ for all $k \ge 1$, hence the theorem is proved.
\end{proof}


\begin{remark} \label{r:fullintervals}
With the assumptions of Theorem~\ref{t:admissible}, we have $b_1(\ell_{u_k}) = \tilde{b}_1(r_{u_k}) = u_k$ for all $k \ge 1$.
If $T \ell_{u_k} = \min X$, then the condition $b(\ell_{u_k}) \preceq u_k u_{k+1} \cdots$ follows from $b(\ell_{u_{k+1}}) \preceq u_{k+1} u_{k+2} \cdots$. 
Similarly, $\widetilde{T} r_{u_k} = \max \widetilde{X}$ implies that $u_k u_{k+1} \cdots \prec \tilde{b}(r_{u_k})$ follows from $u_{k+1} u_{k+2} \cdots \prec \tilde{b}(r_{u_{k+1}})$.
\end{remark}

\begin{remark}
In all our examples, the conditions of Theorem~\ref{t:admissible} are fulfilled and $X$ is a half-open interval, the sets $X_a$ are thus consecutive half-open intervals. 
However, the more general Definition~\ref{d:trfm} is needed at several points in this paper, e.g.\ in the proof of Proposition~\ref{p:boundary}, for the transformation $T_Y$ in Section~\ref{sec:cubic-pisot-units} and when we restrict $T$ to the support of its invariant measure.
\end{remark}

\begin{ex}\label{x:clgreedylazy}
Consider the classical greedy $\beta$-transformation, $T_{\beta} x = \beta x \bmod{1}$. This fits in the above framework if we take $A = \{0,1,\ldots, \lceil \beta \rceil - 1\}$, $X_a = \big[\frac{a}{\beta}, \frac{a+1}{\beta}\big)$ for $a \le \lceil \beta \rceil - 2$, $X_{\lceil\beta\rceil-1} = \big[\frac{\lceil\beta\rceil-1}{\beta},1\big)$. Parry gave a characterization of the $T_{\beta}$-admissible sequences in~\cite{Parry60}. It only depends on the $\widetilde{T}$-expansion of~$1$, since $T \ell_a = 0$ for every $a \in A$ and $\widetilde{T} r_a = 1$ for $a \le \lceil \beta \rceil - 2$.
The transformation $\widetilde{T}$ is sometimes called quasi-greedy $\beta$-transformation.

The lazy $\beta$-transformation with digit set $A = \{0,1,\ldots, \lceil \beta \rceil - 1\}$ is given by~$\widetilde{T}$, where $\ell_0 = 0$, $r_a = \ell_{a+1} = \frac{1}{\beta} \big(\frac{\lceil\beta\rceil-1}{\beta-1} + a\big)$ for $0 \le a \le \lceil \beta \rceil - 2$ and $r_{\lceil\beta\rceil-1} = \frac{\lceil\beta\rceil-1}{\beta-1}$. For the lazy $\beta$-transformation, the characterization of the expansions depends only on the $T$-expansion of $\frac{\lceil\beta\rceil-\beta}{\beta-1}$.
\end{ex}

\begin{ex}
Let $\beta > 1$, $A = \{a_0,a_1,\ldots,a_m\}$, with $0 = a_0 < a_1 < \cdots < a_m$ satisfying
\begin{equation} \label{q:pedicini}
\max_{0\le i<m} (a_{i+1} - a_i) \le \frac{a_m}{\beta-1}.
\end{equation}
The greedy $\beta$-transformation with digit set $A$, as defined in~\cite{DajaniKalle08}, is obtained by setting $X_{a_i} = \big[\frac{a_i}{\beta}, \frac{a_{i+1}}{\beta}\big)$ for $0 \le i < m$ and $X_m = \big[\frac{a_m}{\beta}, \frac{a_m}{\beta-1}\big)$. Condition (\ref{q:pedicini}) was given by Pedicini in~\cite{Pedicini05} and guarantees that $\bigcup_{a\in A} (\beta X_a - a) = X$. The expansions given by this transformation are exactly the expansions obtained from the recursive algorithm that Pedicini introduced in \cite{Pedicini05}. He also gave a characterization of all these expansions, similar to Theorem~\ref{t:admissible}. 
Similarly to Example~\ref{x:clgreedylazy}, only the expansions $\tilde{b}(a_{i+1}-a_i)$, $0 \le i < m$, play a role in this characterization.
\end{ex}

\begin{ex} \label{x:linmod1}
The \emph{linear mod~$1$ transformations} are maps from the interval $[0,1)$ to itself, given by $T x = \beta x + \alpha \bmod{1}$ with $\beta > 1$ and $0 \le \alpha < 1$. They are well-studied, see for example \cite{Hofbauer81,FlattoLagarias96,FlattoLagarias97a,FlattoLagarias97b}. 
We obtain these transformations by taking $A = \{-\alpha, 1-\alpha, \ldots, \linebreak \lceil \beta + \alpha \rceil - 1 - \alpha\}$, $X_{-\alpha} = \big[0,\frac{1-\alpha}{\beta}\big)$, $X_{i-\alpha} = \big[\frac{i-\alpha}{\beta}, \frac{i+1-\alpha}{\beta}\big)$ for $1 \le i \le \lceil \beta + \alpha \rceil - 2$, $X_{\lceil\beta+\alpha\rceil-1-\alpha} = \big[\frac{\lceil\beta+\alpha\rceil-1-\alpha}{\beta},1\big)$. 
If we set $\alpha = 0$, then this is the classical greedy $\beta$-transformation.
\end{ex}

\begin{ex} \label{x:minweight}
In~\cite{FrougnySteiner08}, some examples of specific transformations generating minimal weight expansions are given. These transformations are symmetric (up to the endpoints of the intervals) and depend on two parameters, $\beta > 1$ and~$\alpha$, which lies in an interval depending on~$\beta$. They fit into the above framework by taking $A = \{-1,0,1\}$ and setting $X_{-1} = [-\beta \alpha, -\alpha)$, $X_0 = [-\alpha, \alpha)$ and $X_1 = [\alpha, \beta \alpha)$. Suppose that an $x \in X$ has a finite expansion, i.e., that there is an $N \ge 1$ such that  $b_n(x) = 0$ for all $n > N$. Then the absolute sum of digits of $x$ is $\sum_{k=1}^N |b_k(x)|$ and $x \in \mathbb Z[\beta^{-1}]$. The transformations $T$ from \cite{FrougnySteiner08} generate expansions of minimal weight in the sense that if $x \in \mathbb{Z}[\beta^{-1}]\cap X$, then the absolute sum of digits of its $T$-expansion is less than or equal to that of all possible other expansions of $x$ in base $\beta$ with integer digits.
\end{ex}

\begin{ex} \label{x:sbetad}
In~\cite{AkiyamaScheicher07}, Akiyama and Scheicher define \emph{symmetric $\beta$-transformations} for $\beta > 1$ by setting $T x = \beta x - \lfloor \beta x + 1/2 \rfloor$ for $x \in [-1/2, 1/2)$. 
With our notation, this means that $A = \big\{\big\lfloor \frac{1-\beta}{2} \big\rfloor, \ldots, \big\lceil\frac{\beta-1}{2} \big\rceil\big\}$, $X_{\lfloor(1-\beta)/2\rfloor} = \big[-\frac{1}{2}, \frac{\lfloor(1-\beta)/2\rfloor}{\beta} + \frac{1}{2\beta}\big)$, $X_i = \big[\frac{i}{\beta} -\frac{1}{2\beta}, \frac{i}{\beta} + \frac{1}{2\beta}\big)$ for $\big\lfloor \frac{1-\beta}{2} \big\rfloor < i < \big\lceil\frac{\beta-1}{2} \big\rceil$, and $X_{\lceil(\beta-1)/2\rceil} = \big[\frac{\lceil(\beta-1)/2\rceil}{\beta} - \frac{1}{2\beta}, \frac{1}{2}\big)$. 
\end{ex}

\begin{figure}[ht]
\centering
\subfigure[$\beta = \pi$, lazy]{\qquad\includegraphics{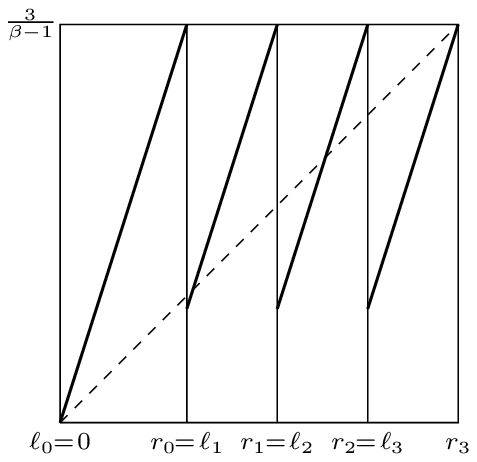}\qquad}
\quad
\subfigure[$\beta = \frac{1+\sqrt{5}}{2}$, $A = \{0, 2 \beta, 5\}$, greedy]{\qquad\includegraphics{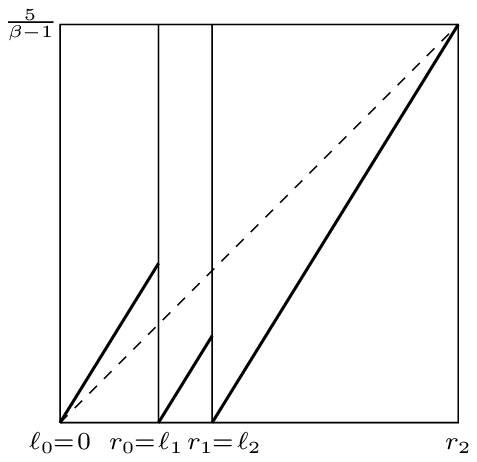}\qquad}
\\
\quad
\subfigure[$\beta = \frac{1+\sqrt{5}}2$, $\alpha = 1/2$, minimal weight]{\qquad\includegraphics{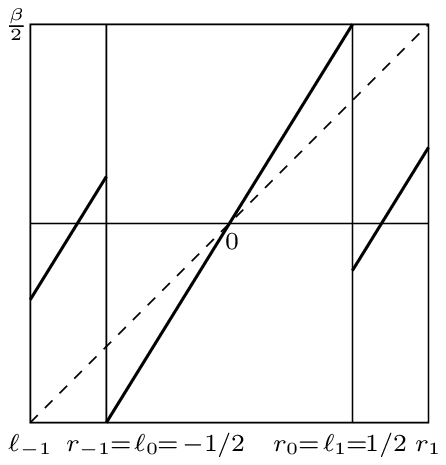}\qquad}
\qquad
\subfigure[$\beta = \frac{1+\sqrt{5}}2$, symmetric]{\quad\includegraphics{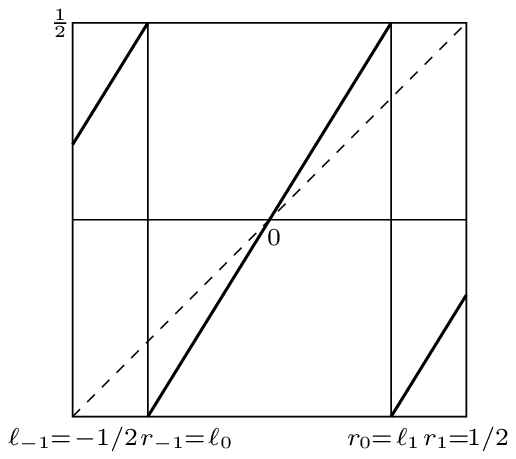}\quad}
\caption{In~(a), we see a lazy $\beta$-transformation, (b)~shows a greedy transformation with arbitrary digits, there is a minimal weight transformation in~(c) and a symmetric $\beta$-transformation in~(d).}
\label{f:genexs}
\end{figure}

\begin{remark}
{}From now on, we will consider only right-continuous $\beta$-transformations. By symmetry, all results can be easily adapted to the corresponding left-continuous $\beta$-transformations.
\end{remark}

We will use the set of $T$-expansions to construct a natural extension and a multiple tiling for~$T$. 
Therefore, we define the set
\begin{equation} \label{e:S}
\mathcal{S} = \{(u_k)_{k\in\mathbb{Z}} \in A^{\mathbb{Z}}:\, u_k u_{k+1} \cdots \ \mbox{is $T$-admissible for all}\ k \in \mathbb{Z}\},
\end{equation}
which is invariant under the shift but not closed. 
The closure of $\mathcal{S}$, which we denote by $\bar{\mathcal{S}}$, is the shift space consisting of the two-sided infinite sequences $u \in A^{\mathbb{Z}}$ such that every finite sequence $u_k u_{k+1}\cdots u_n$ is the prefix of some $T$-expansion.
Hence, $\bar{\mathcal{S}}$~is similar to the $\beta$-shift.
If the conditions of Theorem~\ref{t:admissible} are satisfied, then we have
\begin{align}
\mathcal{S} & = \{(u_k)_{k\in\mathbb{Z}} \in A^{\mathbb{Z}}:\, b(\ell_{u_k}) \preceq u_k u_{k+1} \cdots \prec \tilde{b}(r_{u_k})\ \mbox{for all}\ k \in \mathbb{Z}\}, \nonumber \\
\bar{\mathcal{S}} & = \{(u_k)_{k\in\mathbb{Z}} \in A^{\mathbb{Z}}:\, b(\ell_{u_k}) \preceq u_k u_{k+1} \cdots \preceq \tilde{b}(r_{u_k})\ \mbox{for all}\ k \in \mathbb{Z}\}, \label{e:barS}
\end{align}
and the following condition for $\bar{\mathcal{S}}$ being sofic.
Recall that a shift is sofic if its elements are the labels of the two-sided infinite walks in a finite graph, see~\cite{LindMarcus95}.

\begin{prop} \label{p:sofic}
Let $T$ be as in Theorem~\ref{t:admissible}. 
Then $\bar{\mathcal{S}}$ is a sofic shift if and only if $b(\ell_a)$ and $\tilde{b}(r_a)$ are eventually periodic for all $a \in A$.
\end{prop}

\begin{proof}
Let $\bar{\mathcal{S}}$ be sofic and $\mathcal{G}$ the corresponding finite graph.
By Theorem~3.3.2 in \cite{LindMarcus95}, we can assume, w.l.o.g., that $\mathcal{G}$ is right-resolving, i.e., that, for every vertex $v$ in $\mathcal{G}$ and any $a \in A$, there exists at most one edge labeled with $a$ starting in~$v$.
(In the language of automata, this means that the automaton is deterministic.)
For every $a \in A$, $b(\ell_a)$ is the lexicographically smallest $T$-admissible sequence starting with~$a$. 
Since $\mathcal{G}$ is finite, the lexicographically smallest label  of a right-infinite walk in $\mathcal{G}$ is eventually periodic, thus $b(\ell_a)$ is eventually periodic.
Similarly, $\tilde{b}(r_a)$ is eventually periodic as the lexicographically largest label of a right-infinite walk in~$\mathcal{G}$.
 
Now, let $b(\ell_a)$ and $\tilde{b}(r_a)$ be eventually periodic for all $a \in A$.
Then (\ref{e:barS}) implies that the collection of all follower sets in $\bar{\mathcal{S}}$ is finite, thus $\bar{\mathcal{S}}$ is sofic by Proposition~3.2.9 in \cite{LindMarcus95}. 
\end{proof}

\section{Natural extensions} \label{s:natext}

\subsection{Geometric realization of the natural extension and periodic expansions} \label{sec:geom-real-natur}

Our goal in this section is to define a measure theoretical natural extension for the class of transformations defined in the previous section, under suitable assumptions on $\beta$ and the digit set. This natural extension will allow us to define a multiple tiling of some Euclidean space. The set-up for this multiple tiling is similar to the one Thurston gave in \cite{Thurston89} for the classical greedy $\beta$-transformation.

Let $\beta > 1$ be a Pisot unit with minimal polynomial $x^d - c_1 x^{d-1} - \cdots - c_d \in \mathbb{Z}[x]$ and $\beta_2, \ldots, \beta_d$ its Galois conjugates. Thus, $|\beta_j| < 1$ and $|c_d| = 1$. Set $\beta_1 = \beta$. Let $M_{\beta}$ be the companion matrix
\[ 
M_{\beta} = \begin{pmatrix}c_1 & c_2 & \cdots & c_{d-1} & c_d \\ 1 & 0 & \cdots & 0 & 0 \\ 0 & 1 & \cdots & 0 & 0 \\ \vdots & \vdots & \ddots & \vdots & \vdots \\ 0 & 0 & \cdots & 1 & 0\end{pmatrix}, 
\]
and let $\mathbf{v}_j = \nu_j (\beta_j^{d-1}, \ldots, \beta_j,1)^t$ with $\nu_j \in \mathbb{C}$, $1 \le j \le d$, be right eigenvectors of $M_{\beta}$ such that $\mathbf{v}_1 + \cdots + \mathbf{v}_d = \mathbf{e}_1 = (1,0,\ldots,0)^t$.
(If $\beta_j \in \mathbb{R}$, then $\mathbf{v}_j \in \mathbb{R}^d$, and if $\beta_k$ is the complex conjugate of $\beta_j$, then the entries of $\mathbf{v}_k$ are the complex conjugates of the entries of $\mathbf{v}_j$.)
Let $H$ be the hyperplane of $\mathbb{R}^d$ which is spanned by the real and imaginary parts of $\mathbf{v}_2, \ldots, \mathbf{v}_d$. Then $H \simeq \mathbb{R}^{d-1}$.

Assume that $A \subset \mathbb{Q}(\beta)$, where $\mathbb{Q}(\beta)$ denotes, as usual, the smallest field containing $\mathbb{Q}$ and~$\beta$.
Let $\Gamma_j:\, \mathbb{Q}(\beta) \to \mathbb{Q}(\beta_j)$, $1 \le j \le d$, be the isomorphism defined by $\Gamma_j(\beta) = \beta_j$.
Then we define
\[ 
\psi(u) = \sum_{k\ge1} u_k \beta^{-k} \mathbf{v}_1 - \sum_{k\le0} \sum_{j=2}^d \Gamma_j(u_k) \beta_j^{-k} \mathbf{v}_j \in \mathbb{R}^d 
\]
for two-sided infinite sequences $u = (u_k)_{k\in\mathbb{Z}} \in A^{\mathbb{Z}}$.
The set $\widehat{X} = \psi(\mathcal{S})$, with $\mathcal{S}$ as defined in~(\ref{e:S}), will be our natural extension domain.

Define the maps $\Phi:\, \mathbb{Q}(\beta) \to H$ and $\Psi:\, \mathbb{Q}(\beta) \to \mathbb{Q}^d$ by
\begin{equation} \label{e:PhiPsi}
\Phi(x) = \sum_{j=2}^d \Gamma_j(x) \mathbf{v}_j, \qquad \Psi(x) = \sum_{j=1}^d \Gamma_j(x) \mathbf{v}_j = x \mathbf{v}_1 + \Phi(x). 
\end{equation}
For left-infinite sequences $w = (w_k)_{k \le 0} \in {\vphantom{A}}^{\omega}\!A$, we set
\begin{equation} \label{q:phiseq}
\varphi(w) = \sum_{k\le0} \Phi(w_k \beta^{-k}) = \sum_{k\le0} M_{\beta}^{-k} \Phi(w_k) \in H.
\end{equation}
Then we have
\[ 
\psi(\cdots u_{-1} u_0 u_1 u_2 \cdots) = (\decdot u_1 u_2 \cdots) \mathbf{v}_1 - \varphi(\cdots u_{-1}u_0). 
\]

The set $\widehat{X} = \psi(\mathcal{S})$ is the disjoint union of the sets $\widehat{X}_a = \psi(\{(u_k)_{k\in\mathbb{Z}} \in \mathcal{S}:\, u_1 = a\})$, $a \in A$. 
Therefore, we can define a transformation $\widehat{T}:\, \widehat{X} \to \widehat{X}$ by 
\[ 
\widehat{T} \mathbf{x} = M_{\beta} \mathbf{x} - \Psi(a) \quad\mbox{if}\ \mathbf{x} \in \widehat{X}_a. 
\] 
We have indeed $\widehat{T}\widehat{X}=\widehat{X}$ since, if we denote by $\sigma$ the left-shift, then $\sigma\mathcal{S} = \mathcal{S}$ and
\begin{align*}
\psi(\sigma u) & = (\decdot u_2 u_3 \cdots) \mathbf{v}_1 - \varphi(\cdots u_0 u_1) \\
& = M_\beta (\decdot u_1 u_2 \cdots) \mathbf{v}_1 - u_1 \mathbf{v}_1 - M_{\beta} \varphi(\cdots u_{-1} u_0) - \Phi(u_1) \\
& = M_\beta \psi(u) - \Psi(u_1) \\
& = \widehat{T} \psi(u) 
\end{align*}
for every $u = (u_k)_{k\in\mathbb{Z}} \in \mathcal{S}$.
Note that $\Psi(a) = a\, \mathbf{e}_1$ if $a \in \mathbb{Q}$. 

Define the projection $\pi_1:\, \mathbb{R}^d \to \mathbb{R}$ by $\pi_1(\mathbf{x}) = x$ if $\mathbf{x} = x \mathbf{v}_1 + \mathbf{y}$ for some $\mathbf{y} \in H$, and let $\pi:\, \widehat{X} \to \mathbb{R}$ be the restriction of $\pi_1$ to $\widehat{X}$.
Then we have 
\[ 
\pi(\widehat{T} \mathbf{x}) = \beta \pi(\mathbf{x}) - a = T \pi(\mathbf{x}) \quad\mbox{for all}\ \mathbf{x} \in \widehat{X}_a,
\]
thus $\pi \circ \widehat{T} = T \circ \pi$.

Next we show that $\lambda^d(\widehat{X}) > 0$, where $\lambda^d$ denotes the $d$-dimensional Lebesgue measure.
Since $\mathcal{S}$ is not closed, we consider $\widehat{Y} = \psi(\bar{\mathcal{S}})$. 

\begin{lemma} \label{l:compact}
The set $\widehat{Y}$ is compact, and $\lambda^d(\widehat{X}) = \lambda^d(\widehat{Y})$.
\end{lemma}

\begin{proof}
The set $\bar{\mathcal{S}}$ is a closed subset of the compact metric space $A^{\mathbb{Z}}$, hence $\bar{\mathcal{S}}$ is compact. Since $\psi:\, \bar{\mathcal{S}} \to \mathbb{R}^d$ is a continuous function, $\widehat{Y} = \psi(\bar{\mathcal{S}})$ is compact and thus Lebesgue measurable.

A sequence $(u_k)_{k\in\mathbb{Z}} \in \bar{\mathcal{S}}$ is not in $\mathcal{S}$ if and only if there exists some $k \in \mathbb{Z}$ such that $u_k u_{k+1} \cdots$ is the limit of $T$-admissible sequences and $u_k u_{k+1} \cdots$ is not $T$-admissible, i.e., there exists $k \in \mathbb{Z}$ such that $\decdot u_k u_{k+1} \cdots$ is on the (right) boundary of $X_{u_k}$.
Since there are only finitely many such points, the set of sequences $u_1 u_2 \cdots$ with this property is countable, which implies that $\pi_1(\widehat{Y} \setminus \widehat{X})$ is countable, thus $\lambda^d(\widehat{Y} \setminus \widehat{X}) = 0$, and $\widehat{X}$ is a Lebesgue measurable set with $\lambda^d(\widehat{X}) = \lambda^d(\widehat{Y})$.
\end{proof}

There are several ways to show that $\bigcup_{\mathbf{x}\in q^{-1}\mathbb{Z}^d} (\mathbf{x} + \widehat{Y}) = \mathbb{R}^d$ when $A \subset q^{-1} \mathbb{Z}[\beta]$, $q \in \mathbb{Z}$, and thus that $\lambda(\widehat{Y}) \ge q^{-d}$.
We use the following two theorems, which are interesting in their own right.
The first theorem generalizes a result by Ito and Rao (\cite{ItoRao05}).

\begin{theorem} \label{t:purelyperiodic}
Let $T$ be a right-continuous $\beta$-transformation as in Definition~\ref{d:trfm} with a Pisot unit~$\beta$ and $A \subset \mathbb{Q}(\beta)$. 
Then the $T$-expansion of $x \in X$ is purely periodic if and only if $x \in \mathbb{Q}(\beta)$ and $\Psi(x) \in \widehat{X}$.
\end{theorem}

In the proof of this theorem and later, $(b_1 \cdots b_n)^{\omega}$ denotes a block of digits repeated to the right, ${\vphantom{()}}^{\omega}(b_1 \cdots b_n)$ denotes a block of digits repeated to the left.
For a left-infinite sequence $w = (w_k)_{k\le0} \in {\vphantom{A}}^{\omega}\!A$ and a right-infinite sequence $v = (v_k)_{k\ge1} \in A^{\omega}$, let $w \cdot v$ denote the sequence $(u_k)_{k\in\mathbb{Z}}$ defined by $u_k = w_k$ for all $k \le 0$ and $u_k = v_k$ for all $k \ge 1$. 

\begin{proof}
Assume that $b(x) = (b_1 \cdots b_n)^{\omega}$ for some $n \ge 1$. 
Then ${\vphantom{(}}^{\omega}(b_1 \cdots b_n) \cdot (b_1 \cdots b_n)^{\omega} \in \mathcal{S}$ and 
\begin{align*}
\varphi\big({\vphantom{(}}^{\omega}(b_1 \cdots b_n)\big) & = \sum_{j=2}^d \big(\Gamma_j(b_1) \beta_j^{n-1} + \cdots + \Gamma_j(b_n)\big) (1 + \beta_j^n + \beta_j^{2n} + \cdots) \mathbf{v}_j \\
& = \sum_{j=2}^d \Gamma_j\bigg(\frac{b_1\beta^{n-1}+\cdots+b_n}{1-\beta^n}\bigg) \mathbf{v}_j = -\sum_{j=2}^d \Gamma_j(x) \mathbf{v}_j = -\Phi(x),
\end{align*}
thus
\[
\Psi(x) = x \mathbf{v}_1 + \Phi(x) = \psi\big({\vphantom{(}}^{\omega}(b_1 \cdots b_n) \cdot (b_1 \cdots b_n)^{\omega}\big) \in \psi(\mathcal{S}) = \widehat{X}.
\]

Now, take an $x \in \mathbb{Q}(\beta)$ with $\Psi(x) \in \widehat{X}$, and set $\mathbf{x}_0 = \Psi(x)$.
Since $\widehat{T}$ is surjective, for each $k \ge 0$ there exists an $\mathbf{x}_{k+1} \in \widehat{X}$ with $\widehat{T} \mathbf{x}_{k+1} = \mathbf{x}_k$. Let $q \in \mathbb{Z}$ be such that $\Psi(x)$ and $\Psi(a)$ are in $q^{-1} \mathbb{Z}^d$ for all $a \in A$.
Since $|\!\det M_\beta| = |c_d| = 1$, $M_{\beta}^{-1}$ is an integer matrix, and we obtain $\mathbf{x}_k \in q^{-1} \mathbb{Z}^d$ for all $k \ge 0$.
The set $\widehat{X}$ is bounded, hence we must have $\mathbf{x}_{k+n} = \mathbf{x}_k$ for some $k \ge 0$, $n \ge 1$.
This yields $\widehat{T}^n \mathbf{x}_{k+n} = \mathbf{x}_{k+n}$, which implies $\widehat{T}^n \mathbf{x}_0 = \mathbf{x}_0$ because $\mathbf{x}_0 = \widehat{T}^{k+n} \mathbf{x}_{k+n}$. 
For every $k \ge 1$, $b_k(x)$~is determined by $T^{k-1} x = T^{k-1} \pi(\mathbf{x}_0) = \pi(\widehat{T}^{k-1} \mathbf{x}_0)$, hence $b(x)$ is purely periodic. 
\end{proof}

Note that this theorem gives a nice characterization of rational numbers with purely periodic $T$-expansions, since we have $\Gamma_j(x) = x$ for $x \in \mathbb{Q}$ and thus $\Psi(x) = x \mathbf{e}_1 = (x,0,\ldots,0)^t$.

The following theorem was proved by Frank and Robinson (\cite{FrankRobinson08}) for a slightly smaller class of transformations, and generalizes the result by Bertrand (\cite{Bertrand77}) and Schmidt (\cite{Schmidt80}) for the classical $\beta$-transformation.
Note that we do not need here that $|\!\det M_\beta| = 1$ since we are only looking at the forward orbit of $x$ under~$T$.

\begin{theorem} \label{t:frankrobinson}
Let $T$ be a right-continuous $\beta$-transformation as in Definition~\ref{d:trfm} with a Pisot number~$\beta$ and $A \subset \mathbb{Q}(\beta)$. Then the $T$-expansion of $x \in X$ is eventually periodic if and only if $x \in \mathbb{Q}(\beta)$.
\end{theorem}

\begin{proof}
If $b(x) = b_1 \cdots b_m (b_{m+1} \cdots b_{m+n})^{\omega}$ is eventually periodic, then
\[ 
x = \frac{b_1}{\beta} + \cdots + \frac{b_m}{\beta^m} + \frac{1}{\beta^n-1} \bigg(\frac{b_{m+1}}{\beta^{m-n+1}} + \frac{b_{m+2}}{\beta^{m-n+2}} + \cdots + \frac{b_{m+n}}{\beta^m}\bigg), 
\]
which is clearly in $\mathbb{Q}(\beta)$.

For the other implication, let $x \in \mathbb{Q}(\beta) \cap X$.
Extend the transformation $\widehat{T}$ to $X \mathbf{v}_1 + H$ by setting $\widehat{T} \mathbf{x} = M_{\beta} \mathbf{x} - \Psi(a)$ if $\pi_1(\mathbf{x}) \in X_a$.
Let $q \in \mathbb{Z}$ be such that $\Psi(x)$ and $\Psi(a)$ are in $q^{-1} \mathbb{Z}^d$ for all $a \in A$. 
Then we have $\widehat{T}^k \Psi(x) \in q^{-1} \mathbb{Z}^d$ for each $k \ge 0$. 
Furthermore, since $M_{\beta}$ is contracting on $H$ and $\widehat{T}^k \Psi(x) \in X \mathbf{v}_1 + H$, the set $\{\widehat{T}^k \Psi(x):\, k \ge 0\}$ is bounded, hence finite, thus $(\widehat{T}^k \Psi(x))_{k\ge0}$ is eventually periodic. 
Since $b_k(x)$ is determined by $T^{k-1} x = T^{k-1} \pi_1(\Psi(x)) = \pi_1(\widehat{T}^{k-1} \Psi(x))$, $b(x)$~is eventually periodic.
\end{proof}

We assume now again that $\beta$ is a Pisot unit.

\begin{lemma} \label{l:zd}
The map $\Psi:\, \mathbb{Q}(\beta) \to \mathbb{Q}^d$ is bijective, $\Psi^{-1}(\mathbf{x}) = \pi_1(\mathbf{x})$ for all $\mathbf{x} \in \mathbb{Q}^d$, and $\Psi(\mathbb{Z}[\beta]) = \mathbb{Z}^d$.
\end{lemma}

\begin{proof}
It can be easily seen from the structure of $M_{\beta}$ that every $\mathbf{x} \in \mathbb{Q}^d$ can be written in a unique way as $\mathbf{x} = \sum_{k=0}^{d-1} z_k M_{\beta}^k \mathbf{e}_1$ with $z_k \in \mathbb{Q}$.
Since $M_\beta^k \mathbf{e}_1 = M_\beta^k (\mathbf{v}_1 + \cdots + \mathbf{v}_d) = \beta_1^k \mathbf{v}_1 + \cdots + \beta_d^k \mathbf{v}_d$, we obtain $\mathbf{x} = \Psi(\sum_{k=0}^{d-1} z_k \beta^k)$.
Every $x \in \mathbb{Z}[\beta]$ can be written in a unique way as $x = \sum_{k=0}^{d-1} z_k \beta^k$, thus $\Psi$ is bijective.
Since $\Psi(x) = x \mathbf{v}_1 + \Phi(x)$ with $\Phi(x) \in H$, we obtain that $\Psi^{-1}(\mathbf{x}) = \pi_1(\mathbf{x})$.
If $\mathbf{x} \in \mathbb{Z}^d$ and $x \in \mathbb{Z}[\beta]$, respectively, then we have $z_k \in \mathbb{Z}$ in the above decomposition. 
\end{proof}

\begin{lemma} \label{l:positivemeasure}
If $A \subset q^{-1} \mathbb{Z}[\beta]$, $q \in \mathbb{Z}$, then $\bigcup_{\mathbf{x}\in q^{-1}\mathbb{Z}^d} (\mathbf{x} + \widehat{Y}) = \mathbb{R}^d$, thus $\lambda^d(\widehat{Y}) \ge q^{-d}$.
\end{lemma}

\begin{proof}
Since $\widehat{Y}$ is compact and $q^{-1} \mathbb{Z}^d$ is a lattice, it suffices to show that $\mathbb{Q}^d \subset \bigcup_{\mathbf{x}\in q^{-1}\mathbb{Z}^d} (\mathbf{x} + \widehat{Y})$.

Take a $\mathbf{y} \in \mathbb{Q}^d$.
Since $\pi_1(q^{-1} \mathbb{Z}^d) = q^{-1} \mathbb{Z}[\beta]$ is dense in~$\mathbb{R}$, there exists a $\mathbf{z} \in \mathbb{Q}^d$ with $\mathbf{z} \equiv \mathbf{y} \pmod{q^{-1} \mathbb{Z}^d}$ and $\pi_1(\mathbf{z}) \in X$.
Set $z_0 = \pi_1(\mathbf{z})$.
Since $T$ is surjective, there exist $z_k \in X$, $k \ge 0$, with $T z_{k+1} = z_k$ for all $k \ge 0$. 
Extend $\widehat{T}$ to $X \mathbf{v}_1 + H$, as in the proof of Theorem~\ref{t:frankrobinson}. 
By the bijectivity of~$\Psi$, we have $\Psi \circ T = \widehat{T} \circ \Psi$ on $X \cap \mathbb{Q}(\beta)$, thus $\widehat{T} \Psi(z_{k+1}) = \Psi(z_k)$.

Let $r$ be a multiple of $q$ such that $\mathbf{y} \in r^{-1} \mathbb{Z}^d$.
Since $|\!\det M_\beta| = 1$, we have $\Psi(z_k) \in r^{-1} \mathbb{Z}^d$ for every $k \ge 0$, hence $\Psi(z_{k+n}) \equiv \Psi(z_k) \pmod{q^{-1} \mathbb{Z}^d}$ for some $k \ge 0$, $n \ge 1$.
The assumption $A \subset q^{-1} \mathbb{Z}[\beta]$ implies that $\Psi(A) \subset q^{-1} \mathbb{Z}^d$, thus $\widehat{T} \mathbf{x} \equiv M_\beta \mathbf{x} \pmod{q^{-1} \mathbb{Z}^d}$, which yields $\widehat{T}^n \mathbf{z} \equiv \mathbf{z} \pmod{q^{-1} \mathbb{Z}^d}$.
By Theorem~\ref{t:frankrobinson}, $b(z_0)$ is eventually periodic.
Together with Theorem~\ref{t:purelyperiodic}, this gives $\widehat{T}^{k n} \mathbf{z} = \Psi(T^{k n} z_0) \in \widehat{X}$ for some $k \ge 0$.
Since $\widehat{T}^{k n} \mathbf{z} \in \widehat{X}$ and $\widehat{T}^{k n} \mathbf{z} \equiv \mathbf{z} \pmod{q^{-1} \mathbb{Z}^d}$, we have $\mathbf{z} \in \bigcup_{\mathbf{x}\in q^{-1}\mathbb{Z}^d} (\mathbf{x} + \widehat{X})$, and the same clearly holds for~$\mathbf{y}$. 
Since $\widehat{X} \subset \widehat{Y}$, the lemma is proved.
\end{proof}

Let $\mathcal{B}$ be the Lebesgue $\sigma$-algebra on $X$ and $\widehat{\mathcal{B}}$ the Lebesgue $\sigma$-algebra on $\widehat{X}$. We want to prove that the system $(\widehat{X}, \widehat{\mathcal{B}}, \lambda^d, \widehat{T})$ is a version of the natural extension of the system $(X,\mathcal{B},\mu,T)$, where the measure $\mu$ on $(X,\mathcal{B})$ is defined by $\mu = \lambda^d \circ \pi^{-1}$. In order to do this, we need to show that there are sets $\widehat{N} \in \widehat{\mathcal{B}}$ and $M \in \mathcal{B}$, such that all the following hold. 
\begin{itemize}
\itemsep5pt
\item[(ne1)] 
$\lambda^d(\widehat{N}) = \mu(M) =0$, $\widehat{T} (\widehat{X} \setminus \widehat{N}) \subseteq \widehat{X} \setminus \widehat{N}$ and $T(X \setminus M) \subseteq X \setminus M$.
\item[(ne2)] 
The projection map $\pi:\, \widehat{X} \setminus \widehat{N} \to X \setminus M$ is measurable, measure preserving and surjective.
\item[(ne3)] 
$\pi (\widehat{T} \mathbf{x}) = T \pi (\mathbf{x})$ for all $\mathbf{x} \in \widehat{X} \setminus \widehat{N}$.
\item[(ne4)] 
The transformation $\widehat{T}:\, \widehat{X} \setminus \widehat{N} \to \widehat{X} \setminus \widehat{N}$ is invertible.
\item[(ne5)] 
$\bigvee_{k=0}^{\infty} \widehat{T}^k \pi^{-1} \mathcal{B} = \widehat{\mathcal{B}}$, where $\bigvee_{k=0}^{\infty} \widehat{T}^k \pi^{-1} \mathcal{B}$ is the smallest $\sigma$-algebra containing the $\sigma$-algebras $\widehat{T}^k \pi^{-1} \mathcal{B}$ for all $k \ge 0$.
\end{itemize}
(A~map that satisfies (ne1)--(ne3) is called a \emph{factor map}.)

\begin{lemma}\label{l:disjoint}
For all $a, a' \in A$ with $a \ne a'$, we have $\lambda^d(\widehat{T} \widehat{X}_a \cap \widehat{T} \widehat{X}_{a'}) = 0$. 
\end{lemma}

\begin{proof}
Since $|\!\det M_{\beta}| = |c_d| = 1$ and $\widehat{T} \widehat{X} = \widehat{X}$, we have
\[ 
\sum_{a\in A} \lambda^d(\widehat{T} \widehat{X}_a) = \sum_{a\in A} \lambda^d(M_{\beta} \widehat{X}_a) = \sum_{a\in A} \lambda^d(\widehat{X}_a) = \lambda^d(\widehat{X}) = \lambda^d(\widehat{T} \widehat{X}) = \lambda^d\Big(\bigcup_{a\in A} \widehat{T} \widehat{X}_a\Big), 
\]
which proves the lemma. 
\end{proof}

Let
\[ 
\widehat{N} = \bigcup_{n\in\mathbb{Z}} \widehat{T}^n \bigg(\bigcup_{a,a'\in A,\,a\neq a'} \widehat{T} \widehat{X}_a \cap \widehat{T} \widehat{X}_{a'}\bigg). 
\]
Then by Lemma~\ref{l:disjoint}, $\lambda^d(\widehat{N}) = 0$. Note that $\widehat{T}$ is a bijection on $\widehat{X} \setminus \widehat{N}$. 
Hence, $\widehat{T}$ is an a.e.\ invertible, measure preserving transformation on $(\widehat{X}, \widehat{\mathcal{B}}, \lambda^d)$, which proves (ne4). 
The measure $\mu = \lambda^d \circ \pi^{-1}$, defined on $(X,\mathcal{B})$, satisfies $\mu(X) > 0$ by Lemmas~\ref{l:compact} and~\ref{l:positivemeasure}, and has its support contained in~$X$. 
Hence, $\mu$~is an invariant measure for~$T$, that is absolutely continuous with respect to the Lebesgue measure. 
The projection map $\pi: \widehat{X} \to X$ is measurable and measure preserving, and $T \circ \pi = \pi \circ \widehat{T}$. 
Set $M = \{x \in X:\, \pi^{-1} \{x\} \subseteq N\}$. 
Then $T(X \setminus M) \subseteq X \setminus M$ and $\mu(M) = (\lambda^d \circ \pi^{-1})(M) \le \lambda^d(\widehat{N}) = 0$. 
Since $\pi$ is surjective from $\widehat{X} \setminus \widehat{N}$ to $X \setminus M$, $\pi$ is a factor map from $(\widehat{X}, \widehat{\mathcal{B}}, \lambda^d, \widehat{T})$ to $(X, \mathcal{B}, \mu, T)$. 
This gives (ne1)--(ne3). 
In the next theorem, we prove (ne5).

\begin{theorem} \label{t:natex}
Let $T$ be a right-continuous $\beta$-transformation as in Definition~\ref{d:trfm} with a Pisot unit~$\beta$ and $A \subset \mathbb{Q}(\beta)$.
Then the dynamical system $(\widehat{X}, \widehat{\mathcal{B}}, \lambda^d, \widehat{T})$ is a natural extension of the dynamical system $(X,\mathcal{B},\mu,T)$.
\end{theorem}

\begin{proof}
We have already shown (ne1)--(ne4). The only thing that remains in order to get the theorem is that
\[ 
\bigvee_{k \ge 0} \widehat{T}^k \pi^{-1} (\mathcal{B}) = \widehat{\mathcal{B}}.
\]
By the definition of $\mathcal{S}$, it is clear that $\bigvee_{k\ge0} \widehat{T}^k \pi^{-1}(\mathcal{B}) \subseteq \widehat{\mathcal{B}}$.
To show the other inclusion, take $\mathbf{x}, \mathbf{x}' \in \widehat{X}$, $\mathbf{x} \ne \mathbf{x}'$. Suppose first that $\pi(\mathbf{x}) \neq \pi(\mathbf{x}')$. Then there are two disjoint intervals $B, B' \subset X$ with $\pi(\mathbf{x}) \in B$ and $\pi(\mathbf{x}') \in B'$, thus $\mathbf{x} \in \pi^{-1} (B)$ and $\mathbf{x}' \in \pi^{-1} (B')$.
Now, suppose that $\pi(\mathbf{x}) = \pi(\mathbf{x}') = x$. 
There exist sequences $w, w' \in {\vphantom{A}}^{\omega}\!A$ with $w \cdot b(x), w'\cdot b(x) \in \mathcal{S}$ such that $\mathbf{x} = x \mathbf{v}_1 - \varphi(w)$, $\mathbf{x}' = x \mathbf{v}_1 - \varphi(w')$. Since $\mathbf{x} \ne \mathbf{x}'$, we have $w \ne w'$. Let $n \ge 1$ be the first index such that $w_{-n+1} \neq w'_{-n+1}$, and set
\[
x_n = \sum_{k=1}^n \frac{w_{-n+k}}{\beta^k} + \frac{x}{\beta^n}, \qquad x'_n = \sum_{k=1}^n \frac{w'_{-n+k}}{\beta^k} + \frac{x}{\beta^n}.
\]
Then $x_n \neq x'_n$, so there exist two disjoint intervals $B, B' \subset X$, such that $x_n \in B$ and $x'_n \in B'$. Moreover, $\mathbf{x} \in \widehat{T}^n \pi^{-1} (B)$ and $\mathbf{x}' \in \widehat{T}^n \pi^{-1} (B')$. By the invertibility of $\widehat{T}$, the sets $\widehat{T}^n \pi^{-1} (B)$ and $\widehat{T}^n \pi^{-1} (B')$ are disjoint a.e., hence, for almost all points $\mathbf{x}, \mathbf{x}' \in \widehat{X}$, we can find two disjoint elements of $\widehat{T}^n \pi^{-1} (\mathcal{B})$ such that one point is contained in one element and the other element contains the other point. This shows that $\bigvee_{k\ge0} \widehat{T}^k \pi^{-1} (\mathcal{B}) = \widehat{\mathcal{B}}$ and thus that $(\widehat{X}, \widehat{\mathcal{B}}, \lambda^d, \widehat{T})$ is a natural extension of $(X,\mathcal{B},\mu,T)$.
\end{proof}

\subsection{Shape of the natural extension domain} \label{s:shape}

We can write 
\begin{equation} \label{e:Dx}
\widehat{X} = \bigcup_{x\in X} (x \mathbf{v}_1 - \mathcal{D}_x) \quad\mbox{with}\quad \mathcal{D}_x = \big\{\varphi(w):\, w \cdot b(x) \in \mathcal{S}\big\},
\end{equation}
where $\varphi$ is as in (\ref{q:phiseq}) and $\mathcal{S}$ as in (\ref{e:S}). For the multiple tiling we will construct later on, the prototiles will be the sets $\mathcal{D}_x$ for $x \in \mathbb{Z}[\beta] \cap X$.
In this section, we show some properties of these sets.

\begin{lemma} \label{l:tilecompact}
Every set $\mathcal{D}_x$, $x\in X$, is compact.
\end{lemma}

\begin{proof}
Let $x \in X$ and consider the subset $\mathcal{W} = \{w \in {\vphantom{A}}^{\omega}\!A:\, w \cdot b(x) \in \mathcal{S}\}$ of the compact space~${\vphantom{A}}^{\omega}\!A$. We want to show that $\mathcal{W}$ is closed and, hence, compact. Therefore, take some converging sequence $(w^{(n)})_{n \ge 0} \subseteq \mathcal{W}$ and let $\lim_{n\to\infty} w^{(n)} = w$. For every $k \ge 0$, we can find some $n_k \ge 0$ such that $w^{(n_k)}_{-k} \cdots w^{(n_k)}_0 = w_{-k} \cdots w_0$. This implies that $w_{-k} \cdots w_0b(x)$ is $T$-admissible for every $k \ge 0$, thus $w \cdot b(x) \in \mathcal{S}$, and $\mathcal{W}$ is closed. 
Since $\mathcal{D}_x$ is the image of the compact set $\mathcal{W}$ under the continuous map~$\varphi$, it is compact as well.
\end{proof}

To distinguish different sets $\mathcal{D}_x$, we introduce the set
\begin{equation}\label{q:v}
\mathcal{V} = \big(X \setminus \widetilde{X}\big) \cup \bigcup_{x\in X\cap\widetilde{X}} \bigcup_{1\le k<m_x} \big\{T^k x,\, \widetilde{T}^k x\big\} \cap X,
\end{equation}
where $m_x$ is the minimal positive integer such that
\[ 
\widetilde{T}^{m_x} x = T^{m_x} x, 
\]
with $m_x = \infty$ if $\widetilde{T}^k x \neq T^k x$ for all $k \ge 1$.
Note that $m_x > 1$ only if $x$ is a point of discontinuity of~$T$ (and~$\widetilde{T}$), and that the set of these points is finite.
Furthermore, $X \setminus \widetilde{X}$ is the (finite) set of left boundary points of~$X$. 
Therefore, $\mathcal{V}$ is a finite set if and only if, for every $x \in X \cap \widetilde{X}$, $m_x < \infty$ or $x \in \mathbb{Q}(\beta)$.
(Recall that $x \in \mathbb{Q}(\beta)$ is equivalent with the fact that $b(x)$ and $\tilde{b}(x)$ are eventually periodic by Theorem~\ref{t:frankrobinson}.)
We define furthermore, for every $x \in X$,
\[
J_x = \big\{y \in X:\, y \ge x,\, (x, y\,] \cap \mathcal{V} = \emptyset\big\},
\]
i.e., $J_x = [x,z)$, where $z$ is the smallest value in $\mathcal{V}$ or on the boundary of $X$ with $z > x$.
We will prove the following proposition.

\begin{prop} \label{p:S}
If $x \in X$ and $y \in J_x$, then $\mathcal{D}_x = \mathcal{D}_y$.
If $\mathcal{V}$ is a finite set, then
\begin{equation} \label{e:shape}
\widehat{X} = \bigcup_{x\in\mathcal{V}} (J_x \mathbf{v}_1 - \mathcal{D}_x).
\end{equation}
\end{prop}

The main ingredient of the proof of Proposition~\ref{p:S} is the following simple lemma.
We extend the definition of $\varphi$ to finite sequences $v_1 \cdots v_n \in A^n$, $n \ge 0$, by
\[ 
\varphi(v_1 \cdots v_n) = \sum_{k=1}^n \Phi\big(v_k \beta^{n-k}\big). 
\] 

\begin{lemma} \label{l:replace}
If $x \in X \cap \widetilde{X}$ and $\widetilde{T}^k x = T^k x$, $k \ge 1$, then $\varphi(b_1(x) \cdots b_k(x)) = \varphi(\tilde{b}_1(x) \cdots \tilde{b}_k(x))$.
\end{lemma}

\begin{proof}
We have
\[ 
\sum_{i=1}^k \tilde{b}_i(x) \beta^{k-i} = \beta^k x - \widetilde{T}^k x = \beta^k x - T^k x = \sum_{i=1}^k b_i(x) \beta^{k-i}. 
\]
By applying $\Phi$ to this equation, the lemma is proved.
\end{proof}

\begin{lemma}\label{l:finitew}
Let $x \in X$, $y \in J_x$, and $v_1 \cdots v_n b(x)$ be a $T$-admissible sequence.
Then there exists a $T$-admissible sequence $v_1' \cdots v_n' b(y)$ with
\begin{equation}\label{e:replace}
\varphi(v_1' \cdots v_n') = \varphi(v_1 \cdots v_n) + \mathcal{O}(\rho^n),
\end{equation}
where $\rho = \max_{2\le j\le d} |\beta_j| < 1$ and the constant implied by the $\mathcal{O}$-symbol depends only on~$T$.
\end{lemma}

\begin{proof}
If $v_1 \cdots v_n b(y)$ is $T$-admissible, then the lemma clearly holds with $v_1' \cdots v_n' = v_1 \cdots v_n$.

Otherwise, let $z \in \widetilde{X}$ be maximal such that $v_1 \cdots v_n \tilde{b}(z)$ is $\widetilde{T}$-admissible.
We have $x < z \le y$ because $z = \sup\{z' \in X:\, v_1 \cdots v_n b(z')\ \mbox{is $T$-admissible}\}$, thus $z \in J_x \subseteq X$ and $z \not\in \mathcal{V}$.
Moreover, $z_k = \decdot v_{k+1} \cdots v_n \tilde{b}(z) \not\in X_{v_{k+1}}$ for some $0 \le k < n$.
Let $k$ be minimal with this property. 

Suppose first that $z_k \in X$.
Then we show
\begin{equation} \label{e:vk}
\widetilde{T}^{n-k} z_k = T^{n-k} z_k.
\end{equation}
Let $i < n-k$ be maximal with $\widetilde{T}^i z_k = T^i z_k$. 
Then (\ref{e:vk}) holds or $\widetilde{T}^{i+1} z_k \ne T^{i+1} z_k$.
In the latter case, we must have $\widetilde{T}^i z_k \in X$.
Since $\widetilde{T}^{n-k-i} \widetilde{T}^i z_k = z \in X \setminus \mathcal{V}$, we obtain that $m_{\widetilde{T}^i z_k} \le n-k-i$.
By the maximality of~$i$, we get $m_{\widetilde{T}^i z_k} = n-k-i$, thus (\ref{e:vk}) holds in this case as well.
By Lemma~\ref{l:replace}, we have $\varphi(b_1(z_k) \cdots b_{n-k}(z_k)) = \varphi(v_{k+1} \cdots v_n)$ and thus
\[
\varphi\big(v_1 \cdots v_k b_1(z_k) \cdots b_{n-k}(z_k)\big) = \varphi(v_1 \cdots v_n).
\]
Since $k$ is chosen minimally, we have $\decdot v_i \cdots v_k b(z_k) \in X_{v_i}$ for all $1 \le i \le k$, hence $v_1 \cdots v_k b(z_k)$ is $T$-admissible by Lemma~\ref{l:Xuk}. 
If $v_1 \cdots v_k b_1(z_k) \cdots b_{n-k}(z_k) b(y)$ is $T$-admissible as well, then Lemma~\ref{l:finitew} holds with 
$v_1' \cdots v_n' = v_1 \cdots v_k b_1(z_k) \cdots b_{n-k}(z_k)$.

Now, suppose that $z_k \not\in X$, i.e., that $z_k$ is a right boundary point of~$X$.
Since $\widetilde{T}^{-1} \{z_k\}$ consists only of right boundary points of sets $X_a$, the minimality of $k$ implies that $k = 0$.
We show that there exists some $z' \in X \cap \widetilde{X}$ and $h \ge 1$ such that $\widetilde{T}^h z' = z_0$.
If such $z'$ and $h$, then every set $\widetilde{T}^{-h} \{z_0\}$, $h \ge 1$, would consist only of right boundary points of~$X$.
Since there are only finitely many of those points, $\tilde{b}(z_0)$ would be purely periodic and $\widetilde{T}^n z_0 \not\in X$, contradicting that $\widetilde{T}^n z_0 = z \in X$.
Therefore, we have $\widetilde{T}^h z' = z_0$ for some $z' \in X \cap \widetilde{X}$ and $h \ge 1$.
As in the preceding paragraph, we obtain $\widetilde{T}^{h+n} z' = T^{h+n} z'$, which yields $\varphi(b_1(z') \cdots b_{h+n}(z')) = \varphi(\tilde{b}_1(z') \cdots \tilde{b}_h(z') v_1 \cdots v_n)$ and 
\[
\varphi\big(b_{h+1}(z') \cdots b_{h+n}(z')\big) = \varphi(v_1 \cdots v_n) + \mathcal{O}(\rho^n).
\]
If $b_{h+1}(z') \cdots b_{h+n}(z') b(y)$ is $T$-admissible, then the lemma is proved.

If $v_1 \cdots v_k b_1(z_k) \cdots b_{n-k}(z_k)$ and $b_{h+1}(z') \cdots b_{h+n}(z')$, respectively, is not the desired sequence $v_1' \cdots v_n'$, then we iterate with this sequence as new $v_1 \cdots v_n$ and $z$ as new~$x$.
Since $z > x$ and $A^n$ is a finite set, the algorithm terminates. 
The number of instances of $k = 0$ is bounded by the number of right boundary points of~$X$, thus the error term only depends on $T$ and is $\mathcal{O}(\rho^n)$.
\end{proof}

Lemma~\ref{l:finitew} holds in the other direction too.

\begin{lemma} \label{r:finitew}
Let $x \in X$, $y \in J_x$, and $v_1 \cdots v_n b(y)$ be a $T$-admissible sequence.
Then there exists a $T$-admissible sequence $v_1' \cdots v_n' b(x)$ with $\varphi(v_1' \cdots v_n') = \varphi(v_1 \cdots v_n) + \mathcal{O}(\rho^n)$.
\end{lemma}

\begin{proof}
The proof is similar to the proof of Lemma~\ref{l:finitew}.
If $v_1 \cdots v_n b(x)$ is not $T$-admissible, then let $z \in X$ be minimal such that $v_1 \cdots v_n b(z)$ is $T$-admissible.
We have $x < z \le y$, thus $z \in X \cap \widetilde{X}$ and $z \not\in \mathcal{V}$.
Let $k \ge 0$ be minimal such that $z_k = \decdot v_{k+1} \cdots v_n b(z) \not\in \widetilde{X}_{v_{k+1}}$.

If $z_k \in \widetilde{X}$, then we get $\widetilde{T}^{n-k} z_k = T^{n-k} z_k$ and $\varphi(v_1 \cdots v_k \tilde{b}_1(z_k) \cdots \tilde{b}_{n-k}(z_k)) = \varphi(v_1 \cdots v_n)$, as above.
Furthermore, for every $z' < z$ which is sufficiently close to~$z$, $v_1 \cdots v_k \tilde{b}_1(z_k) \cdots \tilde{b}_{n-k}(z_k) b(z')$ is $T$-admissible.
If $z_k \not\in \widetilde{X}$, then $k = 0$ and there exists some $z' \in X \cap \widetilde{X}$, $h \ge 1$, such that $T^h z' = z_0$.
We obtain $\widetilde{T}^{h+n} z' = T^{h+n} z'$ and $\varphi(\tilde{b}_{h+1}(z') \cdots \tilde{b}_{h+n}(z')) = \varphi(v_1 \cdots v_n) + \mathcal{O}(\rho^n)$.

If the sequence $v_1 \cdots v_k \tilde{b}_1(z_k) \cdots \tilde{b}_{n-k}(\ell_{v_k}) b(x)$ and $\tilde{b}_{h+1}(z') \cdots \tilde{b}_{h+n}(z') b(x)$, respectively, is not $T$-admissible, then we iterate with this sequence as new $v_1 \cdots v_n$ and $z$ as new~$y$.
Note that the new sequence $v_1 \cdots v_n b(y)$ is not $T$-admissible, but $v_1 \cdots v_n b(z')$ is $T$-admissible for every $z'$ in a non-empty interval $[z,y)$.
As above, the algorithm terminates.
\end{proof}

\begin{proof}[Proof of Proposition~\ref{p:S}]
Let $x, y \in X$ satisfy the conditions of the proposition, and take $w = (w_k)_{k\le0}$ with $w \cdot b(x) \in \mathcal{S}$.
For $n \ge 1$, let $w_{-n+1}^{(n)} \cdots w_0^{(n)} = v_1' \cdots v_n'$ be the sequence given by Lemma~\ref{l:finitew} for $v_1 \cdots v_n = w_{-n+1} \cdots w_0$.
By the surjectivity of~$T$, we can extend this sequence to a sequence $w^{(n)} = (w_k^{(n)})_{k\le0}$ with $w^{(n)} \cdot b(y) \in \mathcal{S}$.
Then we have $\varphi(w^{(n)}) = \varphi(w) + \mathcal{O}(\rho^n)$, hence $\lim_{n\to\infty} \varphi(w^{(n)}) = \varphi(w)$. 
Since $\varphi(w^{(n)}) \in \mathcal{D}_y$ and $\mathcal{D}_y$ is compact, we obtain that $\varphi(w) \in \mathcal{D}_y$, hence $\mathcal{D}_x \subseteq \mathcal{D}_y$.
By Lemma~\ref{r:finitew}, we also obtain that $\mathcal{D}_y \subseteq \mathcal{D}_x$.
Therefore, $\mathcal{D}_x = \mathcal{D}_y$ for all $y \in J_x$. 
If $\mathcal{V}$ is finite, then $X = \bigcup_{x\in\mathcal{V}} J_x$, and (\ref{e:shape}) follows from (\ref{e:Dx}).
\end{proof}

The sets $\mathcal{D}_x$ can be subdivided according to the following lemma.

\begin{lemma} \label{l:subdivision}
For every $x \in X$, we have
\begin{equation} \label{e:decomposition}
\mathcal{D}_x = \bigcup_{y\in T^{-1}\{x\}} \Big(M_{\beta} \mathcal{D}_y + \Phi\big(b_1(y)\big)\Big).
\end{equation}
If $\mathcal{V}$ is finite, then this union is disjoint up to sets of measure zero (with respect to $\lambda^{d-1}$). 
\end{lemma}

\begin{proof}
Let $x \in X$. Then for every $a \in A$ for which $a b(x)$ is $T$-admissible, there is a unique $y \in T^{-1}\{x\}$ with $b(y) = a b(x)$. Moreover, for each $w \in {\vphantom{A}}^{\omega}\!A$, we have $w a \cdot b(x) \in \mathcal{S}$ if and only if $w \cdot a b(x) \in \mathcal{S}$. Since $\varphi(w a) = M_{\beta} \varphi(w) + \Phi(a)$, we obtain~(\ref{e:decomposition}).

Assume now that $\mathcal{V}$ is finite.
Let $y, y' \in T^{-1}\{x\}$, $J = J_y \cap X_{b_1(y)}$ and $J' = J_{y'} \cap X_{b_1(y')}$, hence $J - \mathcal{D}_y \subseteq \widehat{X}_{b_1(y)}$ and $J' - \mathcal{D}_{y'} \subseteq \widehat{X}_{b_1(y')}$, by Proposition~\ref{p:S}. 
If $y \ne y'$, then we have $b_1(y) \ne b_1(y')$, thus $\lambda^d(\widehat{T}(J - \mathcal{D}_y) \cap \widehat{T}(J' - \mathcal{D}_{y'})) = 0$ by Lemma~\ref{l:disjoint}. 
We have $T J \cap T J' = [x,z)$ for some $z > x$, hence 
\[
\widehat{T}(J - \mathcal{D}_y) \cap \widehat{T}(J' - \mathcal{D}_{y'}) = (T J \cap T J') \mathbf{v}_1 - \Big(M_{\beta} \mathcal{D}_y + \Phi\big(b_1(y)\big)\Big) \cap \Big(M_{\beta} \mathcal{D}_{y'} + \Phi\big(b_1(y')\big)\Big)
\]
yields that $\lambda^{d-1}((M_{\beta} \mathcal{D}_y + \Phi(b_1(y))) \cap (M_{\beta} \mathcal{D}_{y'} + \Phi(b_1(y')))) = 0$.
\end{proof}

Using Lemma~\ref{l:subdivision}, we will show in Proposition~\ref{p:boundary} that the boundary of every $\mathcal{D}_x$, $x \in X$, has zero measure if $\mathcal{V}$ is finite.
Furthermore, (\ref{e:decomposition}) provides a graph-directed iterated function system (GIFS) in the sense of \cite{MauldinWilliams88,Falconer97} for the sets $\mathcal{D}_x$, $x \in \mathcal{V}$, if $\mathcal{V}$ is finite.
More precisely, there exists a labeled directed graph with set of vertices $\mathcal{V}$ and set of edges $\mathcal{E}$ such that 
\begin{equation} \label{e:GIFS}
\mathcal{D}_x = \bigcup_{(x,x',a) \in \mathcal{E}} \big(M_\beta \mathcal{D}_{x'} + \Phi(a)\big)\quad \mbox{for all}\ x \in \mathcal{V},
\end{equation}
where $(x,x',a)$ is  in $\mathcal{E}$ for $x, x' \in \mathcal{V}$, $a \in A$, if and only  if $(x + a)/\beta \in J_{x'}$.
Note that  the multiplication by $M_\beta$ is a contracting map on~$H$.
Every GIFS has a unique solution with non-empty compact sets, see \cite{MauldinWilliams88,Falconer97}.
The sets $\mathcal{D}_x$, $x \in \mathcal{V}$, form this solution since they are compact by Lemma~\ref{l:compact} and non-empty by the surjectivity of~$T$.

Now, consider the measure $\mu$, defined by $\mu(E) = (\lambda^d \circ \pi^{-1})(E)$ for all measurable sets~$E$. 
If $\mathcal{V}$ is finite, then there exists some constant $c > 0$ such that
\[
\mu(E) = \big(\lambda^d \circ \pi^{-1}\big) \bigg(\bigcup_{x\in\mathcal{V}} J_x \cap E\bigg) = \sum_{x\in\mathcal{V}} c\, \lambda(J_x \cap E)\, \lambda^{d-1}(\mathcal{D}_x) = c \int_E \sum_{x\in\mathcal{V}} \lambda^{d-1}(\mathcal{D}_x)\, 1_{J_x} d\lambda\,.
\]
Hence, the support of $\mu$ is the union of the intervals $J_x$, $x \in \mathcal{V}$, with $\lambda^{d-1}(\mathcal{D}_x) > 0$.
On the support, $\mu$~is absolutely continuous with respect to the Lebesgue measure. 
If the transformation $T$ has a unique invariant measure that is absolutely continuous with respect to the Lebesgue measure, then $\mu$ must be this measure. This is the case for the classical greedy and lazy $\beta$-transformations, and for the greedy and lazy $\beta$-transformations with arbitrary digits as well as for the symmetric $\beta$-transformations from Example~\ref{x:sbetad}.

\subsection{Examples of natural extensions} \label{sec:exampl-natur-extens}

We will discuss some examples. For each example, there is a figure containing the graph of the transformation and the natural extension domain. In the graph of the transformation, dotted lines indicate the orbits of points of importance.

\begin{ex}\label{x:gm}
Let $\beta$ be the golden ratio, i.e., the positive solution of the equation $x^2 - x - 1 = 0$. The other solution of this equation is $\beta_2 = -1/\beta$. Then
\[ 
M_{\beta} = \begin{pmatrix}1 & 1 \\ 1 & 0\end{pmatrix},\quad \mathbf{v}_1 = \frac{1}{\beta^2+1} \begin{pmatrix}\beta^2 \\ \beta\end{pmatrix},\quad \mathbf{v}_2 = \frac{1}{\beta^2+1} \begin{pmatrix}1 \\ -\beta\end{pmatrix}. 
\]
The greedy $\beta$-transformation is given by $A = \{0,1\}$, $X_0 = [0,1/\beta)$, $X_1 = [1/\beta,1)$. 
We have $\widetilde{T} x = T x$ for all $x \in (0,1) \setminus \{1/\beta\}$, $T^k (1/\beta) = 0$ for all $k \ge 1$, $\widetilde{T} (1/\beta) = 1 \not\in X$, $\widetilde{T}(1) = 1/\beta$, thus $\mathcal{V} = \{0, 1/\beta\}$.
The transformation and its natural extension are depicted in Figure~\ref{f:xhat}.

\begin{figure}[ht]
\centering
\includegraphics{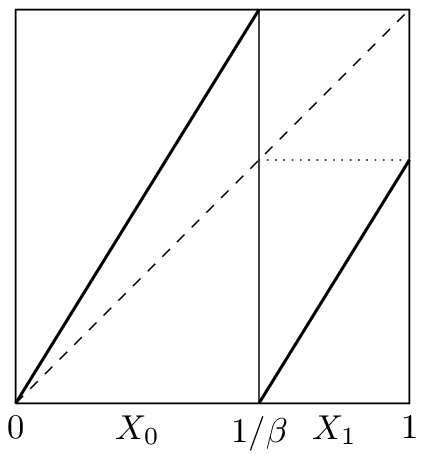}
\quad
\includegraphics{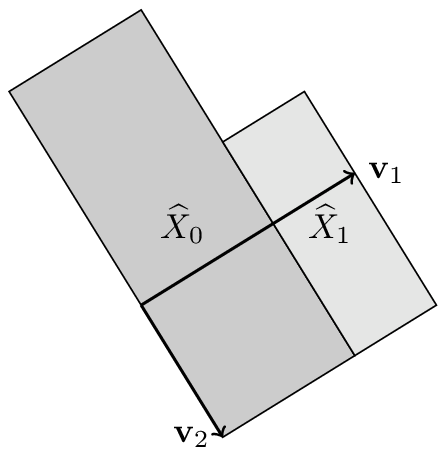}
\includegraphics{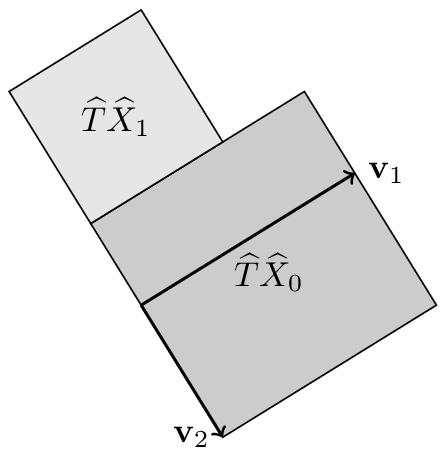}
\caption{The (greedy) $\beta$-transformation, $\beta = (1+\sqrt{5})/2$, and its natural extension.}
\label{f:xhat}
\end{figure}

In general, for the classical greedy $\beta$-transformation, we have $\mathcal{V} = \{0\} \cup \{\widetilde{T}^k(1):\, k \ge 1\} \setminus \{1\}$.
\end{ex}

\begin{ex}\label{x:gm2}
Let $\beta$ be again the golden ratio, but $A = \{-1,0,1\}$, 
\[ 
X_{-1} = \bigg[-\frac{\beta^2+\beta^{-3}}{\beta^2+1}, -\frac{\beta+\beta^{-4}}{\beta^2+1}\bigg),\ X_0 = \bigg[-\frac{\beta+\beta^{-4}}{\beta^2+1}, \frac{\beta+\beta^{-4}}{\beta^2+1}\bigg),\ X_1 = \bigg[\frac{\beta+\beta^{-4}}{\beta^2+1}, \frac{\beta^2+\beta^{-3}}{\beta^2+1}\bigg).
\]
This is an example of a minimal weight transformation with $\alpha = \frac{\beta+\beta^{-4}}{\beta^2+1}$, see Example~\ref{x:minweight}.
The points of discontinuity have the expansions
\begin{gather*}
\tilde{b}(-\alpha) = \bar{1}0010(000\bar{1})^{\omega},\ b(-\alpha) = 0\bar{1}(00\bar{1}0)^{\omega},\
\tilde{b}(\alpha) = 01(0010)^{\omega},\ b(\alpha) = 100\bar{1}0(0001)^{\omega},
\end{gather*}
where we write $\bar{1}$ instead of~$-1$, thus $m_{-\alpha} = m_\alpha = 5$.
Furthermore, $X \setminus \widetilde{X} = \{-\beta \alpha\} = \{T (-\alpha)\}$ and $\widetilde{T} \alpha = \beta \alpha \not\in X$, thus $\mathcal{V}$ consists of 15 of the 16 points $\widetilde{T}^k (-\alpha)$, $T^k (-\alpha)$, $\widetilde{T}^k \alpha$, $T^k \alpha$, $1 \le k < 5$,
\begin{align*}
\mathcal{V} & = \big\{\decdot \bar{1}(00\bar{1}0)^{\omega}, \decdot (\bar{1}000)^{\omega}, \decdot \bar{1}0(0001)^{\omega}, \decdot (0\bar{1}00)^{\omega}, \decdot 0\bar{1}0(0001)^{\omega}, \decdot(00\bar{1}0)^{\omega}, \decdot00\bar{1}0(0001)^{\omega}, \decdot0(000\bar{1})^{\omega}, \\
& \qquad \decdot 0(0001)^{\omega}, \decdot 0010(000\bar{1})^{\omega}, \decdot (0010)^{\omega}, \decdot 010(000\bar{1})^{\omega}, \decdot (0100)^{\omega}, \decdot10(000\bar{1})^{\omega}, \decdot(1000)^{\omega}\big\}.
\end{align*}
In Figure~\ref{f:xhat2}, we can see the transformation and the orbits of $-\alpha, \alpha$ on the left, the natural extension domain $\widehat{X}$ and its decomposition into $\widehat{X}_a$, $a \in A$, as well as the sets $J_x - \mathcal{D}_x$, $x \in \mathcal{V}$, in the middle, and the decomposition of $\widehat{T} \widehat{X}$ on the right.
Since $\widetilde{T}^5 \alpha = T^5 \alpha = \decdot(0001)^{\omega} \not\in \mathcal{V}$, there is no dotted line for this point in the natural extension domain. The point lies between the first and second dotted line, counted from the origin in the direction of $\mathbf v_1$ and is the point above which the boundary between $\widehat T \widehat{X}_1$ and $\widehat T \widehat {X}_0$ makes the step. That this point is not in $\mathcal V$ implies that $\mathcal{D}_x$ does not change here, but we can see that the shape of the sets $\widehat{T} \widehat{X}_1$ and $\widehat{T} \widehat{X}_0$ changes. This means that the decomposition of $\mathcal{D}_x$ according to (\ref{e:decomposition}) changes here.
The same happens at $\widetilde{T}^5 (-\alpha) = T^5 (-\alpha) = \decdot(000\bar{1})^{\omega}$ with $\widehat{T} \widehat{X}_{-1}$ and $\widehat{T} \widehat{X}_0$.
Note also that $-\alpha$ and $\alpha$ are not in~$\mathcal{V}$. 

\begin{figure}[ht]
\centering
\includegraphics{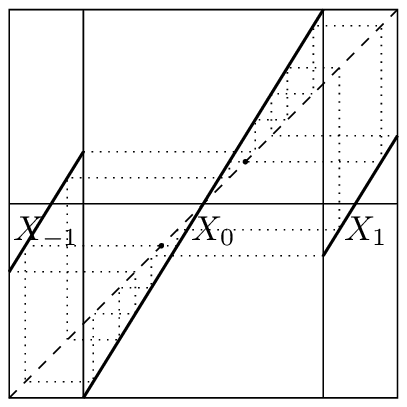}
\quad
\includegraphics{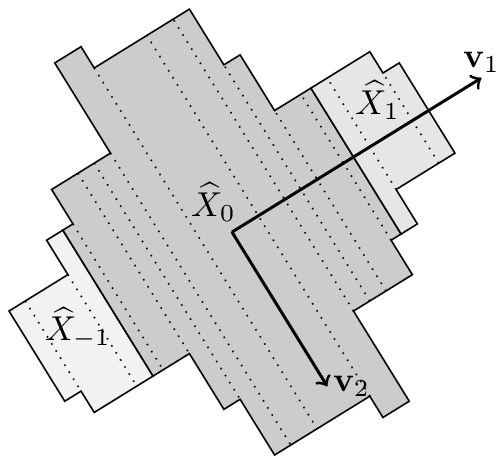}
\includegraphics{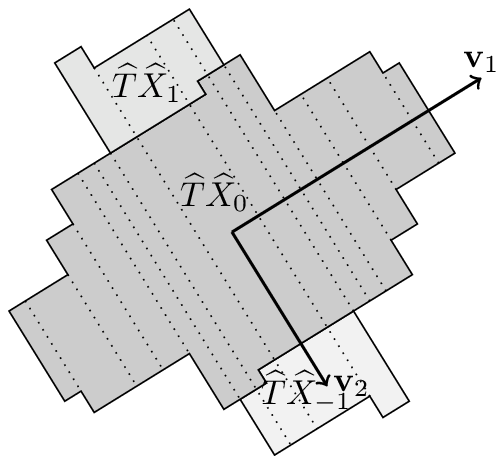}
\caption{The transformation from Example~\ref{x:gm2} and its natural extension.}
\label{f:xhat2}
\end{figure}
\end{ex}

\begin{ex} \label{x:gm3}
Let $T$ be the symmetric $\beta$-transformation for the golden ratio~$\beta$, i.e., $A = \{-1,0,1\}$, $X_{-1} = \big[-\frac{1}{2}, -\frac{1}{2\beta}\big)$, $X_0 = \big[-\frac{1}{2\beta}, \frac{1}{2\beta}\big)$ $X_1 = \big[\frac{1}{2\beta}, \frac{1}{2}\big)$, see Example~\ref{x:sbetad}.
We have 
\[
\textstyle{T^4 \big(\frac{-1}{2\beta}\big) = T^3 \big(\frac{-1}{2}\big) = T^2 \big(\frac{1}{2\beta^2}\big) = T \big(\frac{1}{2\beta}\big) = \frac{-1}{2},\quad \widetilde{T}^4 \big(\frac{1}{2\beta}\big) = \widetilde{T}^3 \big(\frac{1}{2}\big) = \widetilde{T}^2 \big(\frac{-1}{2\beta^2}\big) = \widetilde{T} \big(\frac{-1}{2\beta}\big) = \frac{1}{2},}
\] 
thus $\mathcal{V} = \big\{-\frac{1}{2}, -\frac{1}{2\beta}, -\frac{1}{2\beta^2}, \frac{1}{2\beta^2}, \frac{1}{2\beta}\big\}$.
The transformation and its natural extension are depicted in Figure~\ref{f:xhat3}. Note that $\mathbf{0}$ is a repelling fixed point of the transformation. Here, this implies that, for all $x \in [-\frac{1}{2\beta^2}, \frac{1}{2\beta^2})$, the only sequence $w \in {\vphantom{A}}^{\omega}\!A$ such that $w \cdot b(x)$ is $T$-admissible is $\cdots 00$. Hence, $\mathcal{D}_{x} = \{\mathbf{0}\}$ for all $x \in [-\frac{1}{2\beta^2}, \frac{1}{2\beta^2})$. 

When we construct a multiple tiling for $T$, we want to disregard these sets and we achieve this by restricting $T$ to the support of its invariant measure $\mu$, which is the set $[-\frac{1}{2}, -\frac{1}{2\beta^2}) \cup [\frac{1}{2\beta^2}, \frac{1}{2})$. If we restrict $T$ to this set, $T$ is a right-continuous $\beta$-transformation of which the domain $X$ is not a half-open interval. The set $X_0$ is split into two parts as well. 

For other values of~$\beta$, we refer to Section~\ref{sec:symm-beta-transf}.
\end{ex}

\begin{figure}[ht]
\centering
\includegraphics{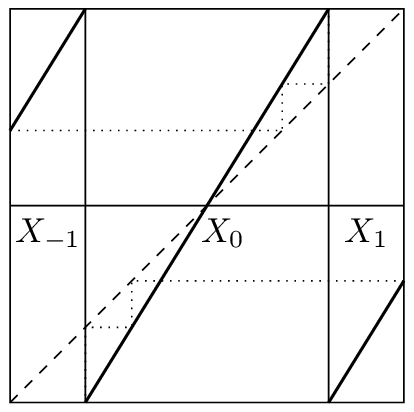}
\quad
\includegraphics{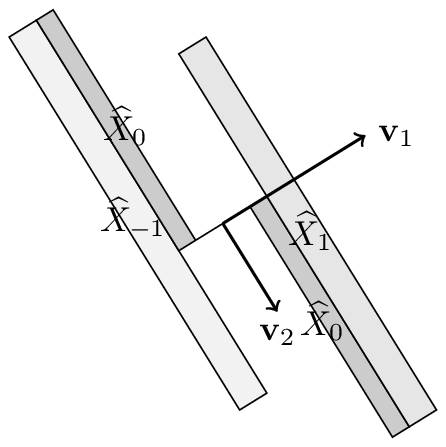}
\includegraphics{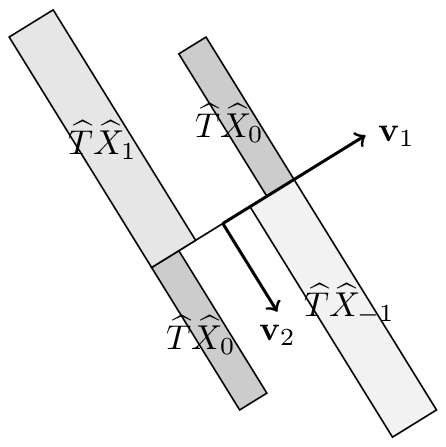}
\caption{The transformation from Example~\ref{x:gm3} and its natural extension.}
\label{f:xhat3}
\end{figure}

\begin{ex} \label{x:gm4}
We will see in Section~\ref{s:tilings} that $\lambda^2(\widehat{X}) = 1$ in Examples~\ref{x:gm}--\ref{x:gm3}.
For a transformation with $\lambda^2(\widehat{X}) > 1$, let $\beta$ be again the golden ratio, now $A = \{-1,1\}$, $X_{-1} = [-1,0)$, $X_1 = [0,1)$. 
Then $\widetilde{T}^3(0) = \widetilde{T}^2(1) = \widetilde{T}(1/\beta) = 0$, $T^3(0) = T^2(-1) = T(-1/\beta) = 0$, thus
$\mathcal{V} = \{-1, -1/\beta, 1/\beta\}$, see Figure~\ref{f:xhat4}.

\begin{figure}[ht]
\centering
\includegraphics{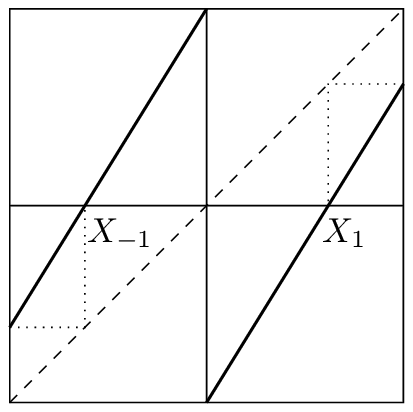}
\quad
\includegraphics{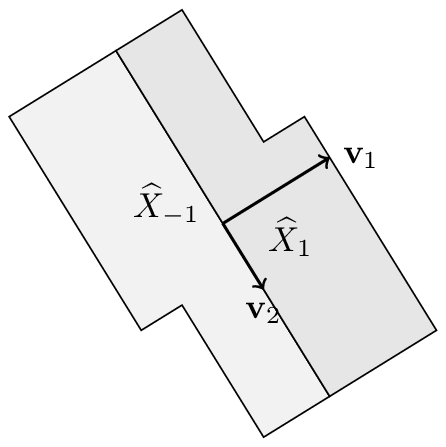}
\includegraphics{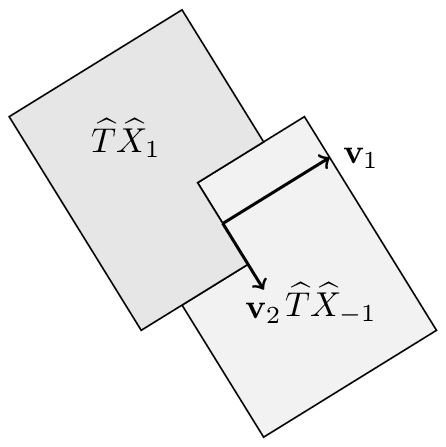}
\caption{The transformation from Example~\ref{x:gm4} and its natural extension.}
\label{f:xhat4}
\end{figure}
\end{ex}

\begin{ex}\label{x:mtribo}
Consider now a minimal weight transformation with the Tribonacci number~$\beta$, which is the real solution of the equation $x^3 - x^2 - x - 1 = 0$. 
Choose $\alpha = \frac{1}{\beta+1}$, i.e., let $A = \{-1,0,1\}$, $X_{-1} = \big[\frac{-\beta}{\beta+1}, \frac{-1}{\beta+1}\big)$, $X_0 = \big[\frac{-1}{\beta+1}, \frac1{\beta+1}\big)$, $X_1 = \big[\frac1{\beta+1}, \frac\beta{\beta+1}\big)$.
Then $\tilde{b}(-\alpha) = \bar{1}(010)^{\omega}$, $b(-\alpha) = (0\bar{1}0)^{\omega}$, $\tilde{b}(\alpha) = (010)^{\omega}$, $b(\alpha) = 1(0\bar{1}0)^{\omega}$, thus $\mathcal{V} = \big\{\frac{-\beta}{\beta+1}, \frac{-1}{\beta+1}, \frac{-1/\beta}{\beta+1}, \frac{1/\beta}{\beta+1}, \frac1{\beta+1}\big\}$.
We see the transformation and $\widehat{X}$ in Figure~\ref{f:mtribo}.

\begin{figure}[ht]
\centering
\includegraphics{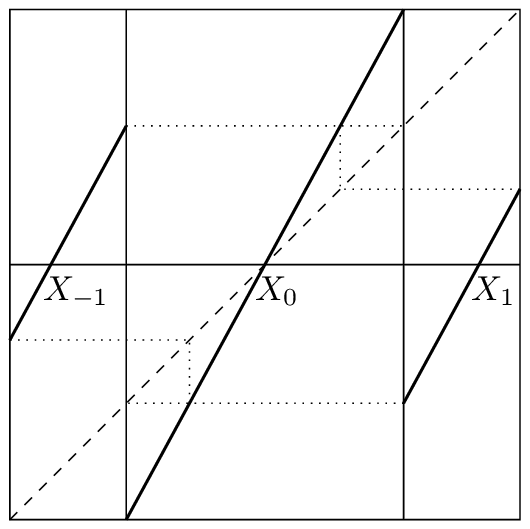}
\qquad
\includegraphics{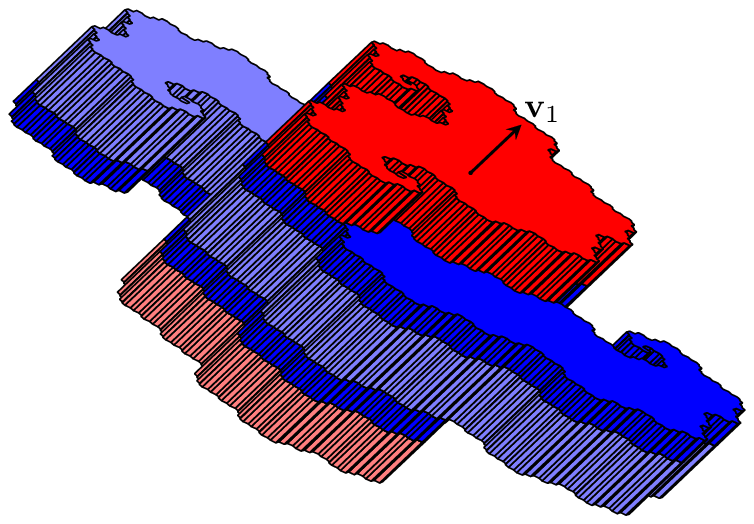}
\caption{The transformation from Example~\ref{x:mtribo} and its natural extension domain.}
\label{f:mtribo}
\end{figure}
\end{ex}

\begin{ex}\label{x:mpisot}
If $\beta$ is the smallest Pisot number, i.e., the real solution of the equation $x^3 - x - 1 = 0$, then $\alpha = \frac{\beta^6}{\beta^8-1}$ provides a minimal weight transformation as in Example~\ref{x:mtribo}.
We have $\tilde{b}(\alpha) = (010^6)^{\omega}$, $b(\alpha) = 1(0^6\bar{1}0)^{\omega}$, thus $\mathcal{V} = \big\{\frac{\pm\beta^k}{\beta^8-1}:\, 0 \le k \le 7\} \setminus \big\{\frac{\beta^7}{\beta^8-1}\}$ since $\tilde{b}(-\alpha)$ and $b(-\alpha)$ are obtained by symmetry from $b(\alpha)$ and $\tilde{b}(\alpha)$. 
See Figure~\ref{f:mpisotnatex}.

\begin{figure}[ht]
\centering
\includegraphics{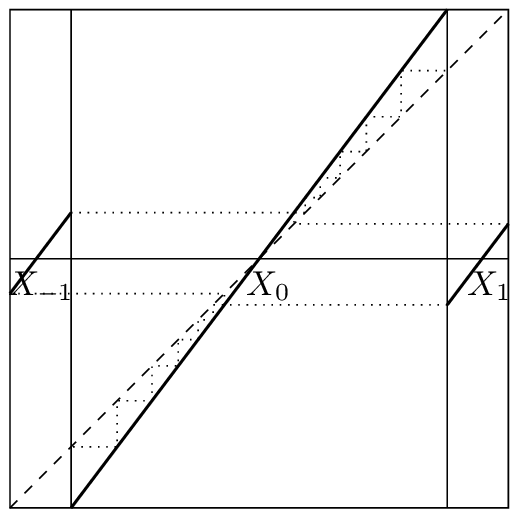}
\qquad
\includegraphics{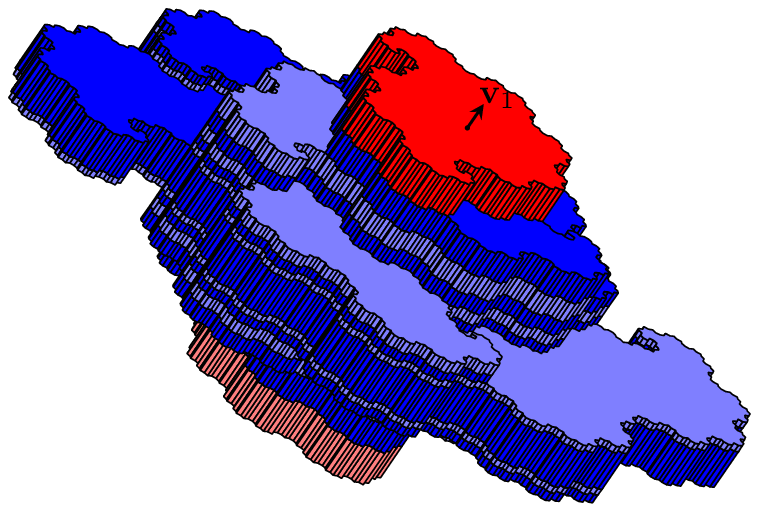}
\caption{The transformation from Example~\ref{x:mpisot} and its natural extension domain.}
\label{f:mpisotnatex}
\end{figure}
\end{ex}

\section{Tilings} \label{s:tilings}

In this section, we consider two types of (multiple) tilings which are closely related. The first one is an aperiodic (multiple) tiling of the hyperplane $H$ by sets $\mathcal{D}_x$ defined in~(\ref{e:Dx}). The second one is a periodic (multiple) tiling of $\mathbb{R}^d$ by the closure $\widehat{Y}$ of the natural extension domain $\widehat{X}$.

As for the natural extensions, $T$~is a right-continuous $\beta$-transformation and $\beta$~a Pisot unit.
We furthermore assume that $A \subset \mathbb{Z}[\beta]$, i.e., that $q = 1$ in Lemma~\ref{l:positivemeasure}, and that the set $\mathcal{V}$, which is defined in (\ref{q:v}), is finite.

\subsection{Tiling of the contracting hyperplane} \label{sec:tiling-contr-hyperpl}

We define tiles in the hyperplane $H$ by 
\[ 
\mathcal{T}_x = \Phi(x) + \mathcal{D}_x \quad\mbox{for all}\quad x \in \mathbb{Z}[\beta] \cap X, 
\] 
with $\Phi$ as in~(\ref{e:PhiPsi}) and $\mathcal{D}_x$ as in~(\ref{e:Dx}).

\begin{remark}
The tiles are often defined in $\mathbb{R}^r \times \mathbb{C}^s$, where $r$ is the number of real conjugates and $2s$ is the number of complex conjugates of~$\beta$.
We clearly have $\mathbb{R}^r \times \mathbb{C}^s \simeq H$.
We choose to work in $H$ because many statements are easier to formulate in $H$ than in $\mathbb{R}^r \times \mathbb{C}^s$.
\end{remark}

The family $\mathcal{T} = \{\mathcal{T}_x\}_{x\in\mathbb{Z}[\beta]\cap X}$ is a \emph{multiple tiling} of the space $H$ if the following properties hold.
\begin{itemize}
\itemsep5pt
\item[(mt1)] 
There are only finitely many different sets $\mathcal{D}_x$, and these sets are compact.
\item[(mt2)] 
The family $\mathcal{T}$ is \emph{locally finite}, i.e., for every $\mathbf{y} \in H$, there is a positive $r$ such that the set $\{x \in \mathbb{Z}[\beta] \cap X:\, \mathcal{T}_x \cap B(\mathbf{y},r) \neq \emptyset\}$ is finite. 
\item[(mt3)] 
$\mathcal{T}$ gives a covering of $H$: for every $\mathbf{y} \in H$, there is a tile $\mathcal{T}_x$, such that $\mathbf{y} \in \mathcal{T}_x$.
\item[(mt4)] 
Every set $\mathcal{D}_x$, $x \in \mathbb{Z}[\beta] \cap X$, is the closure of its interior.
\item[(mt5)] 
There is an integer $m \ge 1$ such that almost all points of $H$ are in exactly $m$ different tiles. The number $m$ is called the \emph{covering degree} of the multiple tiling.
\end{itemize}
A~\emph{tiling} is a multiple tiling with covering degree~$1$. 

The tiles $\mathcal{T}_x$, $x \in \mathbb{Z}[\beta] \cap X$, are translates of a finite collection of compact sets by Lemma~\ref{l:tilecompact}, Proposition~\ref{p:S} and the finiteness of~$\mathcal{V}$. 
This proves~(mt1). 
An important tool for showing (mt2)--(mt5) will be that the set of translation vectors, $\Phi(\mathbb{Z}[\beta] \cap X)$, is a \emph{Delone set} in~$H$, i.e., that it is uniformly discrete and relatively dense in~$H$. Precisely,
\begin{itemize}
\itemsep5pt
\item 
A~set $Z$ is \emph{relatively dense} in $H$ if there is an $R>0$, such that, for every $\mathbf{y} \in H$, $B(\mathbf{y}, R) \cap Z \neq \emptyset$.
\item 
A~set $Z$ is \emph{uniformly discrete} if there is an $r>0$ such that, for every $\mathbf{y} \in Z$, the set $B(\mathbf{y}, r) \cap Z$ contains only one element.
\end{itemize}
In \cite{Moody97}, Moody studied, among other things, Delone sets. He gave a detailed exposition of Meyer's theory, which was developed in~\cite{Meyer72}. According to Meyer, model sets for cut and project schemes are Delone. We will use this to prove Lemma~\ref{l:udrd}. 
We also need the following lemma.

\begin{lemma} \label{l:phi}
The map $\Phi:\, \mathbb{Q}(\beta) \to H$ is injective.
\end{lemma}

\begin{proof}
Recall that $\Phi(x) = \Psi(x) - x \mathbf{v}_1$, and that $\Psi:\, \mathbb{Q}(\beta) \to \mathbb{Q}^d$ is bijective by Lemma~\ref{l:zd}.
Since the coefficients of $\mathbf{v}_1$ are linearly independent over~$\mathbb{Q}$, $\Phi$~is injective.
\end{proof}

\begin{lemma}\label{l:udrd}
Every set $\Phi(\mathbb{Z}[\beta] \cap E)$, where $E \subset \mathbb{R}$ is bounded and has non-empty interior, is uniformly discrete and relatively dense. 
In particular, this holds for $\Phi(\mathbb{Z}[\beta] \cap X)$. 
\end{lemma}

\begin{proof}
By Proposition~2.6 from~\cite{Moody97}, it suffices to show that $\Phi(\mathbb{Z}[\beta] \cap E)$ is a model set.

Define the projection $\pi_H:\, \mathbb{R}^d \to H$ by $\pi_H(\mathbf{x}) = \mathbf{x} - \pi_1(\mathbf{x}) \mathbf{v}_1$, with $\pi_1$ as in Section~\ref{sec:geom-real-natur}, and set $\iota(\mathbf{x}) = (\pi_H(\mathbf{x}), \pi_1(\mathbf{x}))$. 
Then the pair $(H \times \mathbb{R}, \iota(\mathbb{Z}^d))$ is a cut and project scheme, since $\iota(\mathbb{Z}^d)$ is a lattice, $\pi_H$ is injective on $\mathbb{Z}^d$, $\pi_1(\mathbb{Z}^d) = \mathbb{Z}[\beta]$ is dense in $\mathbb{R}$ and $H \simeq \mathbb{R}^{d-1}$.
The injectivity of $\pi_H$ on $\mathbb{Z}^d$ follows again from the linear independence of the coefficients of $\mathbf{v}_1$ over $\mathbb{Q}$.

Now, let $E \subset \mathbb{R}$ be a bounded set with non-empty interior. Then
\[ 
\Phi(\mathbb{Z}[\beta] \cap E) = \big\{\pi_H(\mathbf{x}):\, \mathbf{x} \in \mathbb{Z}^d,\ \pi_1(\mathbf{x}) \in E\big\} 
\]
is a model set, and therefore a Delone set by Proposition~2.6 from~\cite{Moody97}.
\end{proof}

\begin{cor}\label{c:locallyfinite}
The family $\{\mathcal{T}_x\}_{x\in\mathbb{Z}[\beta]\cap X}$ is locally finite.
\end{cor}

\begin{proof}
This follows immediately from the uniform discreteness of $\Phi(\mathbb{Z}[\beta] \cap X)$, the injectivity of~$\Phi$ and the fact that $\|\varphi(w)\|$, $w \in {\vphantom{A}}^{\omega}\!A$, is bounded. 
\end{proof}

Hence we have~(mt2). The next lemma gives~(mt3). 

\begin{lemma} \label{l:cover}
We have $H = \bigcup_{x\in\mathbb{Z}[\beta]\cap X} \mathcal{T}_x$.
\end{lemma}

\begin{proof}
Let $H' = \bigcup_{x\in\mathbb{Z}[\beta]\cap X} \mathcal{T}_x$.
Every point $\mathbf{y} \in H'$ is of the form $\mathbf{y} = \Phi(x) + \varphi(w)$, with $x \in \mathbb{Z}[\beta] \cap X$, $w \in {\vphantom{A}}^{\omega}\!A$, $w \cdot b(x) \in\mathcal{S}$.
We have $M_{\beta} \mathbf{y} = \varphi(w b_1(x)) + \Phi(T x)$, with $T x \in \mathbb{Z}[\beta] \cap X$ since $b_1(x) \in \mathbb{Z}[\beta]$, and thus $M_{\beta} \mathbf{y} \in H'$. Hence, $M_{\beta} H' \subseteq H'$.
Since $T$ is surjective, every tile $\mathcal{T}_x$, $x \in \mathbb{Z}[\beta] \cap X$, is non-empty.
By Lemma~\ref{l:udrd} and since $\varphi(w)$ is bounded, $H'$~is relatively dense in~$H$.
Since $M_{\beta} H' \subseteq H'$ and $M_{\beta}$ is contracting on $H$, $H'$~is dense in~$H$.
By the compactness of the tiles and the local finiteness of $\mathcal{T}$, we obtain $H' = H$.
\end{proof}

Now, we can prove that the boundary of every tile has zero measure. 
This generalizes Theorem~3 in~\cite{Akiyama99}. 

\begin{prop} \label{p:boundary}
We have $\lambda^{d-1}(\partial \mathcal{D}_x) = 0$ for every $x \in X$.
\end{prop}

\begin{proof}
If $\lambda^{d-1}(\mathcal{D}_x) = 0$, then $\lambda^{d-1}(\partial \mathcal{D}_x) = 0$ since $\mathcal{D}_x$ is compact. 
Therefore, we can assume $\lambda^{d-1}(\mathcal{D}_x) > 0$.
We first show that $\lambda^{d-1}(\partial \mathcal{D}_y) = 0$ for some $y \in X$ with $\lambda^{d-1}(\mathcal{D}_y) > 0$, and then extend this property to arbitrary $x \in X$.

Since $H = \bigcup_{z\in\mathbb{Z}[\beta]\cap X} \mathcal{T}_z$, there exists, by Baire's theorem, some $z \in \mathbb{Z}[\beta] \cap X$ such that $\mathcal{T}_z$, and thus $\mathcal{D}_z$, has an inner point (with respect to~$H$).
By iterating (\ref{e:decomposition}), we obtain for all $k \ge 1$ that $\mathcal{D}_z$ is the (up to sets of measure zero) disjoint union of $M_{\beta}^k \mathcal{D}_y + \varphi(b_1(y) \cdots b_k(y))$, $y\in T^{-k}\{z\}$.
Since $M_{\beta}^{-1}$ is expanding on~$H$, there must be some $y \in T^{-k} \{z\}$ for sufficiently large $k$ such that $M_{\beta}^k \mathcal{D}_y + \varphi(b_1(y) \cdots b_k(y))$ is contained in the interior of~$\mathcal{D}_z$, and $\lambda^{d-1}(\mathcal{D}_y) > 0$.
Then every point in $\partial(M_{\beta}^k \mathcal{D}_y + \varphi(b_1(y) \cdots b_k(y)))$ lies also in $M_{\beta}^k \mathcal{D}_{y'} + \varphi(b_1(y')\cdots b_k(y'))$ for some $y' \in T^{-k} \{x\}$, $y'\ne y$. Since the intersection of these sets has zero measure, we obtain $\lambda^{d-1}(\partial \mathcal{D}_y) = 0$.

Now, consider a set $M_{\beta} \mathcal{D}_z + \varphi(b_1(z))$, $z \in T^{-1} \{y\}$, in the subdivision of~$\mathcal{D}_y$. Every point on the boundary of this set is either also on the boundary of another set from the subdivision, or not. If not, then the point is in~$\partial \mathcal{D}_y$. Therefore, we have $\lambda^{d-1}(\partial \mathcal{D}_z) = 0$ for every $z \in T^{-1} \{y\}$.
It remains to show that, when iterating this argument, every $\mathcal{D}_x$ with $\lambda^{d-1}(\mathcal{D}_x) > 0$ occurs eventually in one of these subdivisions.
By Proposition~\ref{p:S}, it is sufficient to consider $x \in \mathcal{V}$. 

Similarly to the graph of the GIFS, let $\mathcal{G}$ be the weighted directed graph with set of vertices $\mathcal{V}' = \{x \in \mathcal{V}:\, \lambda^{d-1}(\mathcal{D}_x) > 0\}$ and an edge from $x$ to $x'$ if and only if $T^{-1} J_x \cap J_{x'} \ne \emptyset$.
Then we have $\lambda^{d-1}(\partial \mathcal{D}_{x'}) = 0$ for every $x' \in \mathcal{V}$ which can be reached from the vertex $x$ satisfying $y \in J_x$. 
Let the weight of an edge be $\lambda^d(\widehat{T}^{-1}(J_x \mathbf{v}_1 - \mathcal{D}_x) \cap (J_{x'} \mathbf{v}_1 - \mathcal{D}_{x'})) > 0$.
Since $\widehat{T}$ is bijective off of a set of Lebesgue measure zero, the sum of the weights of the outgoing edges as well as the sum of the weights of the ingoing edges must equal $\lambda^d(J_x \mathbf{v}_1 - \mathcal{D}_x)$ for every $x \in \mathcal{V}'$.
This implies that every connected component of $\mathcal{G}$ must be strongly connected, in the sense that if two vertices are in the same connected component, then there is a path from one vertex to the other and the other way around. 

Let $\mathcal{C}$ be a (strongly) connected component of~$\mathcal{G}$ and $X' = \bigcup_{x\in\mathcal{C}} J_x$. 
For every $z \in J_x$, $x \in \mathcal{C}$, $\mathcal{D}_{T z}$ contains $M_\beta \mathcal{D}_z + b_1(z)$, thus $\lambda^{d-1}(\mathcal{D}_{T z}) > 0$, hence $T z \in X'$.
Therefore, the restriction of $T$ to $X'$ is a right-continuous $\beta$-transformation, and this restriction changes every~$\mathcal{D}_x$, $x \in X'$, only by a set of measure zero.
The arguments of the preceding paragraph provide some $y \in X'$ with $\lambda^{d-1}(\partial \mathcal{D}_y) = 0$, and we obtain $\lambda^{d-1}(\partial \mathcal{D}_x) = 0$ for all $x \in X'$, thus for all $x \in X$.
\end{proof}

If we want that (mt4) holds, then we clearly have to exclude non-empty tiles of measure zero. 
This means that $X$ has to be the support of the invariant measure $\mu = \lambda^d \circ \pi^{-1}$, which is defined in Section~\ref{sec:geom-real-natur}.
Note that restricting $T$ to the support of $\mu$ changes the tiles only by sets of measure zero.

For a subset $E \subseteq H$, let $\inn(E)$ denote the interior of~$E$ (in~$H$), and let $\overline{E}$ be the closure of~$E$.

\begin{lemma} \label{l:closureinterior}
If the support of the invariant measure $\mu$ is~$X$, then $\mathcal{D}_x = \overline{\inn(\mathcal{D}_x)}$ for every $x \in X$.
\end{lemma}

\begin{proof}
Since the tiles are compact, we have $\overline{\inn(\mathcal{D}_x)} \subseteq \mathcal{D}_x$ for every $x \in X$. For the other inclusion, we show that $\mathcal{T}_x \subseteq \overline{\inn(\mathcal{T}_x)}$ for every $x \in \mathbb{Z}[\beta] \cap X$.
If $\mathbf{y} \in \mathcal{T}_x$, then $\mathbf{y} = \Phi(x) + \varphi(w)$ for some $w = (w_k)_{k\le0} \in {\vphantom{A}}^{\omega}\!A$ with $w \cdot b(x) \in \mathcal{S}$. 
Set $x_k = \decdot w_{-k+1} \cdots w_0 b(x)$ for $k \ge 0$. 
Then $M_{\beta}^k (\Phi(x_k) + \mathcal{D}_{x_k}) \subseteq \mathcal{T}_x$.
We have $\lambda^{d-1}(\mathcal{D}_{x_k}) > 0$ since the support of $\mu$ is~$X$. Also, $\lambda^{d-1}(\partial \mathcal{D}_{x_k}) = 0$ by Proposition~\ref{p:boundary}.
Therefore, there exists a point $\mathbf{y}_k \in M_{\beta}^k (\Phi(x_k) + \inn(\mathcal{D}_{x_k})) \subseteq \inn(\mathcal{T}_x)$.
Since $\lim_{k\to\infty} \mathbf{y}_k = \mathbf{y}$, we have $\mathbf{y} \in  \overline{\inn(\mathcal{T}_x)}$.
\end{proof}

The family $\mathcal{T}$ is \emph{self-replicating}, i.e., for each element $\mathcal{T}_x$ of $\mathcal{T}$, we can write $M^{-1}_{\beta} \mathcal{T}_x$ as the (up to sets of measure zero) disjoint union of elements from $\mathcal{T}$.
By Lemma~\ref{l:subdivision}, we have 
\begin{equation} \label{e:subdivision}
M_{\beta}^{-1} \mathcal{T}_x = M_{\beta}^{-1} \bigg(\Phi(x) + \bigcup_{y\in T^{-1}\{x\}} \Big(M_{\beta} \mathcal{D}_y + \Phi\big(b_1(y)\big)\Big)\bigg) = \bigcup_{y\in T^{-1}\{x\}} \mathcal{T}_y\,. 
\end{equation}

For $\mathbf{y} \in H$ and $r > 0$, the \emph{local arrangement} in $B(\mathbf{y},r)$ is the set
\[ 
\mathcal{P}\big(B(\mathbf{y},r)\big) = \big\{\mathcal{T}_x \in \mathcal{T}:\, \mathcal{T}_x \cap B(\mathbf{y},r) \neq \emptyset\big\}. 
\]
As a next step in proving that the family $\mathcal{T}$ is a multiple tiling, we will show that $\mathcal{T}$ is \emph{quasi-periodic}, i.e., for any $r > 0$, there is an $R > 0$ such that, for any $\mathbf{y}, \mathbf{y}' \in H$, the local arrangement in $B(\mathbf{y},r)$ appears up to translation in the ball $B(\mathbf{y}',R)$. 

\begin{prop} \label{p:quasiperiodic}
The family $\mathcal{T}$ is quasi-periodic.
\end{prop}

\begin{proof}
Note first that, for each $r > 0$, there are only finitely many different local arrangements up to translation, since $\Phi(\mathbb{Z}[\beta] \cap X)$ is uniformly discrete and there are only finitely many different sets~$\mathcal{D}_x$.

Let $r>0$. If two tiles $\mathcal{T}_x$ and $\mathcal{T}_{x+y}$ are in the same local arrangement, then
\begin{equation} \label{q:smally}
\|\Phi(y)\| < 2 (r + \max_{w\in{\vphantom{A}}^{\omega}\!A} \|\varphi(w)\|\big) \quad\mbox{and}\quad y \in X - X.
\end{equation}
By Lemma~\ref{l:udrd} and since $\Phi$ is injective, the set of elements $y \in \mathbb{Z}[\beta]$ satisfying (\ref{q:smally}) is finite. Call this set~$F$, and note that $0 \in F$.

Now, take some local arrangement $\mathcal{P}(B(\mathbf{y},r))$ and $x \in \mathbb{Z}[\beta] \cap X$ such that $\mathcal{T}_x \in  \mathcal{P}(B(\mathbf{y},r))$.
For any $y \in F$, if $x + y \in X$, then there is an $\varepsilon_y > 0$ such that $J_{x+y} = [x + y, x + y + \varepsilon_y)$, and if $x + y \not\in X$, then there is an $\varepsilon_y > 0$ such that $[x + y, x + y + \varepsilon_y) \cap X = \emptyset$. 
Let $\varepsilon = \min_{y\in F} \varepsilon_y$ and consider some $z \in \mathbb{Z}[\beta] \cap [0,\varepsilon)$.
For any $y \in F$, we have either $x + y \in X$ and $\mathcal{T}_{x+y+z} = \mathcal{T}_{x+y} + \Phi(z)$ since $\mathcal{D}_{x+y+z} = \mathcal{D}_{x+y}$, or $x + y \not\in X$ and $x + y + z \not\in X$.
Therefore, $\mathcal{P}(B(\mathbf{y} + \Phi(z),r))$ is equal to $\mathcal{P}(B(\mathbf{y},r))$, up to translation by $\Phi(z)$.

By Lemma~\ref{l:udrd}, $\Phi(\mathbb{Z}[\beta] \cap [0,\varepsilon))$ is relatively dense in~$H$. Thus, every local arrangement occurs relatively densely in~$H$. Since the number of local arrangements is finite, the lemma is proved.
\end{proof}

Lemma~\ref{l:cover} implies that every $\mathbf{y} \in H$ lies in at least one tile. 
By the local finiteness of $\mathcal T$, there exists an $m \ge 1$ such that every element of $H$ is contained in at least $m$ tiles in $\mathcal{T}$ and there exist elements of $H$ that are not contained in $m + 1$ tiles. 
For this~$m$, a point $\mathbf{y} \in H$ lying in exactly $m$ tiles is called an \emph{$m$-exclusive point}. 
A point lying in exactly one tile is called an \emph{exclusive point}.
Similarly to \cite{ItoRao06}, we obtain the following proposition, which gives~(mt5).

\begin{prop} \label{p:itorao}
There exists an $m \ge 1$ such that almost every $\mathbf{y} \in H$ is contained in exactly $m$ tiles.
\end{prop}

\begin{proof}
We first show that the set of points that do not lie on the boundary of a tile is open and of full measure. Let $C = \bigcup_{x\in \mathbb{Z}[\beta] \cap X} \partial \mathcal{T}_x$ denote the union of the boundaries of all the tiles in~$\mathcal{T}$. This set is closed, since it is the countable union of closed sets that are locally finite. Hence, $H \setminus C$ is open. By Proposition~\ref{p:boundary}, we also have $\lambda^{d-1}(C) = 0$.

Let $m$ be as in the paragraph preceding the proposition, and $\mathbf{x} \in H$ be an $m$-exclusive point, lying in the tiles $\mathcal{T}_{x_1}, \ldots, \mathcal{T}_{x_m}$. Since the tiles are closed, there is an $\varepsilon > 0$ such that $B(\mathbf{x},\varepsilon) \cap \mathcal{T}_x = \emptyset$ for all $x \in (\mathbb{Z}[\beta] \cap X) \setminus \{x_1,\ldots,x_m\}$. Since every point lies in at least $m$ tiles, we have $B(\mathbf{x},\varepsilon) \subseteq \mathcal{T}_{x_k}$ for $1 \le k \le m$.
By the self-replicating property, $M_{\beta}^{-1} \mathcal{T}_{x_k}$ subdivides into tiles from $\mathcal{T}$ with disjoint interior, hence almost every point in $M^{-1}_{\beta} B(\mathbf{x},\varepsilon)$ is also contained in exactly $m$ tiles. The same holds for almost every point in $M^{-n}_{\beta} B(\mathbf{x},\varepsilon)$, for all $n \ge 1$.

Now take a point $\mathbf{y} \in H \setminus C$. Since $H \setminus C$ is open, there is an $r > 0$ such that $B(\mathbf{y},r) \subseteq H \setminus C$, i.e., every point in $B(\mathbf{y},r)$ lies in the same set of tiles. By the quasi-periodicity, translations of the local arrangement $\mathcal{P}(B(\mathbf{y}, r))$ occur relatively densely in~$H$. Since the matrix $M_{\beta}^{-1}$ is expanding on~$H$, there is therefore an $n \ge 1$ such that $M^{-n}_{\beta} B(\mathbf{x},\varepsilon)$ contains a translation of $\mathcal{P}(B(\mathbf{y},r))$, which implies that $B(\mathbf{y},r)$ and thus $\mathbf{y}$ lies in exactly $m$ tiles.
\end{proof}

We have now established all the properties of a multiple tiling.

\begin{theorem} \label{t:mtiling}
Let $T:\, X \to X$ be a right-continuous $\beta$-transformation as in Definition~\ref{d:trfm} with a Pisot unit~$\beta$ and $A \subset \mathbb{Z}[\beta]$.
Assume that the invariant measure~$\mu$, given by the natural extension in Section~\ref{sec:geom-real-natur}, has support~$X$, and that the set~$\mathcal{V}$, defined by (\ref{q:v}), is finite.
Then the family $\{\mathcal{T}_x\}_{x\in\mathbb{Z}[\beta]\cap X}$ forms a multiple tiling of the hyperplane~$H$.
\end{theorem}

\begin{remark} \label{r:tiling}
Let $T$ satisfy the assumptions of Theorem~\ref{t:mtiling}.
Then the family $\{\mathcal{T}_x\}_{x\in\mathbb{Z}[\beta]\cap X}$ forms a tiling of $H$ if and only if there exists a point in $H$ which lies in exactly one tile.
\end{remark}

\begin{remark}
If the support of $\mu$ is not equal to~$X$, then there exist tiles of measure zero, and (mt4) does not hold, see Example~\ref{x:gm3} and Remark~\ref{r:zerotiles}. However, we have $\lambda^{d-1}(\partial \mathcal{D}_x) = 0$ for every $x \in \mathbb{Z}[\beta] \cap X$, and the conditions (mt1)--(mt3) and (mt5) are still satisfied.
\end{remark}

\subsection{To find an exclusive point} \label{sec:find-an-m}

Theorem~\ref{t:mtiling} gives conditions under which $\mathcal{T}$ is a multiple tiling, but it gives no information about the covering degree. 
Indeed, it is quite difficult to determine this degree if $m \ge 2$, because one has two find a set of positive measure lying in $m$ tiles, cf.\ Section~\ref{sec:symm-beta-transf}.
It is much easier to prove that $\mathcal{T}$ is a tiling, by Remark~\ref{r:tiling}.

Denote by $P$ the set of points $x \in \mathbb{Z}[\beta] \cap X$ with purely periodic $T$-expansion, i.e., $T^n x = x$ for some $n \ge 1$.
Using the characterization of purely periodic points in Theorem~\ref{t:purelyperiodic}, we obtain the following lemma.

\begin{lemma} \label{l:P}
The origin $\mathbf{0}$ belongs to $\mathcal{T}_x$, $x \in \mathbb{Z}[\beta] \cap X$, if and only if $x \in P$.
The set $P$ is finite.
\end{lemma}

\begin{proof}
Let $x \in \mathbb{Z}[\beta] \cap X$.
By definition, we have $\mathcal{T}_x = \Phi(x) + \mathcal{D}_x$, thus $\mathbf{0} \in \mathcal{T}_x$ if and only if $\Phi(x) \in -\mathcal{D}_x$, i.e., $\Psi(x) \in x \mathbf{v}_1 - \mathcal{D}_x$.
By (\ref{e:Dx}) and Theorem~\ref{t:purelyperiodic}, this is equivalent with $x \in P$.
The finiteness of $P$ follows from the local finiteness of~$\mathcal{T}$.
\end{proof}

Therefore, $\mathbf{0}$~is an exclusive point if and only if $P$ consists only of one element.
This generalizes the (F) property which was introduced in~\cite{FrougnySolomyak92}, see also~\cite{Akiyama99}.
If $\mathbf{0}$ is contained in more than one tile, then it is more difficult to determine the covering degree~$m$. Corollary~\ref{c:intiles} provides an easy way to determine the number of tiles to which a point belongs.
We restrict to points $\Phi(z)$, $z \in \mathbb{Z}[\beta] \cap [0,\infty)$, because of the following lemma.

\begin{lemma}[\cite{Akiyama99}] \label{l:inftydense}
The set $\Phi(\mathbb{Z}[\beta] \cap [0,\infty))$ is dense in~$H$.
\end{lemma}

\begin{proof}
By Lemma~\ref{l:udrd}, $\Phi(\mathbb{Z}[\beta] \cap [0,1))$ is relatively dense in~$H$. 
We have
\[ 
\Phi\big(\mathbb{Z}[\beta] \cap [0,\infty)\big) = \bigcup_{n\ge0} \Phi\big(\mathbb{Z}[\beta] \cap [0,\beta^n)\big) = \bigcup_{n\ge0} M_{\beta}^n \Phi\big(\mathbb{Z}[\beta] \cap[0,1)\big). 
\]
Since $M_{\beta}$ is contracting, we obtain that $\Phi(\mathbb{Z}[\beta] \cap [0,\infty))$ is dense in~$H$.
\end{proof}

\begin{prop} \label{p:inTx}
A~point~$\Phi(z)$, $z \in \mathbb{Z}[\beta] \cap [0,\infty)$, lies in the tile~$\mathcal{T}_x$, $x \in \mathbb{Z}[\beta] \cap X$, if and only if there exists some $y \in P$ and some $k \ge 0$ such that
\begin{equation} \label{e:pk}
T^k (y + \beta^{-k} z) = x \quad\mbox{and}\quad [T^j y, T^j y + \beta^{-k-n+j} z] \subset X_{b_1(T^j y)} \quad\mbox{for all}\ 0 \le j < n,
\end{equation}
where $n \ge 1$ is the period length of~$y$, i.e., $T^n y = y$.
\end{prop}

\begin{proof}
Let $\Phi(z) \in \mathcal{T}_x$, which means that $\Phi(z) = \Phi(x) + \varphi(w)$ for some $w \in {\vphantom{A}}^{\omega}\!A$ with $w \cdot b(x) \in \mathcal{S}$, and set $x_k = \decdot w_{-k+1} \cdots w_0 b(x)$, $w^{(k)} = \cdots w_{-k-1} w_{-k}$, for $k \ge 0$. 
Then
\begin{equation} \label{e:xk}
\Phi(x_k - \beta^{-k} z) = \Phi(x_k) - M_{\beta}^{-k} \big(\Phi(x) + \varphi(w)\big) = M_{\beta}^{-k} \big(\varphi(w_{-k+1} \cdots w_0) - \varphi(w)\big) = -\varphi\big(w^{(k)}\big),
\end{equation}
therefore $\Phi(x_k - \beta^{-k} z)$ is bounded, hence the set $\{x_k - \beta^{-k} z:\, k \ge 0\}$ is finite by Lemma~\ref{l:udrd} and the injectivity of $\Phi$.
This means that there exists some $y$ such that $x_k - \beta^{-k} z = y$ for infinitely many $k \ge 0$.
Since $x_k \in X$ and $X$ is left-closed, we obtain $y \in X$.
Moreover, we can choose $y$ such that $\{k \ge 0:\, x_k - \beta^{-k} z = y\}$ has bounded gaps.
Therefore, there exist two successive elements $k$ and $k + n$ of this set such that $[T^j y, T^j y + \beta^{-k-n+j} z] \subset X_{b_1(T^j y)}$ for all $0 \le j < n$.
We have 
\[
T^n y = T^n (y + \beta^{-k-n} z) - \beta^{-k} z =  T^n x_{k+n} - \beta^{-k} z = x_k - \beta^{-k} z = y,
\]
thus $y \in P$ and $T^k (y + \beta^{-k} z) = T^k x_k = x$.

For the other direction, assume that (\ref{e:pk}) holds. 
Then we have $T^{j n}(y + \beta^{-k-j n} z) = y + \beta^{-k} z$ and thus $T^{k+j n}(y + \beta^{-k-j n} z) = x$ for all $j \ge 0$.
Let $w^{(j)} \in {\vphantom{A}}^{\omega}\!A$, $j \ge 0$, be sequences with $w^{(j)} \cdot b(y + \beta^{-k-j n} z) \in \mathcal{S}$, and set 
\[ 
\mathbf{z}_j = M_{\beta}^{k+j n} \big(\Phi(y + \beta^{-k-j n} z) + \varphi\big(w^{(j)}\big)\big) = \Phi(z) + M_{\beta}^{k+j n} \big(\Phi(y) + \varphi\big(w^{(j)}\big)\big). 
\]
Then we have $\mathbf{z}_j \in \mathcal{T}_{T^{k+j n}(y+\beta^{-k-j n}z)} = \mathcal{T}_x$ and $\lim_{j\to\infty} \mathbf{z}_j = \Phi(z)$, thus $\Phi(z) \in \mathcal{T}_x$.
\end{proof}

\begin{cor} \label{c:intiles}
Let $z \in \mathbb{Z}[\beta] \cap [0,\infty)$.
Choose $k \ge 0$ with $y + \beta^{-k} z \in X$, $[y, y + \beta^{-k-1} z] \subset X_{b_1(y)}$ for all $y \in P$. 
Then $\Phi(z)$ lies exactly in the tiles $\mathcal{T}_{T^k(y+\beta^{-k}z)}$, $y \in P$.
\end{cor}

\begin{proof}
This follows from Proposition~\ref{p:inTx}, since $y \in P$ implies $T^j y \in P$ for all $j \ge 0$.
\end{proof}

The proof of the following proposition shows how to construct an exclusive point from points with weaker properties.
This generalizes the (W) property and Proposition~2 in~\cite{Akiyama02}, where $x = 0$, since $T_{\beta}^k z = 0$ means that the $\beta$-expansion of $z$ is finite.

\begin{prop} \label{p:W}
Let $\varepsilon = \min_{y\in P} \big(\!\min\{z > y:\, z \not\in X_{b_1(y)}\} - y\big)\, \beta$.
Then $\mathcal{T}_x$, $x \in P$, contains an exclusive point if and only if, for every $y \in P$, there exists some $z \in \mathbb{Z}[\beta] \cap [0, \varepsilon)$ and some $k \ge 0$ such that
\begin{equation} \label{e:p0}
T^k (y + z) = T^k (x + z) = x.
\end{equation}
\end{prop}

\begin{proof}
If $\mathcal{T}_x$ contains an exclusive point, then it contains an exclusive point $\Phi(z')$, $z' \in \mathbb{Z}[\beta] \cap [0, \infty)$ by Lemma~\ref{l:inftydense}. 
By Corollary~\ref{c:intiles}, we have some $k \ge 0$ with $T^k(y + \beta^{-k} z') = x$ for all $y \in P$, and $[y, y + \beta^{-k-1} z] \subset X_{b_1(y)}$.
This implies $\beta^{-k} z' < \varepsilon$, hence we can choose $z = \beta^{-k} z'$.

For the converse, let $P = \{x,y_1,\ldots,y_h\}$ and assume that, for every $y \in P$, there exists some $z \in \mathbb{Z}[\beta] \cap [0,\varepsilon)$ and some $k \ge 0$ such that (\ref{e:p0}) holds. 
It suffices to consider the case $h \ge 2$, since $\mathbf{0}$ is an exclusive point of $\mathcal{T}_x$ in case $h = 0$ and, in case $h = 1$, Corollary~\ref{c:intiles} shows that $\beta^k z$ is an exclusive point of $\mathcal{T}_x$ if $z \in \mathbb{Z}[\beta] \cap [0, \varepsilon)$ and $k$ are such that $T^k (y_1 + z) = T^k (x + z) = x$.
If $h \ge 2$, then we will recursively construct a point $z_h$ for which $T^k (y + \beta^{-k} z_h) = x$ for all $y \in P$, with some $k \ge 0$ satisfying the conditions of Corollary~\ref{c:intiles}.

First note that, if (\ref{e:p0}) holds for some $k,\, z$, then it also holds for $z' = \beta^{-n} z$ and $k' = k + n$ if  $n$ is a multiple of the period lengths of $x$ and $y$, since $T^n (y + \beta^{-n} z) = T^n y + z = y + z$ and $T^n (x + \beta^{-n} z) = x + z$ by the choice of~$\varepsilon$. 
Therefore, we can find arbitrarily small $z$ such that (\ref{e:p0}) holds. 
For every such $z$, we can find arbitrarily large $k$ such that (\ref{e:p0}) holds since $x \in P$. 

Choose, as a first step, $z_1' \in \mathbb{Z}[\beta] \cap [0, \varepsilon)$ and $k_1 \ge 1$ such that
\begin{itemize}
\itemsep5pt
\item 
$T^{k_1} (y_1 + z_1') = T^{k_1} (x + z_1') = x$, which can be done by~(\ref{e:p0}),
\item 
$[y, y + \beta^{-1} z_1'] \subset X_{b_1(y)}$ for all $y \in P$, which can be done by choosing $z_1'$ sufficiently small,
\item 
$T^{k_1} (y_2 + z_1') \in P$, which can be done by choosing $k_1$ sufficiently large since the $T$-expansion of $y_2 + z_1'$ is eventually periodic by Theorem~\ref{t:frankrobinson}.
\end{itemize}
Now, suppose that, for $1 \le n < h $, we have $k_n \ge 0$ and $z_n' \in \mathbb{Z}[\beta] \cap [0, \varepsilon)$ that satisfy all of the following, where $z_0 = 0$, $z_n = \beta^{k_n} (z_{n-1} + z_n')$ and $s_n = k_1 + \cdots + k_n$.
\begin{itemize}
\itemsep5pt
\item[(i)] 
$T^{k_n} \big(T^{s_{n-1}} (y_n + \beta^{-s_{n-1}} z_{n-1}) + z_n'\big) = T^{k_n} (x + z_n') = x$,
\item[(ii)] 
$[y, y + \beta^{-s_n-1} z_n] \subset X_{b_1(y)}$ for all $y \in P$,
\item[(iii)] 
$T^{s_{n-1}} (y + \beta^{-s_n} z_n) = T^{s_{n-1}} (y + \beta^{-s_{n-1}} z_{n-1}) + z_n'$ for all $y \in P$, 
\item[(iv)] 
$T^{s_n} (y_{n+1} + \beta^{-s_n} z_n) \in P$.
\end{itemize}
Then by (iv) and (\ref{e:p0}), we have arbitrarily small $z_{n+1}' \in \mathbb{Z}[\beta] \cap [0,\varepsilon)$ and arbitrarily large $k_{n+1} \ge 0$ such that
\[ 
T^{k_{n+1}} \big(T^{s_n}(y_{n+1} + \beta^{-s_n} z_n) + z_{n+1}'\big) = T^{k_{n+1}} (x + z_{n+1}') = x. 
\]
If we set $z_{n+1} = \beta^{k_{n+1}} (z_n + z_{n+1}')$ and $s_{n+1} = s_n + k_{n+1}$, then by choosing $z_{n+1}$ small enough, and by (ii), we can get
\[ 
[y, y + \beta^{-s_{n+1}-1} z_{n+1}] = [y, y + \beta^{-s_n-1} z_n + \beta^{-s_n-1} z_{n+1}'] \subseteq X_{b_1(y)}
\]
for all $y \in P$. 
By choosing $z_{n+1}'$ small enough, we also have
\[ 
T^{s_n} (y + \beta^{-s_{n+1}} z_{n+1}) = T^{s_n} (y + \beta^{-s_n} z_n + \beta^{-s_n} z_{n+1}') = T^{s_n} (y + \beta^{-s_n} z_n) + z_{n+1}' 
\] 
for all $y \in P$.
If $n + 1 < h$, then, by Theorem~\ref{t:frankrobinson}, we can choose $k_{n+1}$ large enough such that 
\[ 
T^{s_{n+1}} (y_{n+2} + \beta^{-s_{n+1}} z_{n+1}) = T^{k_{n+1}} T^{s_n} \big(y_{n+2} + \beta^{-s_n} (z_n + z_{n+1}')\big) \in P. 
\]

For $n = 1$, we have already chosen $z_1' \in \mathbb{Z}[\beta] \cap [0,\varepsilon)$ and $k_1 \ge 1$ with the properties (i)--(iv).
Inductively, we obtain therefore some $z_h \in \mathbb{Z}[\beta] \cap [0,\infty)$ and $s_h \ge 1$ satisfying (i)--(iii).
Note that, by (ii) and (iii), $z = z_h$ satisfies all the conditions of Corollary~\ref{c:intiles} with $k = s_h$. 
Furthermore, by (iii) and (i), we have, for $1 \le n \le h$,
\begin{align*}
T^{s_h} (y_n + \beta^{-s_h} z_h) & = T^{k_h} T^{s_{h-1}} (y_n + \beta^{-s_h} z_h) = T^{k_h} \big(T^{s_{h-1}} (y_n + \beta^{-s_{h-1}} z_{h-1}) + z_h'\big) \\
& = T^{k_h} \big(\cdots \big(T^{k_{n+1}} \big(T^{k_n} \big(T^{s_{n-1}} (y_n + \beta^{-s_{n-1}} z_{n-1}) + z_n'\big) + z_{n+1}'\big) \cdots\big) + z_h'\big) \\
& = T^{k_h} (\cdots (T^{k_{n+1}} (x + z_{n+1}') \cdots) + z_h') = x.
\end{align*}
Similarly, we obtain
\[ 
T^{s_h} (x + \beta^{-s_h} z_h) = T^{k_h} \big(\cdots \big(T^{k_2} \big(T^{k_1} (x + z_1') + z_2'\big) \cdots\big) + z_h'\big) = x. 
\]
Hence, by Corollary~\ref{c:intiles}, $z_h$ is an exclusive point of $\mathcal{T}_{x}$.
\end{proof}

This leads us to define the following generalizations of the (F) and the (W) property, where $\varepsilon = \min_{y\in P} \big(\!\min\{z > y:\, z \not\in X_{b_1(y)}\} - y\big)\, \beta$ as in Proposition~\ref{p:W}.
Recall that $P$ is the set of points in $\mathbb{Z}[\beta] \cap X$ with purely periodic $T$-expansion.
\begin{align*}
\mbox{(F)}:\ &\ P\ \mbox{consists of only one element}. \\
\mbox{(W)}:\ &\ \exists\, x \in P:\, \forall\, y \in P\ \exists\, z \in \mathbb{Z}[\beta] \cap [0,\varepsilon),\, k \ge 0:\, T^k (y + z) = T^k (x + z) = x.
\end{align*}
Clearly, (F) implies (W).
We have the following corollary of Proposition~\ref{p:W}.

\begin{cor} \label{c:W}
Let $T$ satisfy the assumptions of Theorem~\ref{t:mtiling}.
Then the family $\{\mathcal{T}_x\}_{x\in\mathbb{Z}[\beta]\cap X}$ forms a tiling of $H$ if and only if (W) holds.
\end{cor}

\begin{remark}
For the greedy $\beta$-transformation, the condition (W) given here is similar to the property (W$'$) defined in~\cite{Akiyama02}.
Note that it is required in (W$'$) that there exists $z \in \mathbb{Z}[\beta] \cap [0,\varepsilon)$ and $k \ge 0$ satisfying (\ref{e:p0}) for every $\varepsilon > 0$.
Corollary~\ref{c:W} shows that it is sufficient to consider 
\[
\varepsilon = \min_{y\in P} \big(\min(\lfloor \beta y \rfloor + 1, \beta) - \beta y\big).
\]
\end{remark}

\subsection{Tiling of the torus} \label{s:torus}

Similarly to Ito and Rao (\cite{ItoRao06}), we can relate the tiling property of $\{\mathcal{T}_x\}_{x\in\mathbb{Z}[\beta]\cap X}$ to the tiling property of $\{\mathbf{x} + \widehat{Y}\}_{\mathbf{x}\in\mathbb{Z}^d}$, where $\widehat{Y}$ is the closure of the natural extension domain defined in Section~\ref{sec:geom-real-natur}.
Recall that $\bigcup_{\mathbf{x}\in\mathbb{Z}^d} (\mathbf{x} + \widehat{Y}) = \mathbb{R}^d$ by Lemma~\ref{l:positivemeasure}.

\begin{prop} \label{p:tilings}
Let $T$ satisfy the assumptions of Theorem~\ref{t:mtiling}.
Then the family $\{\mathbf{x} + \widehat{Y}\}_{\mathbf{x}\in\mathbb{Z}^d}$ forms a tiling of $\mathbb{R}^d$ if and only if the family $\{\mathcal{T}_x\}_{x\in\mathbb{Z}[\beta]\cap X}$ forms a tiling of~$H$.
\end{prop}

\begin{proof}
Assume first that $\{\mathcal{T}_x\}_{x\in\mathbb{Z}[\beta]\cap X}$, is not a tiling of~$H$. 
Then $\lambda^{d-1}(\mathcal{T}_x \cap \mathcal{T}_{x'}) > 0$ for some $x,\, x'\in \mathbb{Z}[\beta] \cap X$, $x \ne x'$, which implies
\[
\lambda^d\big(\big((J_x - x) \mathbf{v}_1 - \mathcal{T}_x\big) \cap \big((J_{x'} - x') \mathbf{v}_1 - \mathcal{T}_{x'}\big)\big) = \lambda^d\big(\big((J_x - x) \cap (J_{x'} - x')\big) \mathbf{v}_1 - (\mathcal{T}_x \cap \mathcal{T}_{x'})\big)> 0.
\]
Since $(J_x - x) \mathbf{v}_1 - \mathcal{T}_x = J_x \mathbf{v}_1 - \mathcal{D}_x - \Psi(x) \subseteq \widehat{Y} - \Psi(x)$ and 
$(J_{x'} - x') \mathbf{v}_1 - \mathcal{T}_{x'} \subseteq \widehat{Y} - \Psi(x')$, we obtain $\lambda^d((\widehat{Y} - \Psi(x)) \cap (\widehat{Y} - \Psi(x'))) > 0$, hence $\{\mathbf{x} + \widehat{Y}\}_{\mathbf{x}\in\mathbb{Z}^d}$ is not a tiling.

Assume now that $\{\mathcal{T}_x\}_{x\in\mathbb{Z}[\beta]\cap X}$ forms a tiling of~$H$.
Consider a hyperplane $y \mathbf{v}_1 + H$, $y \in \mathbb{Z}[\beta]$.
We have $(\widehat{X} - \Psi(x)) \cap (y \mathbf{v}_1 + H) \ne \emptyset$ for $x \in \mathbb{Z}[\beta]$ if and only if $x + y \in X$. Then
\[ 
\big(\widehat{X} - \Psi(x)\big) \cap (y \mathbf{v}_1 + H) = y \mathbf{v}_1 - \Phi(x) - \mathcal{D}_{x+y} = y \mathbf{v}_1 + \Phi(y) - \mathcal{T}_{x+y} = \Psi(y) - \mathcal{T}_{x+y}. 
\]
Since $\{\Psi(y) - \mathcal{T}_{x+y}\}_{x\in\mathbb{Z}[\beta]\cap(X-y)}$ forms a tiling of $y\mathbf{v}_1 + H$, we have 
\[
\lambda^{d-1}\big(\big(\widehat{X} - \Psi(x)\big) \cap \big(\widehat{X} - \Psi(x')\big) \cap \big(y \mathbf{v}_1 + H\big)\big) = 0\quad \mbox{for all}\ x,x' \in \mathbb{Z}[\beta],\, x \ne x',\, y \in \mathbb{Z}[\beta].
\]
Since $\mathbb{Z}[\beta]$ is dense in $\mathbb{R}$ and $\widehat{X}$ has the shape given in~(\ref{e:shape}), we obtain that $\lambda^d((\widehat{X} - \Psi(x)) \cap (\widehat{X} - \Psi(x'))) = 0$, thus $\lambda^d((\mathbf{x} + \widehat{Y}) \cap (\mathbf{x}' + \widehat{Y})) = 0$ for all $\mathbf{x}, \mathbf{x}' \in \mathbb{Z}^d$ with $\mathbf{x} \ne \mathbf{x}'$.
All other tiling conditions for $\{\mathbf{x} + \widehat{Y}\}_{\mathbf{x}\in\mathbb{Z}^d}$ follow from the compactness of $\widehat{Y}$, from Lemmas~\ref{l:positivemeasure} and~\ref{l:closureinterior}.
\end{proof}

Note that Lemma~\ref{l:positivemeasure} (for $q = 1$) also follows from the fact that $\bigcup_{\mathbf{x}\in\mathbb{Z}^d} (\mathbf{x} + \widehat{X})$ contains, for every $y \in \mathbb{Z}[\beta]$, $\bigcup_{x\in\mathbb{Z}[\beta]\cap(X-y)} (\Psi(y) - \mathcal{T}_{x+y})$ and thus $y \mathbf{v}_1 + H$ by Lemma~\ref{l:cover}.

\begin{remark}
The transformation $\widehat{T}$ induces a toral automorphism since $\widehat{T} \mathbf{x} \equiv M_{\beta} \mathbf{x} \pmod{\mathbb{Z}^d}$.
If furthermore $\bar{\mathcal{S}}$ is a shift of finite type and $\{\mathbf{x} + \widehat{Y}\}_{\mathbf{x}\in\mathbb{Z}^d}$ forms a tiling of $\mathbb{R}^d$, then $\{\widehat{X}_a\}_{a\in A}$ is a Markov partition of the torus $\mathbb{T}^d = \mathbb{R}^d/\mathbb{Z}^d$.
\end{remark}

The next lemma shows that, when $\{\mathbf{x} + \widehat{Y}\}_{\mathbf{x}\in\mathbb{Z}^d}$ forms a tiling of~$\mathbb{R}^d$, the partition $\{\widehat{X}_a\}_{a\in A}$ can be used to determine the $k$-th digit $b_k(x)$ in the $T$-expansion of $x$ up to an error term of order~$\rho^k$, with $\rho = \max_{2\le j\le d} |\beta_j| <1$.
This can be used e.g.\ to obtain distribution properties of the $k$-th digit on polynomial sequences, see~\cite{Steiner02}.

\begin{lemma} \label{l:kthdigit}
For every $x \in X$, $k \ge 1$, we have $\beta^{k-1} x \mathbf{v}_1 \in \widehat{X}_{b_k(x)} + M_{\beta}^{k-1} \mathcal{D}_x \pmod{\mathbb{Z}^d}$.
If $\mathbf{0} \in \mathcal{D}_x$, then we have $\beta^{k-1} x \mathbf{v}_1 \in \widehat{X}_{b_k(x)} \pmod{\mathbb{Z}^d}$.
\end{lemma}

\begin{proof}
By definition, we have $T^{k-1} x \in X_{b_k(x)}$. 
Since $T$ is surjective, there exists some $w \in {\vphantom{A}}^{\omega}\!A$ with $w \cdot b(x) \in \mathcal{S}$, and thus $\varphi(w) \in \mathcal{D}_x$.
Since $\psi(\sigma u) = \widehat{T} \psi(u) \equiv M_{\beta} \psi(u) \pmod{\mathbb{Z}^d}$, we obtain 
\[ 
\beta^{k-1} x \mathbf{v}_1 - M_{\beta}^{k-1} \varphi(w) \equiv M_{\beta}^{k-1} \psi\big(w \cdot b(x)\big) \equiv \psi(w b_1(x) \cdots b_{k-1}(x) \cdot b(T^{k-1}x)\big) \in \widehat{X}_{b_k(x)} 
\] 
modulo $\mathbb{Z}^d$.
If $\mathbf{0} \in \mathcal{D}_x$, then we can choose $w$ with $\varphi(w) = \mathbf{0}$.
\end{proof}

In certain cases, the partition $\{\widehat{X}_a\}_{a\in A}$ determines exactly every digit $b_k(x)$. 

\begin{theorem}
Let $T$ satisfy the assumptions of Theorem~\ref{t:mtiling} and assume that (W) holds.
Then we have, for every $x \in X$ such that $\mathbf{0}$ is an inner point of $\mathcal{D}_x$, and for every $k \ge 1$, 
\[ 
b_k(x) = a \quad\mbox{if and only if}\quad \beta^{k-1} x \mathbf{v}_1 \in \widehat{X}_a \pmod{\mathbb{Z}^d}. 
\]
\end{theorem}

\begin{proof}
By Lemma~\ref{l:kthdigit}, $b_k(x) = a$ implies $\beta^{k-1} x \mathbf{v}_1 \in \widehat{X}_a \pmod{\mathbb{Z}^d}$.
Assume that $\beta^{k-1} x \mathbf{v}_1 \in \widehat{X}_a \pmod{\mathbb{Z}^d}$ for some $a \ne b_k(x)$, i.e., that $\beta^{k-1} x \mathbf{v}_1 \in \widehat{X}_a \cap \widehat{X}_{b_k(x)} \pmod{\mathbb{Z}^d}$.
{}From the proof of Lemma~\ref{l:kthdigit}, we see that $\beta^{k-1} x \mathbf{v}_1 \equiv (T^{k-1} x) \mathbf{v}_1 - \varphi(b_1(x) \cdots b_{k-1}(x)) \pmod{\mathbb{Z}^d}$.
Furthermore, $\varphi(b_1(x) \cdots b_{k-1}(x))$ is an inner point of $\mathcal{D}_{T^{k-1}x}$ since $\mathcal{D}_{T^{k-1}x}$ contains $\varphi(b_1(x) \cdots b_{k-1}(x)) + M_{\beta}^{k-1} \mathcal{D}_x$ and $\mathbf{0}$ is an inner point of $\mathcal{D}_x$.
By Lemma~\ref{l:closureinterior} and since $\widehat{X}$ has the shape given in~(\ref{e:shape}), we obtain that $\widehat{X}_a \cap \widehat{X}_{b_k(x)} \pmod{\mathbb{Z}^d}$ has positive Lebesgue measure.
Since $\widehat{X}_a \cap \widehat{X}_{b_k(x)} = \emptyset$ in~$\mathbb{R}^d$, this implies that $\{\mathbf{x} + \widehat{Y}\}_{\mathbf{x}\in\mathbb{Z}^d}$ is not a tiling.
By Proposition~\ref{p:tilings} and Corollary~\ref{c:W}, this contradicts the (W) property.
\end{proof}

\subsection{Examples of (multiple) tilings} \label{sec:exampl-mult-tilings}

Consider now the multiple tilings defined by the transformations in the examples of Section~\ref{sec:exampl-natur-extens}.
In Figure~\ref{f:torus}, we see the natural extension domain $\widehat{X}$ and its translation by the vectors $(1,0)^t$, $(0,1)^t$ and $(1,1)^t$ for Examples~\ref{x:gm} and~\ref{x:gm2}.
The third picture in Figure~\ref{f:torus} shows $\widehat{X}$ and its translations by the vectors $(2,0)^t$, $(0,2)^t$ and $(2,2)^t$ for Example~\ref{x:gm4}.
It follows that $\mathcal{T}$ is a tiling for Examples~\ref{x:gm} and~\ref{x:gm2}, while the covering degree is~$4$ for Example~\ref{x:gm4}.

\begin{figure}[ht]
\centering
\includegraphics{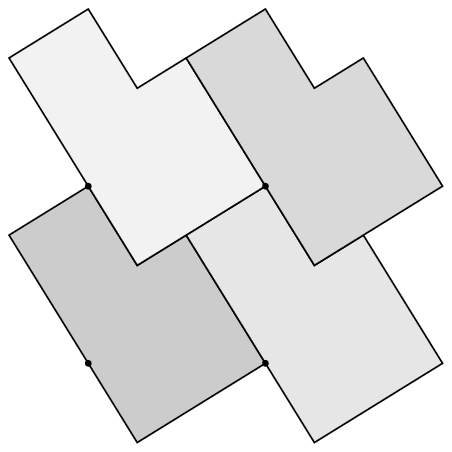}
\quad
\includegraphics{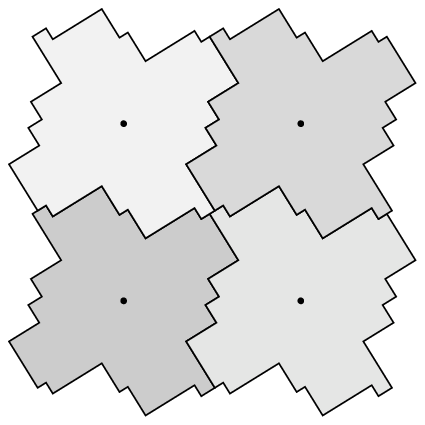}
\quad
\includegraphics{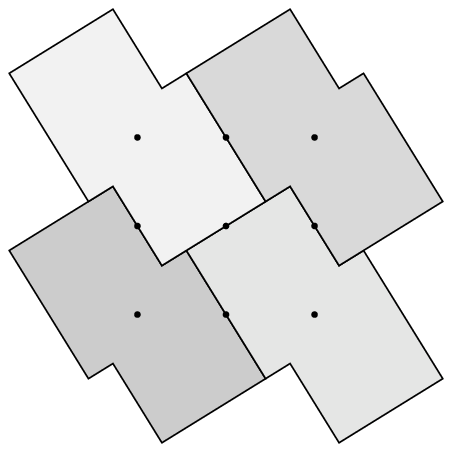}
\caption{Translations of the natural extension domains of Examples~\ref{x:gm} and~\ref{x:gm2} by integer vectors, and of Example~\ref{x:gm4} by vectors in $2 \mathbb{Z}^2$.}
\label{f:torus}
\end{figure}

Indeed, $P = \{0\}$ and thus (F) holds in Examples~\ref{x:gm} and~\ref{x:gm2}.
For Example~\ref{x:gm}, this was first proved in \cite{FrougnySolomyak92}.
For Example~\ref{x:gm2}, this follows from the fact that, for any $x \in \mathbb{Z}[\beta] \cap X$, every $\beta$-expansion of minimal weight is finite, thus $T^n x = 0$ for some $n \ge 0$.
In the following parapraph, we give a short direct proof for $P = \{0\}$ in these two examples.

For general transformations with $\beta = (1+\sqrt{5})/2$ and $A \subseteq \{-1,0,1\}$, note that $|\beta_2| = 1/\beta$, thus $\varphi(w) \in \big[\frac{-1}{1-1/\beta}, \frac{1}{1-1/\beta}\big] \mathbf{v}_2$ for any $w \in {\vphantom{A}}^{\omega}\!A$.
This gives $\mathcal{D}_x \subseteq \big[\frac{-1}{1-1/\beta}, \frac{1}{1-1/\beta}\big] \mathbf{v}_2$ for any $x \in X$.
By Lemma~\ref{l:P}, we obtain $|\Gamma_2(y)| \le \frac{1}{1-1/\beta} = \beta^2$ for any $y \in P$.
The only $x \in \mathbb{Z}[\beta]$ with $|x| < 1$, $|\Gamma_2(x)| \le \beta^2$, are $x = 0$, $x = \pm (\beta - 1) = \pm 1/\beta$, $x = \pm (\beta - 2) = \pm 1/\beta^2$. 
In Examples~\ref{x:gm} and~\ref{x:gm2}, we have $T^2 (1/\beta^2) = T (1/\beta) = 0$, $T^2 (-1/\beta^2) = T (-1/\beta) = 0$, thus $P = \{0\}$, and $\mathcal{T}$ is a tiling.

In Example~\ref{x:gm4}, we have $P = \{-1, -1/\beta, -1/\beta^2, 0, 1/\beta^2\}$.
Using Corollary~\ref{c:intiles}, it is easy to find a point lying in exactly $4$ tiles, e.g.~$\Phi(1)$.
Hence the covering degree is at most~$4$. 
It is slightly harder to show that some set with non-empty interior is covered $4$ times.
Recall that $\mathcal{V} = \{-1, -1/\beta, 1/\beta\}$.
By~(\ref{e:decomposition}), we have the decompositions
\[ 
\mathcal{D}_{-1} = -\frac{\mathcal{D}_0}{\beta} + \mathbf{v}_2,\quad \mathcal{D}_{-1/\beta} = \Big(-\frac{\mathcal{D}_{-1}}{\beta} - \mathbf{v}_2\Big) + \Big(-\frac{\mathcal{D}_{1/\beta^3}}{\beta} + \mathbf{v}_2\Big),\quad \mathcal{D}_{1/\beta} = -\frac{\mathcal{D}_{-1/\beta^3}}{\beta} - \mathbf{v}_2,
\]
The equalities $\mathcal{D}_0 = \mathcal{D}_{1/\beta^3} = \mathcal{D}_{-1/\beta^3} = \mathcal{D}_{-1/\beta}$ provide a GIFS for the sets $\mathcal{D}_{-1}$, $\mathcal{D}_{-1/\beta}$, $\mathcal{D}_{1/\beta}$, cf.~(\ref{e:GIFS}), with the unique solution of non-empty compact sets
\[ 
\mathcal{D}_{-1} = [-1/\beta, \beta^2] \mathbf{v}_2,\quad \mathcal{D}_{-1/\beta} = [-\beta^2, \beta^2] \mathbf{v}_2,\quad \mathcal{D}_{1/\beta} = [-\beta^2, 1/\beta] \mathbf{v}_2. 
\]
Therefore, $[0,\beta] \mathbf{v}_2$ is covered $4$ times by $\mathcal{T}_0 = [-\beta^2, \beta^2] \mathbf{v}_2$, $\mathcal{T}_{-1} = [-\beta, \beta] \mathbf{v}_2$, $\mathcal{T}_{-1/\beta} = [-1, \beta^3] \mathbf{v}_2$ and $\mathcal{T}_{1/\beta^2} = [0,2\beta^2] \mathbf{v}_2$.
Note that an important difference between this transformation and all other examples we are considering is that we do not have $\beta^{-1} x \in T^{-1} \{x\}$ here.

\medskip
We can generalize Example~\ref{x:gm2} in the following way.
Let $\beta$ be the golden ratio and take any $\alpha \in \big(\frac{\beta}{\beta^2+1}, \frac{1}{2}\big]$. 
It was shown in~\cite{FrougnySteiner08} that such a pair $(\beta, \alpha)$ gives a minimal weight transformation. 
Let $T$ be this transformation, i.e., set $A = \{-1,0,1\}$, $X_{-1} = [-\beta \alpha, -\alpha)$, $X_0 = [-\alpha, \alpha)$, $X_1 = [\alpha, \beta \alpha)$.
One can show that $\widetilde{T}^5 \alpha = T^5 \alpha$,  $\tilde{b}_1(\alpha) \cdots \tilde{b}_5(\alpha) = 01001$, $b_1(\alpha) \cdots b_5(\alpha) = 100(-1)0$, as in Example~\ref{x:gm2}. 
By symmetry, we also have $T^5 (-\alpha) = \widetilde{T}^5 (-\alpha)$, hence $\mathcal{V}$ is finite.
Furthermore, we have $T^2 (1/\beta^2) = T (1/\beta) = 0$, $T^2 (-1/\beta^2) = T (-1/\beta) = 0$, thus $P = \{0\}$, and $\mathcal{T}$ is a tiling for any choice of $\alpha \in \big(\frac{\beta}{\beta^2+1}, \frac{1}{2}\big]$.

If $\alpha \not\in \mathbb{Q}(\beta)$, then $b(\alpha)$ is aperiodic by Theorem~\ref{t:frankrobinson}.
By Proposition~\ref{p:sofic}, this implies that $\bar{\mathcal{S}}$ is not sofic, hence there are many possibilities to obtain tilings from non-sofic shifts.
Moreover, $\mathcal{T}$~is not a self-affine tiling in the sense of \cite{Praggastis99,Solomyak97} if $\alpha \not\in \mathbb{Q}(\beta)$, because there exists no coloring $c(\mathcal{T})$ with finitely many colors $c(\mathcal{T}_x)$, $x \in X$, and the following two properties if $c(\mathcal{T}_x) = c(\mathcal{T}_y)$:
\begin{itemize}
\item
$\mathcal{T}_x$ is a translate of $\mathcal{T}_y$, i.e., $\mathcal{D}_x = \mathcal{D}_y$,
\item
the colors of the tiles constituting $M_\beta^{-1} \mathcal{T}_x$ are equal to the colors in $M_\beta^{-1} \mathcal{T}_y$.
\end{itemize}
Indeed, there exists some $x \in \mathcal{V}$ such that $\{T^k \alpha:\, k \ge 1\} \cap J_x$ is infinite.
Take $y, z \in J_x$ such that $y < T^k(\alpha) \le z$ for some $k \ge 1$, and let $k$ be minimal with this property. 
Then the subdivision of $\mathcal{D}_y$ given by iterating (\ref{e:decomposition}) $k$ times is different from that of $\mathcal{D}_z$, which implies that the number of colors must be infinite.

\subsection{Symmetric $\beta$-transformations} \label{sec:symm-beta-transf}

Recall that the symmetric $\beta$-transformation is defined in~\cite{AkiyamaScheicher07} by $T x = \beta x - \lfloor \beta x + 1/2 \rfloor$ on $X = [-1/2, 1/2)$.
If $\beta \le 3$, then this means in our setting that $A = \{-1,0,1\}$, $X_{-1} = \big[-\frac{1}{2}, -\frac{1}{2\beta}\big)$, $X_0 = \big[-\frac{1}{2\beta}, \frac{1}{2\beta}\big)$ $X_1 = \big[\frac{1}{2\beta}, \frac{1}{2}\big)$.   

The case $\beta < 2$ plays a special role. 
In this case, we have $T^{-1} \{0\} = \{0\}$, thus $\mathcal{D}_0 = \{\mathbf{0}\}$. 
Indeed, the support $X'$ of the invariant measure given by the natural extension in Section~\ref{sec:geom-real-natur} is contained in $\big[-\frac{1}{2}, \frac{\beta}{2} - 1\big) \cup \big[1 - \frac{\beta}{2}, \frac{1}{2}\big)$.
Recall that it is natural to restrict $T$ to~$X'$ when we consider (multiple) tilings.
Then $0 \not\in P$. Since $P$ is non-empty, we have by the symmetry of the transformation that if $x \in P$, then also $-x \in P$, which implies that (F) cannot hold. (The fact that $T$ is not symmetric on the endpoints of the intervals plays no role since the endpoints are not in $\mathbb{Z}[\beta]$.)

\subsubsection{Quadratic Pisot units}
If $\beta$ is a quadratic Pisot unit, then Theorem~3.8 in~\cite{AkiyamaScheicher07} shows that (F) holds for the symmetric $\beta$-transformation if and only if $\beta > 2$.

The only quadratic Pisot unit with $\beta < 2$ is the golden ratio $\beta = (1+\sqrt{5})/2$.
The transformation and its natural extension is given in Example~\ref{x:gm3} and Figure~\ref{f:xhat3}, and the translations of~$\widehat{X}$ by integer vectors are depicted in Figure~\ref{f:torus2}.
By the considerations in Section~\ref{sec:exampl-mult-tilings}, we obtain $P = \{\pm 1/\beta^2\}$.
Here, (W) holds since $z = \frac{1}{\beta^5} < \varepsilon = \min \big\{\big(\frac{1}{2} - \frac{1}{\beta^2}\big) \beta, \big(-\frac{1}{2\beta} + \frac{1}{\beta^2}\big) \beta\} = \frac{1}{2\beta^3}$ gives
\begin{align*}
& T^3 (1/\beta^2 + z) = T^3 (2/\beta^3) = T^2 (-1/\beta^3) = T (-1/\beta^2) = 1/\beta^2, \\
& T^3 (-1/\beta^2 + z) = T^3 (-2/\beta^4) = T^2 (-2/\beta^3) = T (1/\beta^3) = 1/\beta^2.
\end{align*}
Hence, Corollary~\ref{c:W} implies that the symmetric $\beta$-transformation with a quadratic Pisot unit $\beta$ always gives a tiling.

\begin{figure}[ht]
\centering
\includegraphics{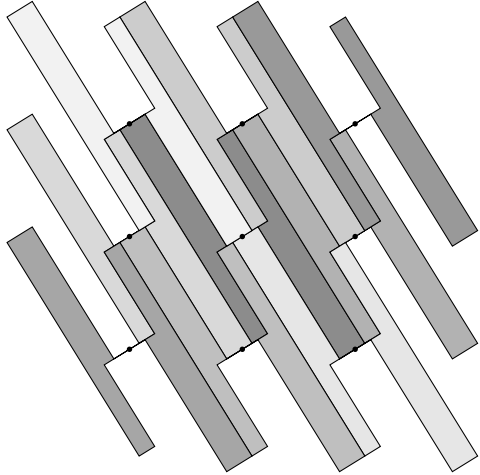}
\caption{The natural extension domain and its translations by integer vectors for the symmetric $\beta$-transformation with $\beta = (1+\sqrt{5})/2$.}
\label{f:torus2}
\end{figure}

\subsubsection{Cubic Pisot units} \label{sec:cubic-pisot-units}
By Lemma~1 in~\cite{Akiyama00}, $\beta > 1$ is a cubic Pisot unit if and only if it is the root of an irreducible polynomial $x^3 - c_1 x^2 - c_2 x - c_3 \in \mathbb{Z}[x]$ with $|c_3| = 1$ and $|c_2 - 1| < c_1 + c_3$.
By Theorem~3.8 in~\cite{AkiyamaScheicher07}, a cubic Pisot unit satisfies (F) with respect to the symmetric $\beta$-transformation if and only if 
\[ 
\beta > 2 \quad\mbox{and}\quad |\beta - c_1| < \frac{c_3}{\beta} + \frac{1}{2}. 
\]

In the rest of the paper, we consider the $4$ cubic Pisot units $\beta < 2$. 
They are given by the polynomials $x^3 - x - 1$ (smallest Pisot number), $x^3 - x^2 - 1$, $x^3 - 2 x^2 + x - 1$ (square of the smallest Pisot number), and $x^3 - x^2 - x - 1$ (Tribonacci number).

\smallskip
Let first $\beta$ be the smallest Pisot number, i.e., $\beta^3 = \beta + 1$.
We will show that $\mathcal{T}$ is a double tiling.
Here, the support of the invariant measure $\mu$ is
\[ X' = \Big[-\frac{1}{2}, \frac{\beta^2}{2} - \beta\Big) \cup \Big[\beta^2 - \frac{\beta}{2} - \frac{3}{2}, \frac{\beta}{2} - 1\Big) \cup \Big[1 - \frac{\beta}{2}, \frac{3}{2} + \frac{\beta}{2} - \beta^2\Big) \cup \Big[\beta - \frac{\beta^2}{2}, \frac{1}{2}\Big). \]
The set $P$ consists of the $8$ points in the orbit of $x = 3 - 2 \beta$, with $b(x) = (01\bar{1}10\bar{1}1\bar{1})^{\omega}$, where we write again $\bar{1}$ instead of~$-1$.
For
\[ 
Y = \Big[\beta^2 - \frac{\beta}{2} - \frac{3}{2}, \frac{\beta}{2} - 1\Big) \cup \Big[1 - \frac{\beta}{2}, \frac{3}{2} + \frac{\beta}{2} - \beta^2\Big),
\]
we have $T^2 x \in Y$ for every $x \in Y$.
Therefore, we consider the induced transformation $T_Y$ on~$Y$.
Since $T_Y x = T^2 x = \beta^2 x - \beta b_1(x) - b_2(x)$, $T_Y$~is a right-continuous $\beta^2$-transformation with $A_Y = \{-\beta+1, -1, 1, \beta-1\} = \{-1/\beta^4, -1, 1, 1/\beta^4\}$, $Y_{-1/\beta^4} = \big[\beta^2 - \frac{\beta}{2} - \frac{3}{2}, -\frac{1}{2\beta}\big)$, $Y_{-1} = \big[-\frac{1}{2\beta}, \frac{\beta}{2} - 1\big)$, $Y_1 = \big[1 - \frac{\beta}{2}, \frac{1}{2\beta}\big)$, $Y_{1/\beta^4} = \big[\frac{1}{2\beta}, \frac{3}{2} + \frac{\beta}{2} - \beta^2\big)$.
We have $T^{-1} Y \cap X' = X' \setminus Y$ and $T^{-2} Y \cap X' = Y$.
Therefore, for every $x \in \mathbb{Z}[\beta] \cap Y$, the tile $\mathcal{T}_x$ defined by $T$ is equal to the tile $\mathcal{T}_x^Y$ defined by~$T_Y$.
This implies 
\[ 
\bigcup_{x\in\mathbb{Z}[\beta]\cap Y} \mathcal{T}_x = \bigcup_{x\in\mathbb{Z}[\beta]\cap Y} \mathcal{T}_x^Y = H. 
\]
Since every $x \in Y$ has a unique preimage $T^{-1} x \in X' \setminus Y$, we obtain $\mathcal{T}_x = M_\beta \mathcal{T}_{T^{-1} x}$.
Hence, we also have $\bigcup_{x\in\mathbb{Z}[\beta]\cap(X'\setminus Y)} \mathcal{T}_x = H$, and $\mathcal{T}$ is a multiple tiling of degree at least~$2$.
By Corollary~\ref{c:intiles}, $\Phi(2)$ lies in $\mathcal{T}_{3-2\beta}$ and $\mathcal{T}_{2\beta^2-3\beta}$.
Thus, $\mathcal{T}$ is a double tiling and $\{\mathcal{T}_x^Y\}_{x\in\mathbb{Z}[\beta]\cap Y}$ a tiling, see Figure~\ref{f:half011}.

\begin{figure}[ht]
\centering
\includegraphics[width=.48\textwidth]{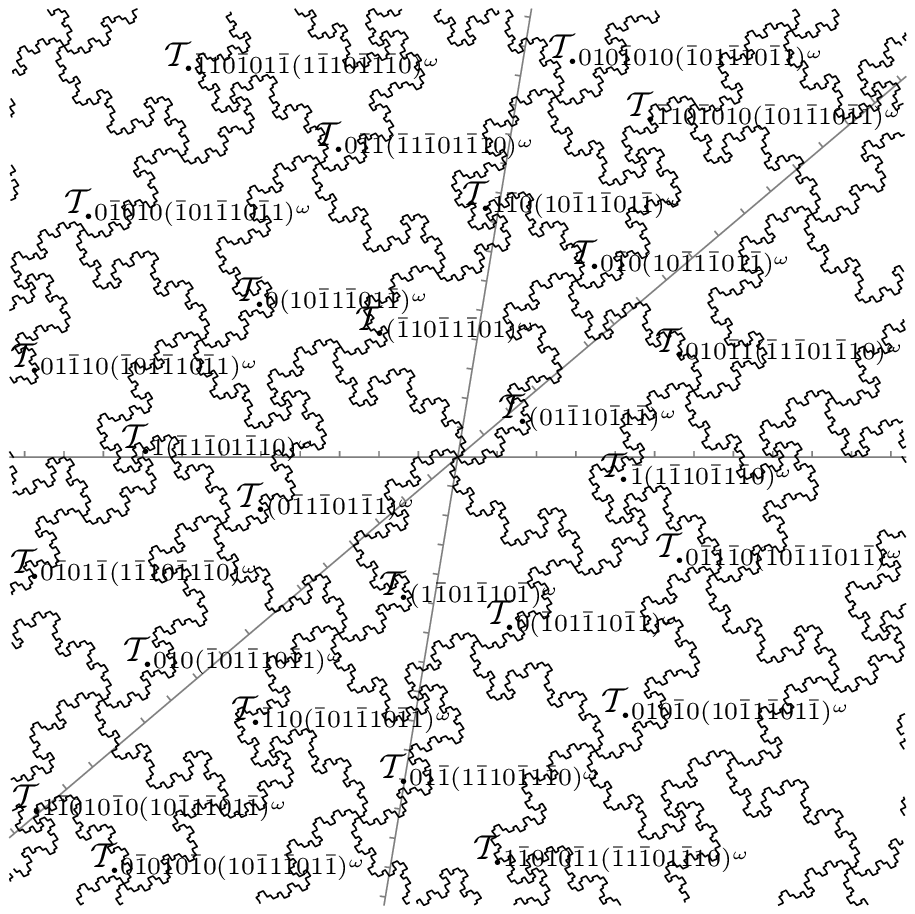}
\quad
\includegraphics[width=.48\textwidth]{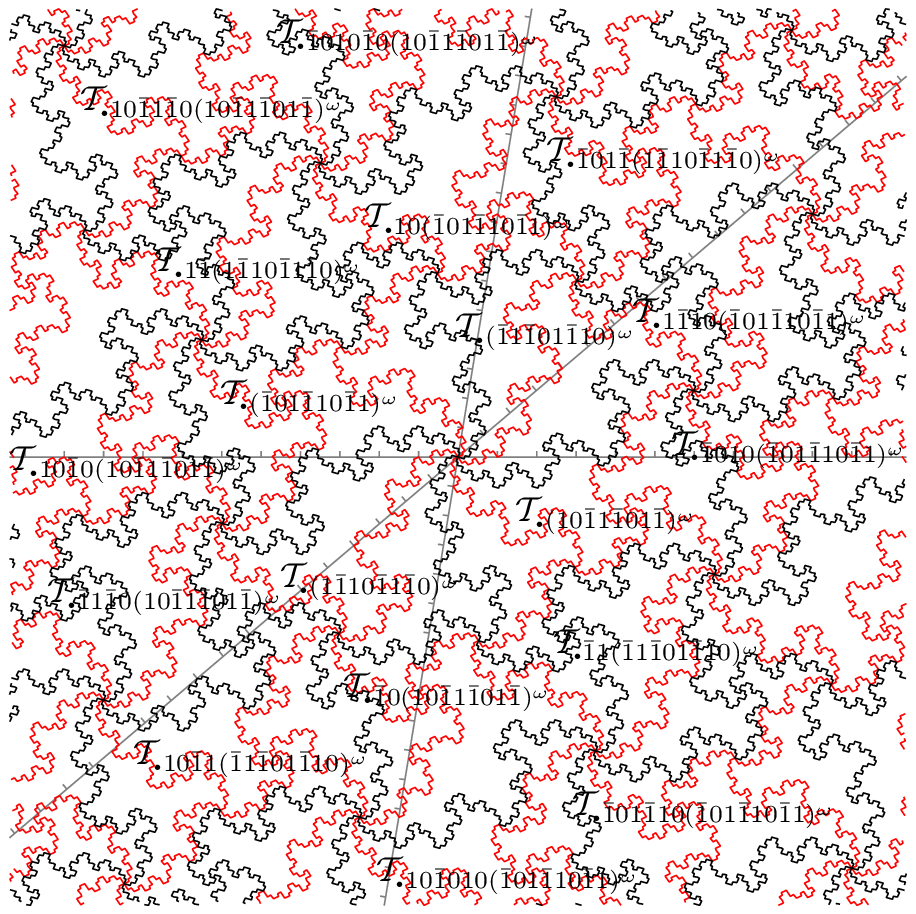}
\caption{The tiling $\{\mathcal{T}_x^Y\}_{x\in\mathbb{Z}[\beta]\cap Y} = \{\mathcal{T}_x\}_{x\in\mathbb{Z}[\beta]\cap Y}$ and the double tiling $\{\mathcal{T}_x\}_{x\in\mathbb{Z}[\beta]\cap X'}$ from the symmetric $\beta$-transformation, $\beta^3 = \beta + 1$.
The tics on the gray lines represent integer multiples of $\Phi(1)$, $\Phi(\beta)$ and $\Phi(\beta^2)$, respectively.}
\label{f:half011}
\end{figure}

\begin{remark}
It is conjectured that the greedy $\beta$-transformation $T_{\beta} x = \beta x - \lfloor \beta x \rfloor$ on $X = [0, 1)$ produces a tiling for every Pisot unit~$\beta$.
This is a version of the \emph{Pisot conjecture}, and it is known to be true for all cubic Pisot units (\cite{AkiyamaRaoSteiner04}). 
It is therefore quite surprising that this conjecture does not hold when we shift $X$ from $[0, 1)$ to $[-1/2, 1/2)$ and set $T x = \beta x - \lfloor \beta x + 1/2 \rfloor$.
\end{remark}

\begin{remark} \label{r:zerotiles}
If we do not restrict $T$ to the support of the invariant measure, then $\{0, \pm 1/\beta^3\}$ are purely periodic points in $\mathbb{Z}[\beta] \cap X$ for the symmetric $\beta$-transformation with the smallest Pisot number. 
We have $\mathcal{D}_x = \{0\}$ for every $x \in \big[\frac{\beta}{2}-1, 1-\frac{\beta}{2}\big)$, while 
\[
\{w \in {\vphantom{A}}^{\omega}\!A:\, w \cdot b(x) \in \mathcal{S}\} = \bigcup_{k\ge0} \{{\vphantom{0}}^{\omega}0 (1\bar{1})^k,\, {\vphantom{0}}^{\omega}0\bar{1} (1\bar{1})^k\} \cup \{{\vphantom{(}}^{\omega} (1\bar{1})\}\quad \mbox{for all}\ x \in \big[\textstyle{\frac{3}{2} + \frac{\beta}{2} - \beta^2, \beta - \frac{\beta^2}{2}}\big),
\]
hence $\mathcal{D}_x$ consists of countably many points for these~$x$. 
For $x \in \big[\frac{\beta^2}{2} - \beta, \beta^2  - \frac{\beta}{2} - \frac{3}{2}\big)$, $\mathcal{D}_x$~is given by symmetry. 
\end{remark}

\medskip
If $\beta$ is the square of the smallest Pisot number, i.e., $\beta^3 = 2 \beta^2 - \beta + 1$, then the support of the invariant measure is $\big[-\frac{1}{2}, \frac{\beta}{2} - 1\big) \cup \big[1 - \frac{\beta}{2}, \frac{1}{2}\big)$ and $P = \{\pm (2 - \beta), \pm (2 \beta - \beta^2)\}$, with $b(2-\beta) = (010\bar{1})^{\omega}$.
By Corollary~\ref{c:intiles}, $\Phi(4)$ is an exclusive point of $\mathcal{T}_{3+2\beta-2\beta^2}$, see Figure~\ref{f:halftilings}.
Note that $M_\beta^{-3} \mathcal{T}$ seems to be equal to the tiling $\{\mathcal{T}_x^Y\}_{x\in\mathbb{Z}[\beta]\cap Y}$ defined above, but we do not prove this relation in this paper.

\begin{figure}[ht]
\centering
\includegraphics[width=.48\textwidth]{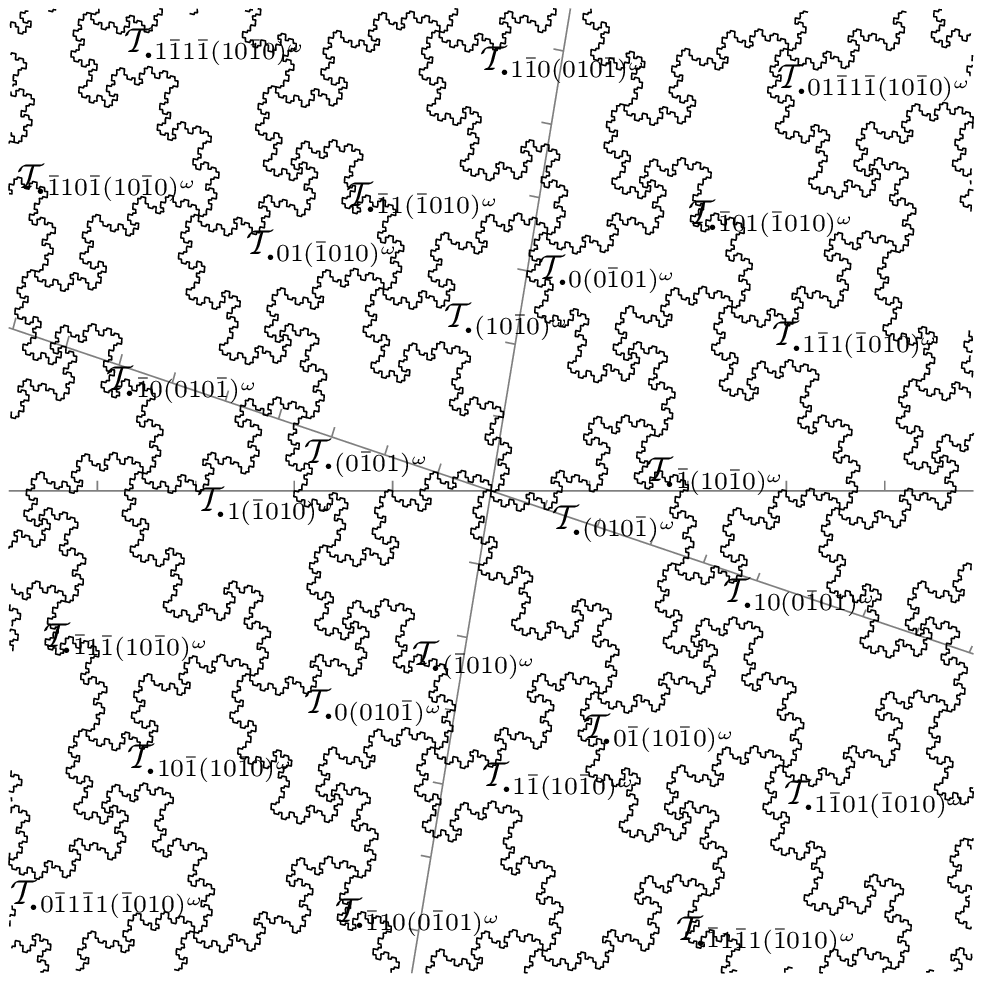}
\quad
\includegraphics[width=.48\textwidth]{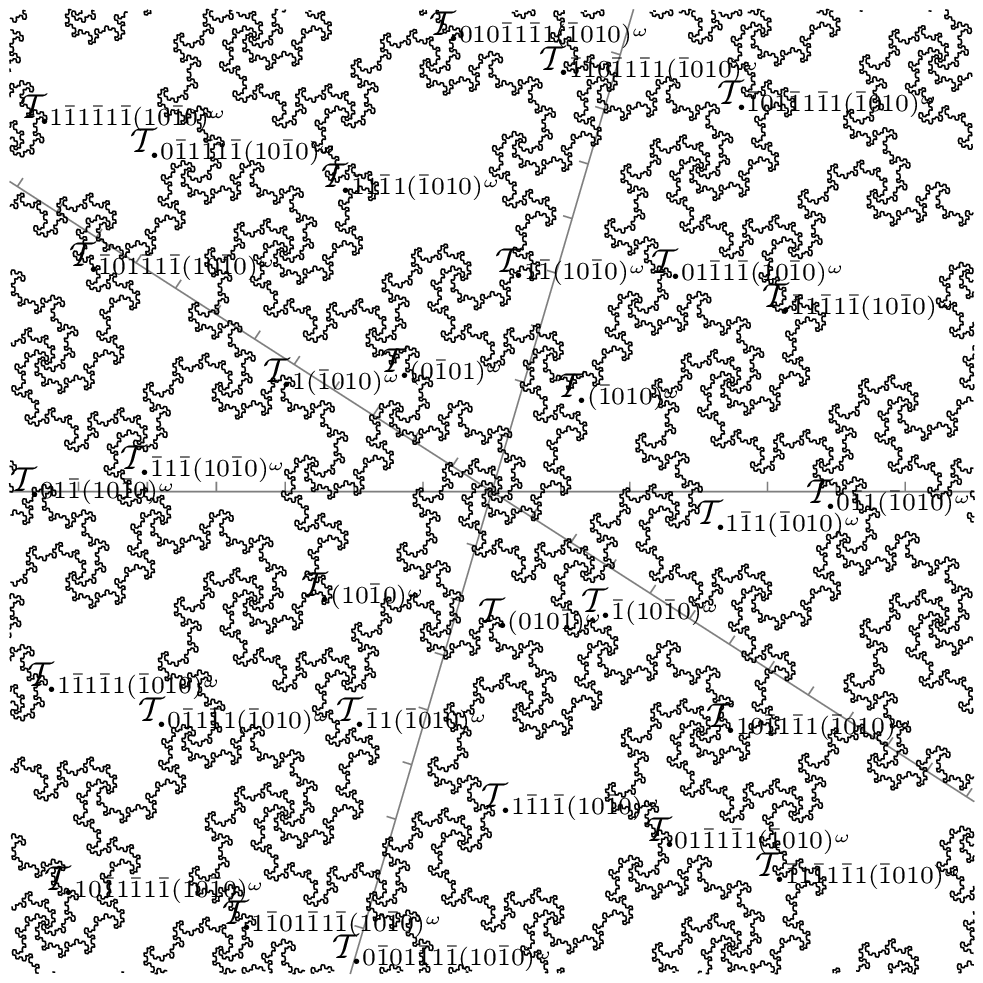}
\caption{Tilings from the symmetric $\beta$-transformation, $\beta^3 = 2\beta^2 - \beta + 1$ (left) and $\beta^3 = \beta^2 + 1$ (right).}
\label{f:halftilings}
\end{figure}

\medskip
The case $\beta^3 = \beta^2 + 1$ is very similar to the preceding one. 
We have $P = \{\pm 1/\beta^2, \pm 1/\beta^3\}$, $b(1/\beta^3) = (010\bar{1})^{\omega}$, and $\Phi(4)$ is an exclusive point of $\mathcal{T}_{4-\beta-\beta^2}$, see Figure~\ref{f:halftilings}.

\medskip
Finally, let $\beta$ be the Tribonacci number, i.e., $\beta^3 = \beta^2 + \beta + 1$. 
The support of the invariant measure is $X' = \big[-\frac{1}{2}, -\frac{1}{2\beta^3}\big) \cup \big[\frac{1}{2\beta^3}, \frac{1}{2}\big)$ and $P = \big\{\pm 1/\beta^3, \pm 1/\beta^2, \pm (1 - 1/\beta)\big\}$, with $b(1/\beta^3) = (01\bar{1})^{\omega}$, $b(-1/\beta^3) = (0\bar{1}1)^{\omega}$.
The degree of the multiple tiling is at most~$2$ since $\Phi(4)$ lies in the tiles $\mathcal{T}_{3-\beta^2}$ and $\mathcal{T}_{4-2\beta^2}$, see Figure~\ref{f:counterex}.

\begin{figure}[ht]
\centering
\includegraphics{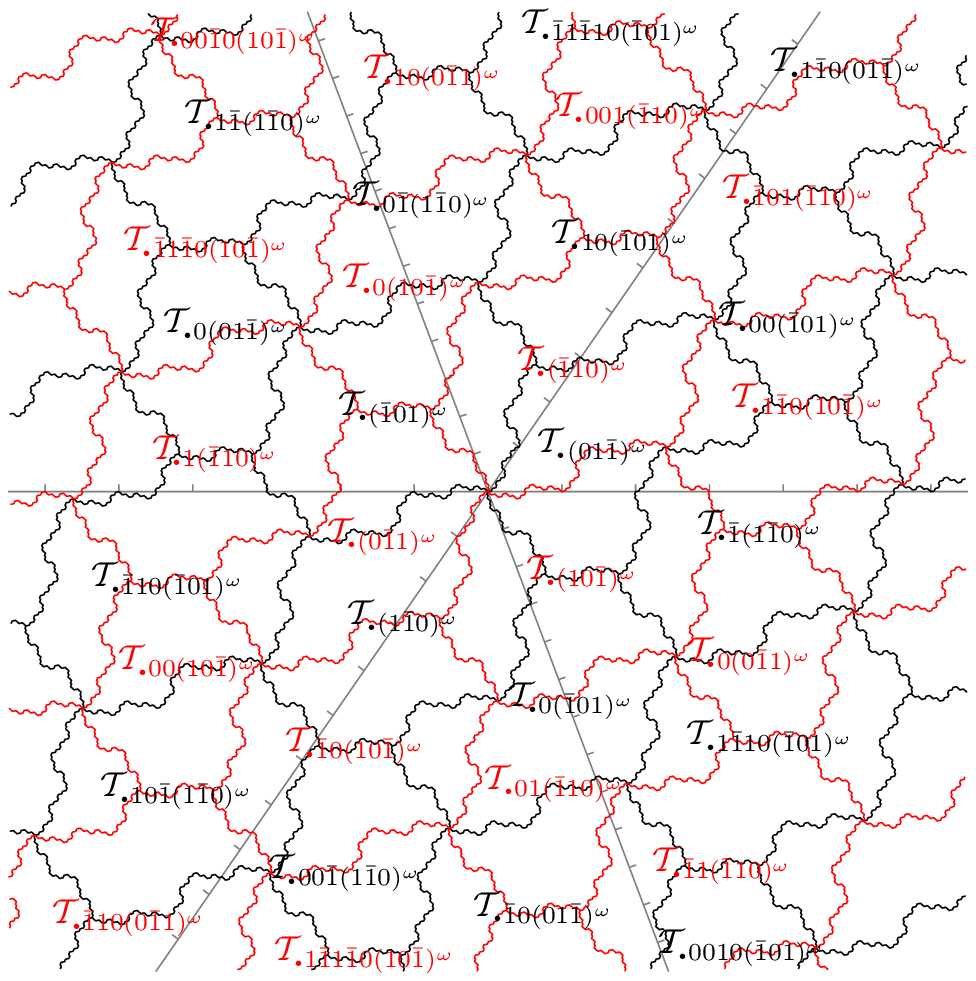}
\caption{Double tiling from the symmetric $\beta$-transformation, $\beta^3=\beta^2+\beta+1$.}
\label{f:counterex}
\end{figure}

Proving that $\mathcal{T}$ is not a tiling is more complicated than for the smallest Pisot number and for Example~\ref{x:gm4}. 
We use the description of $T$-admissible sequences in Theorem~\ref{t:admissible} and the transducer in Figure~\ref{fig:double} to show that the subtile $M_{\beta}^5 \mathcal{T}_{\decdot\bar{1}0(10\bar{1})^{\omega}}$ of $\mathcal{T}_{\decdot(10\bar{1})^{\omega}}$ is also contained in $\mathcal{T}_{\decdot(01\bar{1})^{\omega}}$.

By Remark~\ref{r:fullintervals}, $u = u_1 u_2 \cdots \in A^\omega$ is $T$-admissible if and only if 
\[
b(-1/2) = (\bar{1}001)^{\omega} \preceq u_k u_{k+1} \cdots \prec (100\bar{1})^{\omega} = \widetilde{b}(1/2)\quad \mbox{for all}\ k \ge 1.
\]
Therefore, $\bar{\mathcal{S}}$ is a shift of finite type with forbidden sequences $11,101,1000,1001,\bar{1}\bar{1},\bar{1}0\bar{1},\bar{1}000,\bar{1}00\bar{1}$.
The set $\mathcal{S}$ is obtained from $\bar{\mathcal{S}}$ by excluding the sequences ending with $(100\bar{1})^{\omega}$. 
Restricting $T$ to $X'$ means that we have to exclude 
\[
b\big({\textstyle \frac{-1}{2\beta^3}}\big) = 000(\bar{1}001)^{\omega} \preceq u_k u_{k+1} \cdots \prec 000(100\bar{1})^{\omega} = \tilde{b}\big({\textstyle \frac{1}{2\beta^3}}\big)\quad \mbox{for all}\ k \ge 1,
\]
i.e., the sequence $000$ is forbidden as well.

\begin{figure}[ht]
\includegraphics[scale=.8]{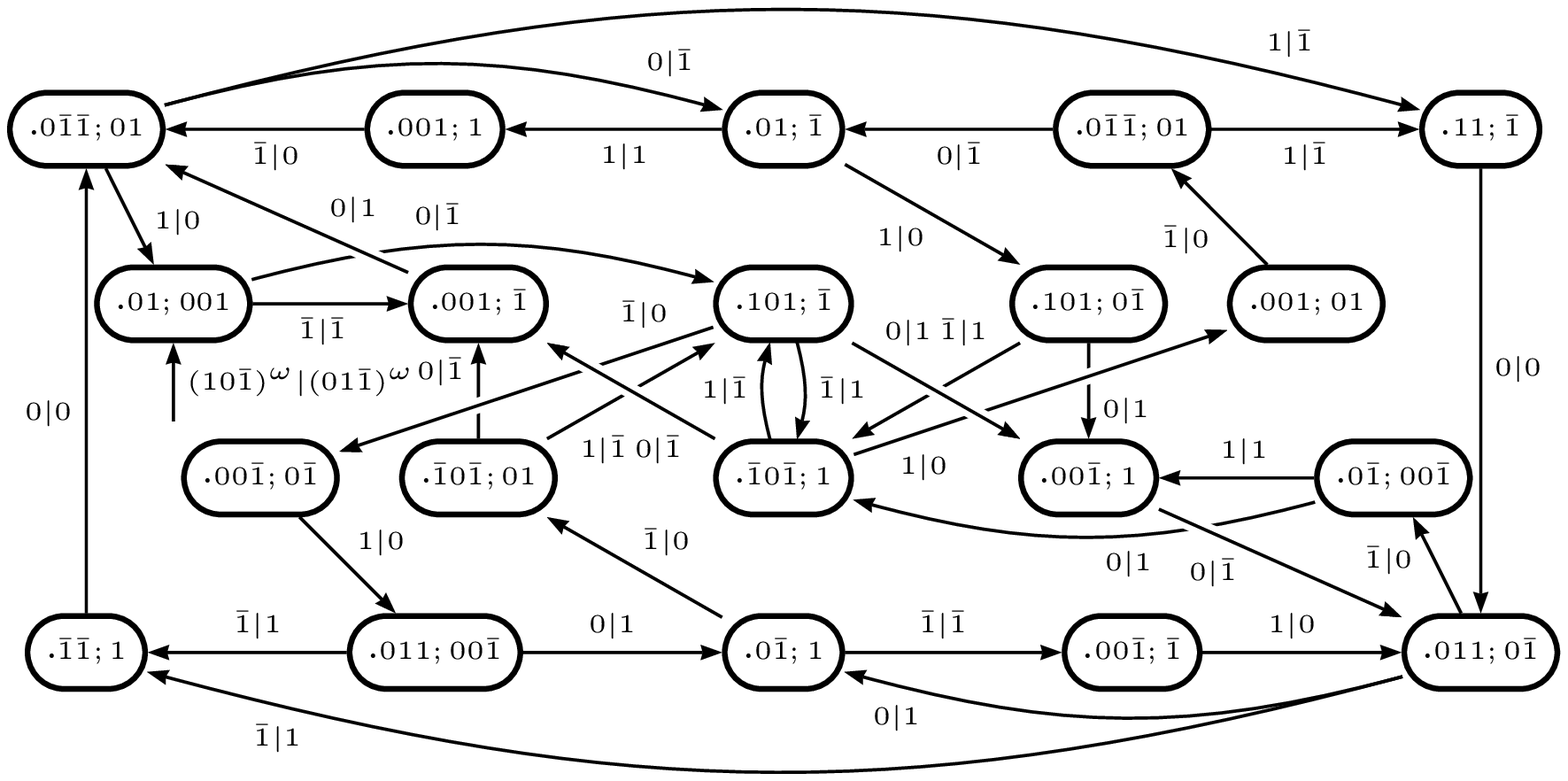}
\caption{Transducer showing that $\mathcal{T}$ is a double tiling, $\beta^3 = \beta^2 + \beta + 1$.}
\label{fig:double}
\end{figure}

The nodes of the transducer in Figure~\ref{fig:double} have labels $(x;v)$, where $x = \decdot u$ stands for the real number $\decdot u 0^{\omega}$ and $v$ is given by the outputs of the incoming paths. 
E.g., in $(\decdot 01; 001)$ we have $x = 1/\beta^2$, and there exists an incoming path with edges $\stackrel{1|0}{\longleftarrow} \stackrel{0|0}{\longleftarrow} \stackrel{\bar{1}|1}{\longleftarrow}$, which implies that the output of the outgoing transition can only be~$\bar{1}$.
(The incoming path $\stackrel{1|0}{\longleftarrow} \stackrel{0|1}{\longleftarrow}$ imposes less restrictions on the following output.) 
For every transition $(x; v) \stackrel{a|a'}{\longrightarrow} (x'; v')$, we have $x' = (x + a - a') / \beta$.
Consider a sequence of transitions starting in $(\decdot 01; 001)$,
\[ (\decdot 01; 001) = (x_0; v_0) \stackrel{w_0|w_0'}{\longrightarrow} (x_1; v_1) \stackrel{w_{-1}|w_{-1}'}{\longrightarrow} (x_2; v_2) \stackrel{w_{-2}|w_{-2}'}{\longrightarrow} \cdots \]
Since $\decdot (10\bar{1})^{\omega} = 1 - 1/\beta$ and $\decdot (01\bar{1})^{\omega} = 1/\beta^3$, we have $\decdot (10\bar{1})^{\omega} - \decdot(01\bar{1})^{\omega} = 1/\beta^2 = x_0$.
By the construction of the transducer, we obtain $x_{k+1} = \decdot w_{-k} \cdots w_0 (10\bar{1})^{\omega} - \decdot w_{-k}' \cdots w_0' (01\bar{1})^{\omega}$ for every $k \ge 0$.
This means that 
\[ 
\varphi(w_{-k} \cdots w_0) + \Phi(\decdot (10\bar{1})^{\omega}) = \varphi(w_{-k}' \cdots w_0') + \Phi(\decdot (01\bar{1})^{\omega}) + M_{\beta}^{k+1} \Phi(x_{k+1}). 
\]
It can be easily verified that $\cdots w_{-1}'w_0'\cdot (01\bar{1})^{\omega} \in \mathcal{S}$ since the forbidden sequences given above are avoided in the output of the transducer.
Since $x_k$ is bounded and $M_{\beta}$ is contracting on $H$, we obtain therefore that $\varphi(\cdots w_{-1} w_0) + \Phi(\decdot (10\bar{1})^{\omega}) = \varphi(\cdots w_{-1}' w_0') + \Phi(\decdot (01\bar{1})^{\omega}) \in \mathcal{T}_{\decdot(01\bar{1})^{\omega}}$.

For proving $M_{\beta}^5 \mathcal{T}_{\decdot\bar{1}0(10\bar{1})^{\omega}} \subseteq \mathcal{T}_{\decdot(01\bar{1})^{\omega}}$, it remains to show that every sequence $\bar{1}010\bar{1} w_{-5} w_{-6} \cdots$ satisfying $\cdots w_{-6} w_{-5} \bar{1}010\bar{1} \cdot (10\bar{1})^{\omega} \in \mathcal{S}$ is the input of a path in the transducer.
The paths with input $\bar{1}010\bar{1}$ starting in $(\decdot 01; 001)$  lead to the set of states
\[ 
\bar{Q}_1 = \{(\decdot \bar{1}0\bar{1}; 1),\, (\decdot 00\bar{1}; 0\bar{1}),\, (\decdot 0\bar{1}; 0\bar{1}),\, (\decdot \bar{1}\bar{1}; 1)\}. 
\]
We show that every path $\bar{1}010\bar{1} w_{-5} \cdots w_{-k}$, $k \ge 5$, with $w_{-k} = 1$ leads to one of the sets
\begin{align*}
Q_1 & = \{(\decdot 101; \bar{1}),\, (\decdot 001; 01),\, (\decdot 01; 001),\, (\decdot 11; \bar{1})\} \\
Q_2 & = \{(\decdot 101; \bar{1}),\, (\decdot 001; 01),\, (\decdot 00\bar{1}; 1),\, (\decdot 011; 00\bar{1})\}, \\
Q_3 & = \{(\decdot 01; 001),\, (\decdot 11; \bar{1}),\, (\decdot 101; 0\bar{1}),\, (\decdot 001; 1)\}, \\
Q_4 & = \{(\decdot 101; \bar{1}),\, (\decdot 001; 01),\, (\decdot 011; 00\bar{1}),\, (\decdot 11; \bar{1})\}, \\
Q_5 & = \{(\decdot 00\bar{1}; 1),\, (\decdot 101; \bar{1}),\, (\decdot011; 0\bar{1})\},
\end{align*}
and that every path $\bar{1}010\bar{1} w_{-5} \cdots w_{-k}$, $k \ge 5$, with $w_{-k} = \bar{1}$, leads to one of the sets $\bar{Q}_1$, $\bar{Q}_2$, $\bar{Q}_3$, $\bar{Q}_4$, $\bar{Q}_5$, where $\bar{Q}_i$ is defined by exchanging $1$ and $\bar{1}$ in $Q_i$.
We have the transitions
\begin{gather*}
\bar{Q}_1 \stackrel{1}{\longrightarrow} Q_2,\, \bar{Q}_1 \stackrel{01}{\longrightarrow} Q_1,\, \bar{Q}_1 \stackrel{001}{\longrightarrow} Q_3,\, \bar{Q}_2 \stackrel{1}{\longrightarrow} Q_4,\, \bar{Q}_2 \stackrel{01}{\longrightarrow} Q_3,\, \bar{Q}_2 \stackrel{001}{\longrightarrow} Q_3,\, \bar{Q}_3 \stackrel{1}{\longrightarrow} Q_5,\, \bar{Q}_3 \stackrel{01}{\longrightarrow} Q_1, \\
\bar{Q}_3 \stackrel{001}{\longrightarrow} Q_3,\, \bar{Q}_4 \stackrel{1}{\longrightarrow} Q_4,\, \bar{Q}_4 \stackrel{01}{\longrightarrow} Q_3,\, \bar{Q}_4 \stackrel{001}{\longrightarrow} Q_3,\, \bar{Q}_5 \stackrel{1}{\longrightarrow} Q_1,\, \bar{Q}_5 \stackrel{01}{\longrightarrow} Q_3,\, \bar{Q}_5 \stackrel{001}{\longrightarrow} Q_3.
\end{gather*} 
Together with the symmetric transitions $Q_i \stackrel{0^n\bar{1}}{\longrightarrow} \bar{Q}_j$, this shows inductively the assertion.
Since $\lambda^{d-1}(\mathcal{D}_{\decdot\bar{1}0(10\bar{1})^{\omega}}) > 0$, we have shown that $\mathcal{T}$ is not a tiling, but a double tiling.

\begin{remark}
In the transducer in Figure~\ref{fig:double}, every output gives a $T$-admissible sequence.
It is possible to construct a similar (but larger) transducer where both the input and the output give $T$-admissible sequences. 
This means that almost every point in $\mathcal{T}_{\decdot(10\bar{1})^{\omega}} \cap \mathcal{T}_{\decdot(01\bar{1})^{\omega}}$ is represented by a unique path in the new transducer.
Then this intersection has positive measure if and only if the largest eigenvalue of the new transducer is equal to~$\beta$, cf. Corollary~5.3 in~\cite{Siegel04}.  
Since the set of differences $y - x$ with $\mathcal{T}_x \cap \mathcal{T}_y \ne \emptyset$ is finite by Proposition~\ref{p:quasiperiodic}, this provides an effective method for deciding the tiling property whenever $\bar{\mathcal{S}}$ is a sofic shift, cf.\ Theorem~4.1 in~\cite{SiegelThuswaldner} and Theorem~5.4 in~\cite{Siegel04}. 
\end{remark}

\bibliographystyle{alpha}
\bibliography{mtiling}
\end{document}